\documentclass[10.5pt,a4wide, reqno]{amsart}
\usepackage{a4wide}
\usepackage{pst-all}
\usepackage{pstcol}
\usepackage{graphics,graphicx}
\usepackage[english]{babel}

\newtheorem{thm}{Theorem}[section]
\newtheorem{cor}[thm]{Corollary}
\newtheorem{lem}[thm]{Lemma}
\newtheorem{prop}[thm]{Proposition}

\newtheorem{ex}[thm]{Example}
\newtheorem{rem}[thm]{Remark}
\numberwithin{equation}{section}

\newcommand{\R}{\mathbb{R}}

\begin{document}

\title[Control and Stabilization of the Benjamin Equation]
 {On the controllability and Stabilization of the Benjamin Equation}

\author{M. Panthee  and  F. Vielma Leal  }

\address{Department of Mathematics, Statistics \& Computing
 Science, Campinas University, Sao Paulo, Brasil}

\email{mpanthee@ime.unicamp.br,\; fvielmaleal7@gmail.com}

\thanks{This work was partially supported
 by FAPESP, Brazil, with grants 2016/25864-6 and 2015/06131-5 and by CNPq, Brazil with grant 308131/2017-7.}

\thanks{}




\dedicatory{}

\commby{Mahendra Panthee}


\begin{abstract}
The aim of this paper is to study the controllability and stabilization for the Benjamin equation on a periodic domain $\mathbb{T}$.
We show that the Benjamin equation is globally exactly controllable and globally exponentially stabilizable in $H_{p}^{s}(\mathbb{T}),$ with $s\geq 0.$
First we prove propagation of compactness, propagation of regularity of solution in Bourgain's spaces and unique continuation property, and use them to obtain the global exponential stabilizability corresponding to a natural feedback law. Combining the global exponential stability  and the local controllability result we  prove the global controllability as well.
Also, we prove that the closed-loop system with a different feedback control law 
is locally exponentially stable with an arbitrary decay rate.
Finally, a time-varying feedback law is  designed to guarantee
a global exponential stability with an arbitrary  decay rate.
The results obtained here extend the ones we proved for the linearized Benjamin equation in \cite{Vielma and Panthee}.
\end{abstract}

\maketitle

\section{Introduction}

We consider the  Benjamin equation posed on a periodic domain $\mathbb{T}:=\mathbb{R}/ (2\pi \mathbb{Z}),$
\begin{equation}\label{2D-BO}
  \partial_{t}u-\alpha \mathcal{H}\partial^{2}_{x}u-\partial^{3}_{x}u
  +\partial_{x}(u^{2})=0, \quad u(x, 0) = u_0\;\;\;\;x\in\mathbb{T},\;\;t\in\mathbb{R},
\end{equation}
where $u=u(x,t)\in \R$,
$\alpha>0$  is a constant and  $\mathcal{H}$ denotes the Hilbert transform defined by
$\widehat{\mathcal{H}(f)}(k)= -i\;\text{sgn}(k)\widehat{f}(k),$ $\,\forall\, k\in\mathbb{Z},$
with $\widehat{f}(k)$ the Fourier transform of $f$ given by
$$\widehat{f}(k):=\frac{1}{2\pi}\int_{0}^{2\pi}f(x)e^{-ikx}dx.$$

The equation \eqref{2D-BO} posed on spatial domain $\mathbb{R}$ was derived by Benjamin
\cite{5} to study the gravity-capillarity surface waves of solitary type in deep water and serves as a generic model for unidirectional propagations of long waves in a two-fluid system where the lower fluid with
greater density is infinitely deep and the interface is subject to capillarity.
The author in \cite{5} also showed that solutions of the Benjamin equation satisfy the conserved quantities,
\begin{equation}\label{cons-1}
I_{1}(u):=\frac{1}{2}\int_{\mathbb{R}} u^{2}(x,t)\;dx, = I_1(u_0)
\end{equation}
and
\begin{equation}\label{cons-2}
\displaystyle{I_{2}(u):=\int_{\mathbb{R}} \left[\frac{1}{2}(\partial_{x}u)^{2}(x,t)-
\frac{\alpha}{2}u(x,t)\mathcal{H}\partial_{x}u(x,t)-\frac{1}{3}u^{3}(x,t)\right]\;dx} = I_2(u_0).
\end{equation}
We note that the relations \eqref{cons-1} and \eqref{cons-2} hold in the periodic case as well.

The well-posedness
of the Cauchy problem \eqref{2D-BO} for given data in $H^{s}(\mathbb{R})$ and $H_p^{s}(\mathbb{T})$ has been extensively studied for many years, see ~\cite{19,20,22,23,Linares Scialom}.
The best known global well-posedness result in $L^{2}(\mathbb{R})$ is due to Linares
\cite{23}. The local well-posedness below $L^{2}(\mathbb{R}),$ is studied
by Kozono, Ogawa and Tanisaka \cite{19} and
 Chen,  Guo, and Xiao \cite{20} for $s\geq -\frac{3}{4}.$
In the periodic case, the best global well-posedness in $L^{2}(\mathbb{T})$ is due to 
 Linares \cite{23} and local well-posedness in $H_{p}^{s}(\mathbb{T})$ for $s\geq -\frac{1}{2}$ is due to  Shi and Junfeng \cite{22}.The Benjamin equation also admits solitary waves solutions. Several works have been devoted to study the existence, stability and asymptotic properties of such solutions, see for instance ~\cite{15,16,5,18}.

Our aim here is to study the equation \eqref{2D-BO} in the context of control theory with a forcing term
$f=f(x,t)$ 
\begin{equation}\label{2D-BO1}
  \partial_{t}u-\alpha \mathcal{H}\partial^{2}_{x}u-\partial^{3}_{x}u
  +\partial_{x}(u^{2})=f(x,t), \;\;\;\;x\in \mathbb{T},\;\;t\in \mathbb{R}.
\end{equation}
In order to keep  the mass $I_{1}(u)$ conserved in the control system \eqref{2D-BO1}, we demand
\begin{equation}\label{mediaf}
\int_{0}^{2\pi}f(x,t)\;dx= 0.
\end{equation}

The control $f$ is allowed to act on only a small subset $\omega$ of the domain $\mathbb{T}.$ This situation includes more cases of practical interest and is therefore more relevant in general. For this reason, we consider $g(x)$ as a real non-negative smooth function defined on $\mathbb{T},$
such that,
\begin{equation}\label{gcondition}
2\pi[g]:=\int_{0}^{2\pi}g(x)\;dx=1,
\end{equation}
where $[g]$ denotes the mean value $g$ over the interval $(0,2\pi)$. We assume
$\text{supp} \;g=\omega \subset \mathbb{T},$ where
$\omega=\{x\in \mathbb{T}: g(x)>0 \}$ is an open interval.
We restrict our attention to controls of the form
\begin{equation}\label{EQ1}
f=G(h):=g(x)\left[h(x,t)-\int_{0}^{2\pi}g(y)
h(y,t)\;dy\right],\;\forall x\in \mathbb{T},\;t\in [0,T],
\end{equation}
where $h$ is a function defined in $\mathbb{T}\times [0,T].$
 Thus, $h\equiv h(x,t)$ can be considered as a new control function. Moreover,
for each $t\in [0,T] $ we have
that \eqref{mediaf} is satisfied.
Observe that
if $s\in \mathbb{R},$ then the operator $G:L^{2}\left([0,T]; H^{s}_{p}(\mathbb{T})\right)
\rightarrow L^{2}\left([0,T]; H^{s}_{p}(\mathbb{T})\right)$
is linear and bounded. Furthermore,
The operator $G:L^{2}(\mathbb{T})\rightarrow L^{2}(\mathbb{T})$
is linear, bounded and self-adjoint (see \cite{Vielma and Panthee}).

In this work, we study the following two important problems in control theory.

\noindent
\textbf{\emph{Exact control problem:}}
Given an initial state $u_{0}$ and a terminal state $u_{1}$
in a certain space with $[u_{0}]=[u_{1}],$ can one find a control input
$f$ such that the solution $u$ of equation \eqref{2D-BO1} satisfies  $u(x,0)=u_{0}(x)$ and $u(x,T)=u_{1}(x) \;\;\forall x\in \mathbb{T},$ for some final time $T>0$?

\noindent
\textbf{\emph{Stabilization Problem:}}
Given an initial state $u_{0}$
in a certain space. Can one find a feedback control law: $f=Ku$ so that the resulting closed
loop system
\begin{equation}\label{2D-BO2}
  \partial_{t}u-\alpha \mathcal{H}\partial^{2}_{x}u-\partial^{3}_{x}u
  +\partial_{x}u^{2}=Ku, \;\;u(x,0)=u_{0}\;\;x\in \mathbb{T},\;\;t\in \mathbb{R}^{+},
\end{equation}
is asymptotically stable  as $t\rightarrow \infty$?

Control and stabilization of dispersive equations has been widely studied in the literature, see \cite{Coron, 1, Linares Rosier, Lions, Rosier 1, Rosier Zhang,  Rosier and Zhang 2, Rosier Zhang 2, Russell} and references therein. In particular, for the KdV equation, we refer to \cite{14,10,
Zhang 1, Russell and Zhang, Rosier 1, Coron Crepau, Menzala Vasconcellos Zuazua, Rosier and Zhang 2} and for the BO equation  we refer to \cite{1,Laurent Linares and Rosier, Linares Rosier} and the references therein.
 The Benjamin equation \eqref{2D-BO} displays both the third order local term $-\partial_{x}^{3}u$, as in the KdV equation, and the second order nonlocal term $-\alpha \mathcal{H} \partial_{x}^{2}u$, as in the BO equation.
So, it is natural to analyze the Benjamin equation
from the control and stabilization point of view and check
whether it behaves in the similar way as the KdV and BO
equations.  In this regard, our work is inspired by the works  of Laurent, Linares and Rosier  \cite{Laurent Linares and Rosier}, Linares and  Ortega \cite{1}, Rosier and Zhang \cite{14} and Russell and Zhang  \cite{10},  where the authors studied the  controllability and stabilization of the individual BO and the  KdV equations posed on periodic domains. 

We recall that  the authors in \cite{Vielma and Panthee}  considered the controllability and stabilization issues for the linearized Benjamin equation on a periodic domain
$\mathbb{T},$  and proved the following results.

\begin{thm}[\cite{Vielma and Panthee}]\label{ControlLa}
	Let $s\geq 0,$
	 $\alpha>0,$ and  $T>0$ be given.
	Then for each $u_{0},\; u_{1}\in H_{p}^{s}(\mathbb{T})$ with $[u_{0}]=[u_{1}],$ there exists a function $h\in L^{2}([0,T];H_{p}^{s}(\mathbb{T}))$
	such that the unique solution $u\in C([0,T];H^{s}_{p}(\mathbb{T}))$ of the non homogeneous linear  IVP associated to equation \eqref{2D-BO1} (with $f(x,t)=Gh(x,t)$)
	satisfies $u(x,T)=u_{1}(x), \;x\in\mathbb{T}.$ Moreover, there exists a positive constant $\nu \equiv \nu(s,g,T)> 0$
	such that
$$	\|h\|_{L^{2}([0,T];H_{p}^{s}(0,2\pi))} \leq \nu\; (\|u_{0}\|_{H_{p}^{s}(0,2\pi)}
	+\|u_{1}\|_{H_{p}^{s}(0,2\pi)}).$$
\end{thm}

Also, employing the feedback control law
$K=-GG^{\ast},$ the following result regarding  stabilization of the linearized Benjamin equation posed on $\mathbb{T}$ is proved in \cite{Vielma and Panthee}.
\begin{thm}[\cite{Vielma and Panthee}]\label{stabilization1}
	Let $\alpha>0,$  $g$ as in \eqref{gcondition},  and  $s\geq 0$ be given. Then there exist positive constants $M=M(\alpha, g, s)$ and $\gamma=\gamma(g),$ such that for any
	$u_{0}\in H_{p}^{s}(\mathbb{T}),$  the unique solution $u\in C([0,+\infty), H_{p}^{s}(\mathbb{T}))$
	of the closed-loop system
	\begin{equation}\label{atabilizationL23}
	\partial_{t}u-\alpha \mathcal{H}\partial^{2}_{x}u-\partial^{3}_{x}u=-GG^{\ast}u,
\;\;u(x,0)=u_{0}(x),\;\;x\in \mathbb{T},\;\;t>0.
	\end{equation}
 satisfies
$$\|u(\cdot,t)-[u_{0}]\|_{H_{p}^{s}(\mathbb{T})}\leq M
	e^{-\gamma t}\|u_{0}-[u_{0}]\|_{H_{p}^{s}(\mathbb{T})},\;\;\;\text{for all}\;\;
	t\geq 0.$$
\end{thm}

Furthermore, the authors in \cite{Vielma and Panthee} also showed that it is possible to  find a linear feedback law such that  the resulting closed-loop system \eqref{atabilizationL23} is stable with arbitrary decay rate.
\begin{thm}[\cite{Vielma and Panthee}]\label{estabilization2}
	Let $s\geq 0,$ $\alpha>0,$ $\lambda>0,$ and $u_{0}\in H_{p}^{s}(\mathbb{T})$ be given. Then there exists a linear bounded operator $K_{\lambda}$  from $H_{p}^{s}(\mathbb{T})$ to $H_{p}^{s}(\mathbb{T})$ such that
	the unique solution $u\in C([0,+\infty), H_{p}^{s}(\mathbb{T}))$ of the closed-loop system
	\begin{equation}\label{estabiliza1}
	\partial_{t}u-\alpha \mathcal{H}\partial^{2}_{x}u-\partial^{3}_{x}u=-K_{\lambda}u,
\;\;u(x,0)=u_{0}(x), \;\; x\in \mathbb{T},\;\; t>0.
	\end{equation}
	satisfies
	$$\|u(\cdot,t)-[u_{0}]\|_{H_{p}^{s}(\mathbb{T})}\leq
	M\;e^{-\lambda\;t}\|u_{0}-[u_{0}]\|_{H_{p}^{s}(\mathbb{T})},$$
	for all $t\geq0,$ and some positive constant $M=M(g,\lambda, \alpha, s).$
\end{thm}

Now, a natural question is, whether one can get similar results for the nonlinear Benjamin equation \eqref{2D-BO}.
Extending the linear results to the corresponding nonlinear systems
 is difficult in general.  Nevertheless,
motivated by the work in Laurent et al. \cite{Laurent,Laurent C1, 14}  (see also
\cite{Dehman Lebeau and Zuazua,Dehman Gerard and Lebeau}),
we use the Bourgain's spaces and the techniques motivated from the microlocal analysis to
get certain  propagation of compactness and regularity properties to the solutions of the Benjamin equation posed on a periodic domain $\mathbb{T}$.  We use these properties together with
the unique continuation property (see Proposition \ref{UCP2} below)
to establish the global stabilization and exact controllability
for the nonlinear system \eqref{2D-BO1}.

In what follows, we describe 
the main results obtained in this work. We start with the following local control result.

\begin{thm}[Local control]\label{Smalldatacontrol}
	Let $T>0,$ $s\geq0,$ $\alpha>0,$
  and  $\mu \in \mathbb{R}$ be given.  Then there exists  $\delta>0$ such that
	for any $u_{0},u_{1}\in H_{p}^{s}(\mathbb{T})$ with $[u_{0}]=[u_{1}]=\mu$  and
$$\|u_{0}-\mu\|_{H_{p}^{s}(\mathbb{T})}\leq \delta,\;\;\;\;\;\|u_{1}-\mu\|_{H_{p}^{s}(\mathbb{T})}\leq \delta,$$
	one can find a control $h\in L^{2}([0,T];H_{p}^{s}(\mathbb{T}))$ such that the IVP associated to 
	\eqref{2D-BO1} with $f=Gh$
	has a unique solution $u\in C([0,T];H_{p}^{s}(\mathbb{T}))$ satisfying
$$ u(x,0)=u_{0}(x), \;\;\;\;u_{1}(x,T)=u_{1}(x),\;\;\;\text{for all}\;x\in \mathbb{T}.$$
\end{thm}

We show that the same feedback control law $K=-GG^{\ast}$ stabilizing the linearized Benjamin equation, stabilizes the nonlinear Benjamin equation as well. More precisely, we prove the following result.

\begin{thm}\label{GEF61}
	 Let $s\geq 0,$
 $\alpha>0,$  and	$\mu\in \mathbb{R}$ be given.
    There exists a constant $k'>0$
	such that for any $u_{0}\in H_{p}^{s}(\mathbb{T})$ with $[u_{0}]=\mu,$
	the corresponding solution $u\in C([0,+\infty), H_{p}^{s}(\mathbb{T}))$ of the
    closed-loop system
	\eqref{2D-BO2} with $Ku=-GG^{\ast}u,$ satisfies
$$\|u(\cdot,t)-[u_{0}]
	\|_{H_{p}^{s}(\mathbb{T})}\leq
	\alpha_{s,k'}(\|u_{0}-[u_{0}]\|_{H_{p}^{s}(\mathbb{T})})\;
	e^{-k't}\|u_{0}-[u_{0}]
	\|_{H_{p}^{s}(\mathbb{T})},\;\;\;\text{for all}\;t\geq 0,$$
	 where
	$\alpha_{s,k'}:\mathbb{R}^{+}\longrightarrow
	\mathbb{R}^{+}$ is a nondecreasing
	continuous function depending on $s$ and $k'.$
\end{thm}

The global controllability result is derived by
a combination of the exponential stabilization result
presented in Theorem \ref{GEF61} and the local control result presented in Theorem \ref{Smalldatacontrol},
as is usual in  control theory (see for instance
\cite{Dehman Lebeau and Zuazua, Dehman Gerard and Lebeau, Laurent, 14, Laurent Linares and Rosier}).
Indeed, given the initial data $u_{0}$  to be controlled, by means of the damping
term $Ku=-GG^{\ast}u$ supported in $\omega,$ i.e by solving the IVP
\eqref{2D-BO2} (with $Ku=-GG^{\ast}u$), we
drive it to a small state in a sufficiently large time.
We do the same with the final state $u_{1}$ by solving the
system backwards in time, due to the time reversibility of the  Benjamin equation. This produces two states which are small enough so that the local controllability result for small data applies.
Using this procedure, we obtain the following large data control result as a direct consequence of
Theorems \ref{Smalldatacontrol} and \ref{GEF61}.
\begin{thm}[Large data control]\label{Largedatacontrol3}
	Let  $s\geq0,$ $\alpha>0,$ 
	$\mu \in \mathbb{R},$  and $R>0$ be given.
	Then there exists a time $T>0,$ such that
	for any $u_{0},u_{1}\in H_{p}^{s}(\mathbb{T})$ with $[u_{0}]=[u_{1}]=\mu$  and
	$$\|u_{0}\|_{H_{p}^{s}(\mathbb{T})}\leq R,\;\;\;\;\;\|u_{1}\|_{H_{p}^{s}(\mathbb{T})}\leq R,$$
	 one can find a control input $h\in L^{2}([0,T];H_{p}^{s}(\mathbb{T}))$ such that the IVP associated to 
	 \eqref{2D-BO1} with $f=Gh$
	admits  a unique solution $u\in C([0,T];H_{p}^{s}(\mathbb{T}))$ satisfying
	$$u(x,0)=u_{0}(x), \;\;\;\;u(x,T)=u_{1}(x),\;\;\;\text{for all}\;x\in \mathbb{T}.$$
	
	Thus, the equation \eqref{2D-BO1} is globally exactly controllable.
\end{thm}

To get the decay rate $k'$ in Theorem \ref{GEF61} arbitrarily large,
a different control law is needed. First, the same feedback control law
stabilizing the linearized Benjamin equation
given in Theorem \ref{estabilization2} allows us to get
the following local stabilization result for
the Benjamin equation.

\begin{thm}\label{LER1}
	Let $0<\lambda'<\lambda$ and $s\geq0$ be given. Assume
	 $\alpha>0.$  There exist $\delta>0$
and a linear bounded operator $K_{\lambda}:H_{p}^{s}(\mathbb{T})\rightarrow H_{p}^{s}(\mathbb{T})$ such that
 for any
	$u_{0}\in H_{p}^{s}(\mathbb{T})$ with
	$\|u_{0}-[u_{0}]\|_{H_{p}^{s}(\mathbb{T})}\leq \delta,$ the corresponding solution
$u\in C([0,+\infty), H_{p}^{s}(\mathbb{T}))$ of the closed-loop system
	\eqref{2D-BO2} with $Ku=K_{\lambda}u,$ satisfies
$$\|u(\cdot,t)-[u_{0}]
	\|_{H_{p}^{s}(\mathbb{T})}\leq
	C\;e^{-\lambda't}\|u_{0}-[u_{0}]
	\|_{H_{p}^{s}(\mathbb{T})},\;\;\;\text{for all}\;\;t\geq 0,$$
	where $C>0$ is a constant independent of $u_{0}.$
\end{thm}

Next, the feedback laws involved in Theorems \ref{GEF61} and \ref{LER1}
can be combined into a time-varying feedback law, as in \cite{Coron and Rosier, 14}, ensuring
a global stabilization result with an arbitrary  large decay
rate for the Benjamin equation.

\begin{thm}\label{LER2}
Let	$s\geq 0,$ $\lambda>0,$  and $\alpha>0,$ 
be given.  For any
$u_{0}\in H_{p}^{s}(\mathbb{T})$ with $\mu=[u_{0}],$ there exists
a continuous map $Q_{\lambda}:H_{p}^{s}(\mathbb{T})\times \mathbb{R}\rightarrow H_{p}^{s}(\mathbb{T})$
which is periodic in the second variable, and
such that the unique solution $u\in C([0,+\infty), H_{p}^{s}(\mathbb{T}))$ of the closed-loop
system \eqref{2D-BO2} with $Ku=-GQ_{\lambda}(u,t),$ satisfies
$$\|u(\cdot,t)-[u_{0}]
	\|_{H_{p}^{s}(\mathbb{T})}\leq
	\gamma_{s,\lambda, \mu}(\|u_{0}-[u_{0}]\|_{H_{p}^{s}(\mathbb{T})})\;
	e^{-\lambda'' t}\|u_{0}-[u_{0}]
	\|_{H_{p}^{s}(\mathbb{T})},\;\;\;\text{for all}\;\;t\geq 0,$$
where $\gamma_{s,\lambda,\mu}:\mathbb{R}^{+}\rightarrow \mathbb{R}^{+}$ is a nondecreasing continuous function
 depending on $s,\;\lambda$ and $\mu.$
\end{thm}

Note that, any solution $u$ of the IVP associated to 
\eqref{2D-BO1} with $f=Gh$ possesses a constant mean value, say $\mu:=[u(\cdot,t)]=[u_{0}].$ In order to work on  Bourgain's spaces in the periodic setting we need $[u(\cdot,t)]=0.$ To achieve this, it is convenient to defie	 $\widetilde{u}(x,t)=u(x,t)-\mu,$ so that $[\widetilde{u}(\cdot,t)]=0,$ for all $t\in [0,T].$ Moreover, $\widetilde{u}$ solves
\begin{equation}\label{nonlinearintro3}
\left \{
\begin{array}{l l}
\partial_{t}\widetilde{u}-\alpha H\partial^{2}_{x}\widetilde{u}-\partial^{3}_{x}\widetilde{u}+2\mu\partial_{x}\widetilde{u}+2\widetilde{u}\partial_{x} \widetilde{u}=Gh(x,t),&  t\in(0,T),\;\;x\in \mathbb{T}\\
\widetilde{u}(x,0)=\widetilde{u}_{0}(x):=u_{0}(x)-\mu, & x\in \mathbb{T},
\end{array}
\right.
\end{equation}
where  $\widetilde{u}_{0}\in H_{p}^{s}(\mathbb{T})$ with $s\geq0.$ Observe that $\widetilde{u}$ is a solution of the IVP \eqref{nonlinearintro3}, if and only if,
 $u(x,t)=\widetilde{u}(x,t)+\mu$ solves the IVP associated to 
\eqref{2D-BO1} with $f=Gh$.

From now on, for simplicity, we suppress the notation $\widetilde{u}$ by $u,$  $\mu$ will denote a given real constant, and we work on system  \eqref{nonlinearintro3}. We define
$$H_{0}^{s}(\mathbb{T}):=\left\{u\in H_{p}^{s}(\mathbb{T}):
[u(\cdot,t)]=0,\;\;\text{for all}\;\;t\geq0\right\}.$$

If $s=0,$ then we denote $H_{0}^{0}(\mathbb{T})$ by $L_{0}^{2}(\mathbb{T}).$ 
It is known that
$H_{0}^{s}(\mathbb{T})$ is a closed subspace of $H_{p}^{s}(\mathbb{T})$ for all
$s\geq 0.$ In particular, $L_{0}^{2}(\mathbb{T})$ is a closed subspace of $L^{2}(\mathbb{T}).$
Furthermore,
$(H_{0}^{s}(\mathbb{T}), \|\cdot\|_{H_{p}^{s}(\mathbb{T})})$ is a Hilbert space for all $s\geq 0$ and it is easy to show that if $s\geq r \geq 0$ then
$H_{0}^{s}(\mathbb{T}) \hookrightarrow H_{0}^{r}(\mathbb{T}),$ with dense  embedding (see Proposition 6.1 and Remark 6.2 in \cite{Vielma and Panthee}).
We establish a local control result in $H_{0}^{s}(\mathbb{T})$
for the system \eqref{nonlinearintro3} and exponential stability
results in $H_{0}^{s}(\mathbb{T})$ for the system
\begin{equation}\label{nonlinearintro4}
\left \{
\begin{array}{l l}
\partial_{t}u-\alpha H\partial^{2}_{x}u-\partial^{3}_{x}u+2\mu\partial_{x}u +2u\partial_{x} u=Ku(x,t),&  t\geq 0,\;\;x\in \mathbb{T},\\
u(x,0)=u_{0}(x), & x\in \mathbb{T},
\end{array}
\right.
\end{equation}
that will imply all the results stated in Theorems \ref{Smalldatacontrol}-\ref{LER2}.

This paper is organized as follows: In Section \ref{section 2}, we summarize some results
on the controllability of the linear system associated to \eqref{nonlinearintro3} and the stabilization of the linear system associated to \eqref{nonlinearintro4}. In Section \ref{section 3}, we  analyze
Bourgain's spaces properties and derive the propagation of compactness, the propagation
of regularity  and the unique continuation property for the Benjamin equation.
In Section \ref{section 4}, the local controllability 
is obtained. Section \ref{section 5} is devoted to study the stabilization
of the Benjamin equation by a time-invariant feedback control law.
In Section  \ref{section 6}, we investigate the stabilization 
by constructing a time-varying feedback control law.
Finally, in the Appendix  we include some results
used in this work.

\section{Preliminary Results}\label{section 2}

In this section we recall some results related to the controllability of the linear system associated to \eqref{nonlinearintro3} and the exponential stabilization of the linear system associated to \eqref{nonlinearintro4}
(see sections 4 and 5 in \cite{Vielma and Panthee}).
\subsection{Control of the Linear System}
\label{section3}
We begin by considering the IVP associated to the linear part of \eqref{nonlinearintro3}
\begin{equation}\label{linearintro3}
	\left \{
	\begin{array}{l l}
		\partial_{t}u-\alpha \mathcal{H}\partial^{2}_{x}u-\partial^{3}_{x}u+2\mu\partial_{x}u =Gh(x,t),&  t\in(0,T),\;\;x\in \mathbb{T}\\
		u(x,0)=u_{0}(x), & x\in \mathbb{T}.
	\end{array}
	\right.
\end{equation}

For $u_{0}\in H_{0}^{s}(\mathbb{T})$, $s\geq0$, the IVP \eqref{linearintro3} possesses a unique global solution and is described by the unitary group $H^{s}_{0}(\mathbb{T})$
\begin{equation}\label{semgru2}
		 U_{\mu}(t)u_0:=e^{(\alpha H\partial_{x}^{2}+\partial_{x}^{3}-2\mu\partial_{x})t}u_0
		=\left(e^{i(-k^{3}-2\mu k+\alpha k|k|)t}\widehat{u_0}(k)\right)^{\vee}.
\end{equation}
For details see Remark 4.8  in \cite{Vielma and Panthee}.

As shown in \cite{Vielma and Panthee} (see Remark  4.9 there), the system \eqref{linearintro3} is exactly controllable in any positive time $T$ and holds the following property.
\begin{rem}\label{sol4}
For $s\geq 0$ and any $T>0$ given, there exists  a
	bounded linear  operator
	\begin{equation*}
	\Phi_{\mu}:H_{0}^{s}(\mathbb{T})\times H_{0}^{s}(\mathbb{T})\rightarrow L^{2}([0,T];H_{0}^{s}(\mathbb{T}))
	\end{equation*}
	defined by
	$h=\Phi_{\mu}(u_{0},u_{1}),$ for all $u_{0},\;u_{1}\in H_{0}^{s}(\mathbb{T})$
	such that
	\begin{equation*}
	u_{1}=U_{\mu}(T)u_{0}+\int_{0}^{T}U_{\mu}(T-s) (G(\Phi_{\mu}(u_{0},u_{1})))(\cdot,s)\;ds,
	\end{equation*}
	for $(u_{0},u_{1})\in H_{0}^{s}(\mathbb{T})\times H_{0}^{s}(\mathbb{T})$ and
	\begin{equation*}\|\Phi_{\mu}(u_{0},u_{1})\|_{L^{2}([0,T];H_{0}^{s} (\mathbb{T}))} \leq \nu\;
		 (\|u_{0}\|_{H_{0}^{s}(\mathbb{T})}
		+\|u_{1}\|_{H_{0}^{s}(\mathbb{T})}),
		\end{equation*}
	where $\nu$ depends only on $s,\;T,$ and $g.$ Therefore, for any $T>0$ the following observability inequality holds
\begin{equation*}
	\int_{0}^{T}\|G^{\ast}U_{\mu}(\tau)^{\ast}(\phi)(x)
	\|^{2}_{L^{2}_{0}(\mathbb{T})}\;d\tau
	\geq \delta^{2}\; \|\phi\|^{2}_{L_{0}^{2}(\mathbb{T})},\;\;\;\text{ for any}\;\;
	\phi\in L_{0}^{2}(\mathbb{T}),\;\;\text{ some}\;\;\delta>0.
	\end{equation*}
\end{rem}

\subsection{Stabilization of the Linear Benjamin Equation}

In this sub-section we  recall the stabilization results (see section 5.1 in \cite{Vielma and Panthee}) for the linear system 
\begin{equation}\label{atabilizationL2}
\left \{
\begin{array}{l l}
\partial_{t}u-\alpha \mathcal{H}\partial^{2}_{x}u-\partial^{3}_{x}u+2\mu\partial_{x} u=Ku,&  t>0,\;\;x\in \mathbb{T}\\
u(x,0)=u_{0}(x), & x\in \mathbb{T},
\end{array}
\right.
\end{equation}
associated to  \eqref{nonlinearintro4}.  The authors in \cite{Vielma and Panthee} proved that, for  $u_0\in H_{0}^{s}(\mathbb{T})$, $s\geq0$ the IVP \eqref{atabilizationL2} posseses a unique global solution and is exponentialy asymptotically stable when $t$ goes to infinity with simple  feedback control law, $Ku=-GG^{\ast}u.$

\begin{thm}\label{st35}
	Let $\alpha>0,$ $\mu\in\mathbb{R},$ $g$ as in \eqref{gcondition},  and  $s\geq 0$ be given. There exist positive constans $M=M(\alpha, \mu, g, s)$ and $\gamma=\gamma(g),$ such that for any
	$u_{0}\in H_{0}^{s}(\mathbb{T}),$  the unique solution $u$
	of \eqref{atabilizationL2} with $K=-GG^{\ast}$ satisfies
$$\|u(\cdot,t)\|_{H_{0}^{s}(\mathbb{T})}\leq M
	e^{-\gamma t}\|u_{0}\|_{H_{0}^{s}(\mathbb{T})},\;\;\;\text{for all}\;\;
	t\geq 0.$$
\end{thm}

\subsection{Stabilization of the Linear Benjamin Equation
	in with Arbitrary Decay Rate}

In this subsection, we recall the stabilization result from section 5.2 in \cite{Vielma and Panthee} where  an appropriate linear feedback law is chosen so that the decay rate of the resulting closed-loop system is arbitrarily large.

Let $a>0$ be any fixed number. For given $\lambda >0,$ and $s\geq 0,$  define the operator
\begin{equation}\label{st1}
L_{\lambda}\phi=\int_{0}^{a}e^{-2\lambda \tau}\;U_{\mu}(-\tau)GG^{\ast}
U_{\mu}(-\tau)^{\ast}\phi\;d\tau,
\;\;\;\text{for all}\;\;\phi\in H_{p}^{s}(\mathbb{T}).
\end{equation}
The operator $L_{\lambda}$ satisfies  the following properties.

\begin{lem}\label{st9}
The operator	$L_{\lambda}:H_{p}^{s}(\mathbb{T})\longrightarrow
H_{p}^{s}(\mathbb{T})$ is linear and bounded.   Moreover,
	$L_{\lambda}$ is an isomorphism from $H_{0}^{s}(\mathbb{T})$
	onto $H_{0}^{s}(\mathbb{T})$, for all $s\geq 0.$
\end{lem}

\begin{rem}\label{st36}
	From Lemma  \ref{st9} we infer that there exists a positive constant $C=C(\delta,s,\lambda,a,g)$ such that\;
	$\|L_{\lambda}^{-1}\psi\|_{H_{0}^{s}(\mathbb{T})}
	\leq C
	\|\psi\|_{H_{0}^{s}(\mathbb{T})},\;\;
	\text{for all}\;\;\psi\in H_{0}^{s}(\mathbb{T}).$
\end{rem}

Choosing the feedback control law as
\begin{equation}\label{feedback}
Ku= \left\{
\begin{array}{lcl}
-K_{\lambda}u:=-G G^{\ast} L_{\lambda}^{-1}u, & \mbox{if} &  \lambda>0
\\
&           &          \\
-K_{0}u:=-G G^{\ast}u,
& \mbox{if} &\lambda=0,
\end{array}
\right.
\end{equation}
we can rewrite  the resulting closed-loop system \eqref{atabilizationL2} in the following form
\begin{equation}\label{st21}
\left \{
\begin{array}{l l}
\partial_{t}u-\alpha \mathcal{H}\partial^{2}_{x}u-\partial^{3}_{x}u+2\mu\partial_{x} u=-K_{\lambda}u,&  t>0,\;\;x\in \mathbb{T}\\
u(x,0)=u_{0}(x), & x\in \mathbb{T},
\end{array}
\right.
\end{equation}
where $K_{\lambda}$ is a bounded linear operator on
$H^{s}_{p}(\mathbb{T}),$ with $s\geq 0.$
With these considerations, we have the following result.

\begin{thm}\label{st37}
	Let $\alpha>0,$ $\mu \in \mathbb{R},$ $s\geq 0,$ and $\lambda>0$ be given.
	For any $u_{0}\in H_{0}^{s}(\mathbb{T}),$ the system \eqref{st21} admits a unique solution $u\in C(\mathbb{R}^{+}, H_{0}^{s}(\mathbb{T})).$ Moreover,
	there exists $M=M(g,\lambda,\delta, \alpha, \mu, s)>0$ such that
	\begin{equation*}
	    \|u(\cdot,t)\|_{H_{0}^{s}(\mathbb{T})}\leq
	M\;e^{-\lambda\;t}\|u_{0}\|_{H_{0}^{s}(\mathbb{T})},
	\;\;\;\text{for all}\;\;t\geq0.
	\end{equation*}
\end{thm}

\section{Bourgain's Spaces Associated to Benjamin Equation}\label{section 3}

In this section we introduce the Fourier transform norm spaces, the so called Bourgain's spaces and derive some preliminary estimates to get a control result for the Benjamin equation.

\subsection{Bourgain's spaces and their properties}
In order to simplify the notation, in this subsection,  we denote $U_{\mu}(t)$ by $V(t),$ i.e. $V(t)\varphi:=e^{(\alpha H\partial_{x}^{2}+\partial_{x}^{3}-2\mu\partial_{x})t}\varphi$.
Given $r\in\mathbb{R},$ we define an operator
$D^{r}:\mathcal{D}'(\mathbb{T})\rightarrow \mathbb{C}$
by
\begin{equation}\label{difop}
\widehat{D^{r}v}(k)=\left\{
                      \begin{array}{ll}
                        |k|^{r}\widehat{v}(k), & if \;k\neq 0; \\
                        \widehat{v}(0), & if \; k=0,
                      \end{array}
                  \right.
\end{equation}
and
$\partial_{x}^{r}:\mathcal{D}'(\mathbb{T}) \rightarrow\mathbb{C},$
 by
$\widehat{\partial_{x}^{r}v}(k)=(ik)^{r}\;\widehat{v}(k).$

For given $b,s\in \mathbb{R}$ we define the Bourgain's space $X_{s,b}$ associated to the Benjamin equation on $\mathbb{T}$
as the closure of the space of Schwartz functions $\mathcal{S}(\mathbb{T}\times \mathbb{R})$ under the norm 
\begin{equation}\label{x-sb-norm}
\displaystyle{\|v\|_{X_{s,b}}}:=\displaystyle{
\left(\sum_{k=-\infty}^{\infty}\int_{\mathbb{R}}\langle k\rangle^{2s}\langle \tau -\phi(k) \rangle^{2b}|\widehat{v}(k,\tau)|^{2}
\;d\tau\right)^{\frac{1}{2}}},
\end{equation}
where $\phi(k)=-k^{3}-2\mu k +\alpha k |k|$ is the phase function, $\langle \cdot \rangle:=\big({1+|\cdot|^{2}}\big)^{\frac12},$ and
$\widehat{v}(k,\tau)$ denotes the Fourier transform of $v$ with respect to the both  space and time variables given by
$$\displaystyle{\widehat{v}(k,\tau):=
\frac{1}{2\pi}\int\limits_{\mathbb{R}} \int\limits_{\mathbb{T}}v(x,t)e^{-i(t\tau+kx)}\;dx\;dt.}$$
 Sometimes, we use
$\widehat{v}(k,t)$ (respectively $\widehat{v}(x,\tau)$) to denote the Fourier transform in space variable
$x$ (respectively in time variable $t$). In particular $\displaystyle{\|v\|_{X_{s,0}}}=\displaystyle{\|v\|_{L^{2}(\mathbb{R}_{t};H_{p}^{s}(\mathbb{T}))}}.$  Note that, $X_{s,b}$ is a Hilbert space and for $b>\frac12$ 
\begin{equation}\label{cont-embd}
X_{s,b}\subset C(\mathbb{R};H^{s}_{p}(\mathbb{T})),
\end{equation}
the imbedding being continuous.  As noted in \cite{13, Kenig Ponce and Vega, 14, 23, 11}, while dealing with the bilinear estimates  in the periodic case one needs to consider $b=\frac12$ for which  the imbedding  \eqref{cont-embd} fails. To overcome this situation, we introduce the space $Y_{s,b}$ as completion of the space of Schwartz functions $\mathcal{S}(\mathbb{T}\times \mathbb{R})$ under the norm 
\begin{equation}\label{Y-sb}
\displaystyle{\|v\|_{Y_{s,b}}}:=\displaystyle{
\left(\sum_{k=-\infty}^{\infty}\left(\int_{\mathbb{R}}\langle k\rangle^{s}\langle \tau - \phi(k) \rangle^{b}|\widehat{v}(k,\tau)|
\;d\tau\right)^{2}\right)^{\frac{1}{2}}},
\end{equation}
and define the space  
\begin{equation}\label{Z-sb}
Z_{s,b}:=X_{s,b} \cap Y_{s,b-\frac{1}{2}}
\end{equation}
 endowed with the norm
$\|\cdot\|_{Z_{s,b}}:=\|\cdot\|_{X_{s,b}}+\|\cdot\|_{Y_{s,b-\frac{1}{2}}}.$
For a given interval $I,$ we define  $X_{s,b}(I)$  as the restriction space
of $X_{s,b}$ to the interval $I$ with the norm
\begin{equation}\label{X-sbI}
\|f\|_{X_{s,b}(I)}:=\inf\left\{\|\tilde{f}\|_{X_{s,b}}| \tilde{f}=f\;\text{on}\;\mathbb{T}\times I\right\}.
\end{equation}
If  $I=[0,T]$, for simplicity, we denote $X_{s,b}(I)$ by  $X_{s,b}^{T}$. In and analogous manner we define $Y_{s,b}(I), $  $Z_{s,b}(I)$, $Y_{s,b}^{T}$ and  $Z_{s,b}^{T}$.

In what follows, we record some properties of the spaces $X_{s,b}$ and $Z_{s,b}$.
\begin{prop}\label{prop1}
	 The space $X_{s,b}$  have the
	 following properties:
	\begin{itemize}
	\item [ii)] If $s_{1}\leq s_{2}$ and $b_{1}\leq b_{2},$ then $ X_{s_{2},b_{2}}$
	 is continuously imbedded in  $ X_{s_{1},b_{1}}$. The same holds for $X_{s,b}(I)$ too.
	
	   \item [iii)]For a given finite interval $I,$ if $s_{1}< s_{2}$ and $b_{1}< b_{2},$ then  $X_{s_{2},b_{2}}(I)$
	 is compactly imbedded in  $ X_{s_{1},b_{1}}(I).$
	\end{itemize}
\end{prop}

\begin{prop}[\cite{Colliander Keel Staffilani Takaoka and Tao, 11, 13}]\label{contimbedded}
Let $I$ be an interval and $s\in \mathbb{R}$.
The space $Z_{s,\frac{1}{2}}$ \;(resp. $Z_{s,\frac{1}{2}}(I)$) is continuously imbedded in the space $C(\mathbb{R};H^{s}_{p}(\mathbb{T}))$
\;(resp. $C(I;H^{s}_{p}(\mathbb{T}))$).
\end{prop}

\begin{lem}\label{multi38}
	Let $s,b \in \mathbb{R}.$ The space $X_{s,b}$ is reflexive and its dual is given by $X_{-s,-b}.$ 
\end{lem}

\begin{proof}
	It follows by the fact that $\phi(k)$ is an odd function (see Tao \cite[page 97]{11}).
\end{proof}

\begin{lem}\label{prop2}
	Let $s,r\in \mathbb{R}.$ Then,  for any $v\in X_{s,b}$ (resp. $X_{s,b}(I)$) $D^{r}v\in X_{s-r,b}$ (resp. $X_{s-r,b}(I)$). The same is valid for the operator
	$\partial_{x}^{r}.$ Moreover, there exists a positive constant $C,$ independent on $v,$ such that
	$$\displaystyle{\|D^{r}v\|_{X_{s-r,b}}}
	\leq C\; \displaystyle{\|v\|_{X_{s,b}}}.$$
\end{lem}

The ideas to prove the majority of the  results in the following two subsections
are similar to those derived in the KdV case (see \cite{13, Kenig Ponce and Vega, Colliander Keel Staffilani Takaoka and Tao, Colliander Keel Staffilani Takaoka and Tao 2, Kappeler and Topalov, 12}).

\subsection{Linear and integral estimates}

To derive some  estimates localized in time variable, we introduce a cut-off function
  $\eta\in C^{\infty}_{c}(\mathbb{R})$  such that $\eta\equiv1,$\;if $t\in[-1,1]$ and $\eta\equiv0$,\;if $t \notin (-2,2).$
For $T>0$ given, we define $$\displaystyle{\eta_T \in C^{\infty}_{c}(\mathbb{R})\;\;\text{ by}\;\; \eta_T(t):=\eta\left(\frac{t}{T}\right)}.$$

\begin{prop}\label{intervalstimate}
Let $s,b\in \mathbb{R}$ and $T>0$  be given. Then for all $v_{0}\in H^{s}_{p}(\mathbb{T})$, we have
\begin{equation}\|\eta_{T}(t)V(t)v_{0}\|_{X_{s,b}}\leq C_{\eta,b}\;T^{\frac{1}{2}}\;\|v_{0}\|_{H^{s}_{p}(\mathbb{T})},\qquad\|\eta_{T}(t)V(t)v_{0}\|_{Y_{s,b}}\leq C_{\eta,b}\;T^{\frac{1}{2}}\;\|v_{0}\|_{H^{s}_{p}(\mathbb{T})},
 \end{equation}
  \begin{equation}\|V(t)v_{0}\|_{X_{s,b}^{T}}\leq C_{\eta,b}\;T^{\frac{1}{2}}\;\|v_{0}\|_{H^{s}_{p}(\mathbb{T})},\qquad\|V(t)v_{0}\|_{Y_{s,b}^{T}}\leq C_{\eta,b}\;T^{\frac{1}{2}}\;\|v_{0}\|_{H^{s}_{p}(\mathbb{T})},
\end{equation}
where $C_{\eta,b}$ is a positive constant that depends only on $\eta$ and $b.$
\end{prop}

\begin{cor} \label{linstimateZ}
Let $s,b\in \mathbb{R}$ and $T>0$  be given. Then for all $v_{0}\in H^{s}_{p}(\mathbb{T})$, one has
\begin{equation}
\|\eta_{T}(t)V(t)v_{0}\|_{Z_{s,b}}\leq C_{\eta,b}\;T^{\frac{1}{2}}\;\|v_{0}\|_{H^{s}_{p}(\mathbb{T})},\qquad \|V(t)v_{0}\|_{Z_{s,b}^{T}}\leq C_{\eta,b}\;T^{\frac{1}{2}}\|v_{0}\|_{H^{s}_{p}(\mathbb{T})},
\end{equation}
where $C_{\eta,b}$ is a positive constant that depends only of $\eta$ and $b.$
\end{cor}
\begin{proof}
It follows from  Proposition \ref{intervalstimate}
and the definition of the  $Z_{s,b}$
 and $Z^{T}_{s,b}$--norms.
\end{proof}
\begin{thm}\label{intestimateZ3}Let $b=\frac{1}{2}$  and $T>0$ be given. Then, one has
\begin{equation}\label{eq-ll3}
\left\|\eta_{T}(t)\int_{0}^{t}V(t-\tau)f(\tau)\;d\tau\right\|_{Z_{s,\frac{1}{2}}}\leq C_{\eta,T}\;
\|f\|_{Z_{s,-\frac{1}{2}}},\qquad\forall\; f\in Z_{s,-\frac{1}{2}},
\end{equation}
and 
\begin{equation}\label{eq-ll4}
\left\|\int_{0}^{t}V(t-\tau)f(\tau)\;d\tau\right\|_{Z^{T}_{s,\frac{1}{2}}}\leq C_{\eta,T}\;
\|f\|_{Z^{T}_{s,-\frac{1}{2}}},\qquad\forall\; f\in Z^T_{s,-\frac{1}{2}},
\end{equation}
where $C_{\eta,T}$ is a positive constant depending on $\eta,$ and the final time $T.$

If $T\leq 1,$ then the positive constant $C$ involved in \eqref{eq-ll3}  and \eqref{eq-ll4} does not depend on $T.$
\end{thm}
\begin{proof}
	The proof is similar to the proof of Lemma 3.16  in \cite{12}.
\end{proof}

\begin{prop}\label{Xestimate1}
Let $\;T>0,$ $-\frac{1}{2}<b'<b<\frac{1}{2}$ and $s \in \mathbb{R}$ be given. Then for all 
 $v\in X_{s,b}^{T}$, there exists $C>0$, such that 
$$\|v\|_{X_{s,b'}^{T}}\leq C\; T^{b-b'}\|v\|_{X_{s,b}^{T}}.$$

\end{prop}
\begin{proof}
	The proof is similar to the proof of Lemma 2.11 in Tao \cite{11}.
\end{proof}

\begin{prop}[{\cite[page 105]{11}}] \label{Yestimate2}
	For all   $s,b \in \mathbb{R},$  $\epsilon>0,$ and
	$v\in X_{s,b},$ we have
	$$\|v\|_{Y_{s,b-\frac{1}{2}-\epsilon}}\leq C_{\epsilon}\; \|v\|_{X_{s,b}}, $$	
	where $C_{\epsilon}$ is a positive constant.
\end{prop}

\begin{cor}\label{Yestimate1}
For all $T>0,$  $s,b \in \mathbb{R},$  $\epsilon>0,$ and
 $v\in X_{s,b}^{T}$, one has
\begin{equation}\label{Yestimate6}
\|v\|_{Y_{s,b-\frac{1}{2}-\epsilon}^{T}}\leq C_{\epsilon}\; \|v\|_{X_{s,b}^{T}},
\end{equation}
where $C_{\epsilon}$ is a positive constant.
\end{cor}

\begin{proof}
Let  $v\in X_{s,b}^{T}$ and consider $\widetilde{v}$ an extension  in $X_{s,b}$ such that $\|\widetilde{v}\|_{X_{s,b}}\leq 2\;
\|v\|_{X_{s,b}^{T}}.$
From Proposition \ref{Yestimate2}, we have
\begin{equation*}
\begin{split}
\|v\|_{Y_{s,b-\frac{1}{2}-\epsilon}^{T}}&\leq
\|\widetilde{v}\|_{Y_{s,b-\frac{1}{2}-\epsilon}}
\leq C_{\epsilon}\;
\left\|\widetilde{v}\right\|_{X_{s,b}}
\leq 2C_{\epsilon}\;
\left\|v\right\|_{X_{s,b}^{T}}.
\end{split}
\end{equation*}
\end{proof}

\subsection{Some Nonlinear estimates}
 We start with the following result, which is fundamental to estimate the nonlinear term $2u\partial_{x}u,$ in $Z_{s,-\frac{1}{2}}$ norm.

\begin{thm}\label{fundaimme}
 There exists $C_{\alpha} >0$  depending only  on $\alpha$, such that for all $v\in X_{0,\frac{1}{3}}$, one has
\begin{equation}\label{L4-est}
\|v\|_{L^{4}(\mathbb{T}\times \mathbb{R})}\leq\;C_{\alpha}\|v\|_{X_{0,\frac{1}{3}}}.
\end{equation}
\end{thm}

\begin{proof}  
 Observe that \;
 $ \displaystyle{\|v\|^{2}_{X_{0,\frac{1}{3}}}}
  \sim\displaystyle{\sum_{k=-\infty}^{\infty}\int_{\mathbb{R}} \left(1 + |\tau -\phi(k)|\right) ^{\frac{2}{3}}|\widehat{v}(k,\tau)|^{2}\;d\tau}.$

 Also note that, we can write $\displaystyle{v(x,t)=\sum_{m=0}^{\infty}v_{2^m}(x,t)},$  where
 $$\displaystyle{\widehat{v_{2^m}}(k,\tau)=\widehat{v}(k,\tau)\cdot\chi_{2^{m}\leq 1 + |\tau -\phi(k)|< 2^{m+1}}}.$$
In this way, we have
 \begin{equation}\label{priorestimt}
 \begin{split}
  \displaystyle{\|v\|^{2}_{X_{0,\frac{1}{3}}}}
  &\sim\displaystyle{\sum_{k=-\infty}^{\infty}\int_{\mathbb{R}}
  \left|\sum_{m=0}^{\infty} \widehat{v}(k,\tau)\cdot2^{\frac{m}{3}} \cdot
  \chi_{2^{m}\leq 1 + |\tau -\phi(k)|< 2^{m+1}} \right|^{2}\;d\tau}.
  \end{split}
 \end{equation}

 Now, using Plancherel's identity in \eqref{priorestimt}, we obtain
\begin{equation}\label{equiv1}
  \displaystyle{\|v\|^{2}_{X_{0,\frac{1}{3}}}}\sim
  \displaystyle{\sum_{m=0}^{\infty} 2^{\frac{2m}{3}}\|v_{2^{m}}\|^{2}_{L^{2}(\mathbb{T}\times \mathbb{R})} }.
\end{equation}

On the other hand,
{\small
\begin{equation}\label{lcuanorm}
  \displaystyle{\|v\|^{2}_{L^{4}(\mathbb{T}\times\mathbb{R})}}=
 \displaystyle{\|v^{2}\|_{L^{2}(\mathbb{T}\times\mathbb{R})}}\leq
 2\displaystyle{\sum_{m\leq m'}\left\|v_{2^m}\,
v_{2^{m'}}\right\|_{L^{2}(\mathbb{T}\times\mathbb{R})}}
 =2\displaystyle{\sum_{m,n\geq 0}\left\|v_{2^m}\,
v_{2^{m+n}}\right\|_{L^{2}(\mathbb{T}\times\mathbb{R})}}.
\end{equation}}

Once again,  using Plancherel's identity, we get
\begin{equation}\label{priorestimt3}
\begin{split}
\displaystyle{\left\|v_{2^m}\,
v_{2^{m+n}}\right\|_{L^{2}(\mathbb{T}\times\mathbb{R})}}
&= \displaystyle{\left\|\sum_{k_{1}\in \mathbb{Z}}\int_{\mathbb{R}}\widehat{v_{2^m}}(k_{1},\tau_{1})\;
\widehat{v_{2^{m+n}}}(k-k_{1},\tau-\tau_{1})\;d\tau_{1}\right\|_{l^{2}_{k}L^{2}_{\tau}}}.
\end{split}
\end{equation}

We estimate the RHS of  \eqref{priorestimt3} separately in the range $|k|\leq 2^{a}([\alpha]+2)$ and
$|k|> 2^{a}([\alpha]+2),$ where the natural number $a$ will be determined later.
\begin{equation}\label{priorestimt4}
{\scriptsize
\begin{split}
\displaystyle{\left\|v_{2^m}\,
v_{2^{m+n}}\right\|_{L^{2}(\mathbb{T}\times\mathbb{R})}}
&\leq \displaystyle{\left( \sum_{|k|\leq 2^{a}([\alpha]+2)}\left\|\sum_{k_{1}\in \mathbb{Z}}\int_{\mathbb{R}}\widehat{v_{2^m}}(k_{1},\tau_{1})\;
\widehat{v_{2^{m+n}}}(k-k_{1},\tau-\tau_{1})\;d\tau_{1}\right\|^{2}_{L^{2}_{\tau}}\right)^{\frac{1}{2}}}\\
&\displaystyle{\qquad+\left(\sum_{|k|> 2^{a}([\alpha]+2)}\left\|\sum_{k_{1}\in \mathbb{Z}}\int_{\mathbb{R}}\widehat{v_{2^m}}(k_{1},\tau_{1})\;
\widehat{v_{2^{m+n}}}(k-k_{1},\tau-\tau_{1})\;d\tau_{1}\right\|^{2}_{L^{2}_{\tau}}
\right)^{\frac{1}{2}}}\\
&=:I+II.
\end{split}}
\end{equation}

To estimate $I$, we  use the triangular and Young's inequalities and obtain
\begin{equation}\label{priorestimt5}
{\scriptsize
\begin{split}
\displaystyle{\left\|\sum_{k_{1}\in \mathbb{Z}}\int_{\mathbb{R}}\widehat{v_{2^m}}(k_{1},\tau_{1})\;
\widehat{v_{2^{m+n}}}(k-k_{1},\tau-\tau_{1})\;d\tau_{1}\right\|_{L^{2}_{\tau}}}
&\leq \displaystyle{\sum_{k_{1}\in \mathbb{Z}}\left\|\widehat{v_{2^m}}(k_{1},\cdot)\right\|_{L^{1}(\mathbb{R})}
\left\|\widehat{v_{2^{m+n}}}(k-k_{1},\cdot)\right\|_{L^{2}(\mathbb{R})}}.
\end{split}}
\end{equation}

Now applying Cauchy-Schwartz inequality, we get
\begin{equation}\label{priorestimt6}
{\scriptsize
\begin{split}
\displaystyle{\left\|\widehat{v_{2^m}}(k_{1},\cdot)\right\|_{L^{1}(\mathbb{R})}}
&\leq  \displaystyle{ \left(\mu(\{\tau_{1}:2^{m}\leq 1+|\tau_{1}-\phi(k_{1})|<2^{m+1}\})\right)^{\frac{1}{2}}
\left(\int_{\mathbb{R}}\left|\widehat{v}(k_{1},\tau_{1})\cdot \chi_{2^{m}\leq 1+|\tau_{1}-\phi(k_{1})|<2^{m+1}}\right|^{2}\;d\tau_{1}\right)^{\frac{1}{2}}}\\
&\leq \displaystyle{C_{2}\;2^{\frac{m}{2}}
\|\widehat{v_{2^{m}}}(k_{1},\cdot)\|_{L^{2}(\mathbb{R})}},
\end{split}}
\end{equation}
where $\mu$ is the Lebesgue measure. Therefore, applying Cauchy-Schwartz inequality,
 Plancherel, and the invariance of the norm under translations,  we obtain
\begin{equation}\label{priorestimt7}
{\small
\begin{split}
\displaystyle{\left\|\sum_{k_{1}\in \mathbb{Z}}\int_{\mathbb{R}}\widehat{v_{2^m}}(k_{1},\tau_{1})\;
\widehat{v_{2^{m+n}}}(k-k_{1},\tau-\tau_{1})\;d\tau_{1}\right\|_{L^{2}_{\tau}}}
&\leq \displaystyle{C_{2}\;2^{\frac{m}{2}} \|v_{2^{m}}\|_{L^{2}(\mathbb{T}\times\mathbb{R})} \cdot
\|v_{2^{m+n}}\|_{L^{2}(\mathbb{T}\times\mathbb{R})}}.
\end{split}}
\end{equation}

Therefore, from definition of $I$ in \eqref{priorestimt4} and inequalities  \eqref{priorestimt5}, \eqref{priorestimt6}, and \eqref{priorestimt7}, we get
\begin{equation}\label{Iest}
I\leq\displaystyle{C_{3}(\alpha)\;2^{\frac{a+m}{2}} \|v_{2^{m}}\|_{L^{2}(\mathbb{T}\times\mathbb{R})} \cdot
\|v_{2^{m+n}}\|_{L^{2}(\mathbb{T}\times\mathbb{R})}}.
\end{equation}

To estimate $II,$  define
$$\theta :=\displaystyle{\left\|\sum_{k_{1}\in \mathbb{Z}}\int_{\mathbb{R}}\widehat{v_{2^m}}(k_{1},\tau_{1})\;
\widehat{v_{2^{m+n}}}(k-k_{1},\tau-\tau_{1})\;d\tau_{1}\right\|_{L^{2}_{\tau}}}.$$

First, denote $\chi_{m}(k,\tau):=\chi_{2^{m}\leq 1+|\tau-\phi(k)|<2^{m+1}}(\tau). $ Applying-Cauchy Schwartz inequality in $k_{1}$, and $\tau_{1},$
we get
$${
\begin{split}
\theta
&\leq\displaystyle{ \sup\limits_{|k|> 2^{a}([\alpha]+2)}\sup_{\tau \in \mathbb{R}}\left(\;\chi_{m}\ast\chi_{m+n}(k,\tau)\;\right)^{\frac{1}{2}}
\left\|\left(\sum_{k_{1}\in \mathbb{Z}}\int_{\mathbb{R}}|\widehat{v_{^{2^m}}}(k_{1},\tau_{1})|^{2}\;
|\widehat{v_{2^{m+n}}}(k-k_{1},\tau-\tau_{1})|^{2}\;d\tau_{1}\right)^{\frac{1}{2}}\right\|_{L^{2}_{\tau}}}.
\end{split}}$$

Therefore,  using the translation invariance of the norm, and Plancherel inequality,  we obtain
\begin{equation}\label{priorestimt11}
\begin{split}
II&\leq\displaystyle{\left\|\chi_{m}\ast\chi_{m+n}(k,\tau)\right\|^{\frac{1}{2}}_{l^{\infty}_{|k|> 2^{a}([\alpha]+2)}L^{\infty}_{\tau}}\cdot
\left\|v_{2^{m+n}}\right\|_{L^{2}(\mathbb{T}\times\mathbb{R})}\cdot
\left\|v_{2^{m}}\right\|_{L^{2}(\mathbb{T}\times\mathbb{R})}}.
\end{split}
\end{equation}

To estimate the convolution term in inequality \eqref{priorestimt11},
we write for fixed $k$ with\\ $|k|> 2^{a}(2+[\alpha])$ and $\tau$
\begin{equation}\label{conv1}
\chi_{m}\ast\chi_{m+n}(k,\tau)
=\sum_{k_{1}\in \mathbb{Z}}\int_{\mathbb{R}}\chi_{m}(k_{1},\tau_{1})\cdot\chi_{m+n}(k-k_{1},\tau-\tau_{1})\;d\tau_{1}.
\end{equation}

From the support condition on $\chi_{m}$ and  $\chi_{m+n}$ we note that for each $k_{1}$ fixed
there exist $C_{4}\geq0$ and $C_{5}>0$ such that $$C_{4}2^{m}\leq |\tau_{1}-\phi(k_{1})|<C_{2}2^{m}.$$
Thus, $\tau_{1}=\phi(k_{1})+O(2^{m}).$
In a similar way, we have that $\tau-\tau_{1}=\phi(k-k_{1})+O(2^{m+n}).$ In consequence,
\begin{equation}\label{orden}
\tau=\phi(k_{1})+\phi(k-k_{1})+O(2^{m+n}),
\end{equation}
and
\begin{align*}
\int_{\mathbb{R}}\chi_{m}(k_{1},\tau_{1})\cdot\chi_{m+n}(k-k_{1},\tau-\tau_{1})\;d\tau_{1}
&\leq \mu(\{\tau_{1}\in \mathbb{R}:2^{m}\leq 1+|\tau_{1}-\phi(k_{1})|<2^{m+1}\}).
\end{align*}

Therefore, for each fixed $k_{1},$ the $\tau_{1}$ integral in (\ref{conv1}) is $O(2^{m}).$
To calculate the numbers of $k_{1}'s$ for which the integral is non-zero, note that \eqref{orden}
implies
\begin{equation}\label{orden1}
\frac{\tau}{k}=-k^{2}+3kk_{1}-3k_{1}^{2}-2\mu +\frac{\alpha k_{1}|k_{1}|}{k}+\frac{\alpha (k-k_{1})|k-k_{1}|}{k}
+O(C_{6}(\alpha)2^{m+n-a}).
\end{equation}

Thus, we must study four cases:

\noindent
{\bf{Case \textbf{1.}  $k-k_{1}\geq0$ and $k_{1}\geq0:$}}
By identity \eqref{orden1} we have
$$\frac{\alpha k_{1}^{2}}{k}+\frac{\alpha (k-k_{1})^{2}}{k}+3kk_{1}-3k_{1}^{2}-k^{2}-2\mu =\frac{\tau}{k}+O(C_{6}(\alpha)2^{m+n-a}).$$
Note that,\;
$\displaystyle{3kk_{1}-3k_{1}^{2} +\frac{\alpha k_{1}^{2}}{k}+\frac{\alpha (k-k_{1})^{2}}{k}
=\displaystyle{\left(\frac{2\alpha}{k}-3\right)\left(k_{1}^{2}-kk_{1}\right)+\alpha k}}.$
Therefore,\\
$$\left(\frac{2\alpha}{k}-3\right)\left(k_{1}-\frac{k}{2}\right)^{2}=
\frac{\tau}{k}+\left(\frac{2\alpha}{k}-3\right)\frac{k^{2}}{4}-\alpha k+k^{2}-2\mu
+O(C_{6}(\alpha)2^{m+n-a}).$$

Using that $|k|>2^{a}([\alpha]+2),$ we observe that $\displaystyle{\left|\frac{2\alpha}{k}-3\right|>1}.$  
This implies,
$$\displaystyle{\left(k_{1}-\frac{k}{2}\right)^{2}=
\frac{\tau}{2\alpha-3k}+\frac{k^{2}}{4}-\frac{\alpha k^{2}}{2\alpha-3k}+\frac{k^{3}}{2\alpha-3k}-\frac{2\mu k}{2\alpha-3k}
+O(C_{6}(\alpha)2^{m+n-a})}.$$

\noindent
{\bf{Case \textbf{2.}  $k-k_{1}\geq0$ and $k_{1}\leq0:$}}  In this case identity \eqref{orden1} implies,
$$-\frac{\alpha k_{1}^{2}}{k}+\frac{\alpha (k-k_{1})^{2}}{k}+3kk_{1}-3k_{1}^{2}-k^{2}-2\mu =\frac{\tau}{k}+O(C_{6}(\alpha)2^{m+n-a}).$$
With the similar calculations as in {\bf Case \textbf{1}}, we obtain
$$\displaystyle{\left(k_{1}-\left(-\frac{2\alpha}{3k}+1\right)\frac{k}{2}\right)^{2}=-\frac{\tau}{3k}+
\left(-\frac{2\alpha}{3k}+1\right)^{2}\frac{k^{2}}{4}+\frac{\alpha k}{3}-\frac{k^{2}}{3}-\frac{2\mu}{3}+O(C_{6}(\alpha)2^{m+n-a})}.
$$

\noindent
{\bf{Case \textbf{3.} $k-k_{1}\leq0$ and $k_{1}\leq0:$}} In this case identity \eqref{orden1} implies,\\
$$-\frac{\alpha k_{1}^{2}}{k}-\frac{\alpha (k-k_{1})^{2}}{k}+3kk_{1}-3k_{1}^{2}-k^{2}-2\mu =\frac{\tau}{k}+O(C_{6}(\alpha)2^{m+n-a}).$$
Thus,
$$\displaystyle{\left(-\frac{2\alpha}{k}-3\right)\left(k_{1}-\frac{k}{2}\right)^{2}=
\frac{\tau}{k}+\left(-\frac{-2\alpha}{k}-3\right)\frac{k^{2}}{4}+\alpha k+k^{2}+2\mu +O(C_{6}(\alpha)2^{m+n-a})}.
$$

Using  $|k|>2^{a}([\alpha]+2),$ we observe that $\displaystyle{\left|-\frac{2\alpha}{k}-3\right|>1}.$  Therefore,
\begin{align*}
&\displaystyle{\left(k_{1}-\frac{k}{2}\right)^{2}=
\frac{\tau}{-2\alpha-3k}+\frac{k^{2}}{4}+\frac{\alpha k^{2}}{-2\alpha-3k}+\frac{k^{3}}{-2\alpha-3k}+\frac{2\mu k}{-2\alpha-3k} +O(C_{6}(\alpha)2^{m+n-a})}.
\end{align*}

\noindent
{\bf{Case \textbf{4.} $k-k_{1}\leq0$ and $k_{1}\geq0:$}} In this case identity \eqref{orden1} implies,\\
$$\frac{\alpha k_{1}^{2}}{k}-\frac{\alpha (k-k_{1})^{2}}{k}+3kk_{1}-3k_{1}^{2}-k^{2}-2\mu =\frac{\tau}{k}+O(C_{6}(\alpha)2^{m+n-a}).$$
Thus,
$$\displaystyle{\left(k_{1}-\left(\frac{2\alpha}{3k}+1\right)\frac{k}{2}\right)^{2}=
-\frac{\tau}{3k}+\left(\frac{2\alpha}{3k}+1\right)^{2}\frac{k^{2}}{4}-\frac{\alpha k}{3}-\frac{k^{2}}{3}-\frac{2\mu}{3}+
O(C_{6}(\alpha)2^{m+n-a})}.
$$

Therefore, in all cases $k_{1}$ takes at most $O(C_{6}(\alpha)2^{\frac{m+n-a}{2}})$ values. Thus,
\begin{equation}\label{priorestimt12}
\begin{split}
\displaystyle{\left\|\chi_{m}\ast\chi_{m+n}(k,\tau)\right\|_{l^{\infty}_{|k|> 2^{a}([\alpha]+2)}L^{\infty}_{\tau}}}
&\leq C_{7}(\alpha)2^{m}\cdot2^{\frac{m+n-a}{2}}
=C_{7}(\alpha)2^{\frac{3m+n-a}{2}}.
\end{split}
\end{equation}
We can conclude from inequalities \eqref{priorestimt11} and \eqref{priorestimt12} that
\begin{equation}\label{IIstm}
II \leq\displaystyle{C_{7}(\alpha)2^{\frac{3m+n-a}{4}}
\left\|v_{2^{m+n}}\right\|_{L^{2}(\mathbb{T}\times\mathbb{R})}\,
\left\|v_{2^{m}}\right\|_{L^{2}(\mathbb{T}\times\mathbb{R})}}.
\end{equation}

Using estimates \eqref{Iest} and \eqref{IIstm} in \eqref{priorestimt4}, we get
\begin{align*}
\displaystyle{\left\|v_{2^m}\,
v_{2^{m+n}}\right\|_{L^{2}(\mathbb{T}\times\mathbb{R})}}
&\leq C_{7}(\alpha)\left(2^{\frac{m+a}{2}}+2^{\frac{3m+n-a}{4}}\right)
\left\|v_{2^{m+n}}\right\|_{L^{2}(\mathbb{T}\times\mathbb{R})}\,
\left\|v_{2^{m}}\right\|_{L^{2}(\mathbb{T}\times\mathbb{R})}.
\end{align*}
Taking $a=\frac{m+n}{3},$ we obtain
\begin{equation}\label{priorestimt13}
\displaystyle{\left\|v_{2^m}\,
v_{2^{m+n}}\right\|_{L^{2}(\mathbb{T}\times\mathbb{R})}}
\leq C_{8}(\alpha)2^{\frac{4m+n}{6}}
\left\|v_{2^{m+n}}\right\|_{L^{2}(\mathbb{T}\times\mathbb{R})}\,
\left\|v_{2^{m}}\right\|_{L^{2}(\mathbb{T}\times\mathbb{R})}.
\end{equation}

Therefore, inequalities \eqref{lcuanorm} and \eqref{priorestimt13} imply
\begin{equation}\label{priorestimt14}
{\scriptsize
\begin{split}
 \displaystyle{\|v\|^{2}_{L^{4}(\mathbb{T}\times\mathbb{R})}}
&\leq 2C_{8}(\alpha)\sum_{n\geq 0}2^{-\frac{n}{6}}\left(\sum_{m\geq 0}2^{\frac{2(m+n)}{3}}
\left\|v_{2^{m+n}}\right\|^{2}_{L^{2}(\mathbb{T}\times\mathbb{R})}\right)^{\frac{1}{2}}
\,\left(\sum_{m\geq 0}2^{\frac{2m}{3}}
\left\|v_{2^{m}}\right\|^{2}_{L^{2}(\mathbb{T}\times\mathbb{R})}\right)^{\frac{1}{2}}.
\end{split}}
\end{equation}

Thus from inequality \eqref{priorestimt14} and identity \eqref{equiv1}, we obtain
\begin{align*}
 \displaystyle{\|v\|^{2}_{L^{4}(\mathbb{T}\times\mathbb{R})}}
 &\leq C_{9}(\alpha)\|v\|^{2}_{X_{0,\frac{1}{3}}}\left(\sum_{n\geq 0}2^{-\frac{n}{6}}\right)
 \leq C(\alpha)\|v\|^{2}_{X_{0,\frac{1}{3}}}.
\end{align*}
\end{proof}

\begin{cor} \label{dualestim}
Let $f\in L^{\frac{4}{3}}(\mathbb{T}\times \mathbb{R})$. Then, there exists  $C_{\alpha}>0$, such that
$$\|f\|_{X_{0,-\frac{1}{3}}}\leq C_{\alpha}\|f\|_{L^{\frac{4}{3}}(\mathbb{T}\times \mathbb{R})}.$$
\end{cor}

\begin{proof}
It follows from Theorem \ref{fundaimme} that $X_{0,\frac{1}{3}}\hookrightarrow L^{4}(\mathbb{T}\times \mathbb{R}),$ so that
$\left(L^{4}(\mathbb{T}\times \mathbb{R})\right)'\hookrightarrow \left(X_{0,\frac{1}{3}}\right)',$ i.e.,
$L^{\frac{4}{3}}(\mathbb{T}\times \mathbb{R})\hookrightarrow X_{0,-\frac{1}{3}}.$
\end{proof}

\begin{lem}\label{estimate1}
For all $k,k_{1}\in \mathbb{Z}$  with $k\neq 0,\;k_{1}\neq 0,$ and $k\neq k_{1}$, we have 
\begin{equation}\label{ineq-51}
|3kk_{1}(k-k_{1})|\geq\frac{3}{2}k^{2},
\end{equation}
\begin{equation}\label{ineq-52}
| k_{1}(k-k_{1})|\geq\frac{1}{2}|k|.
\end{equation}
\end{lem}
\begin{proof}
The proof of \eqref{ineq-51} follows by simple calculations considering six possible cases:
\begin{align*}
&R^{+++}:=\{k-k_{1}>0, k>0, k_{1}>0\}\\
&R^{++-}:=\{k-k_{1}>0, k>0, k_{1}<0\}\\
&R^{+--}:=\{k-k_{1}>0, k<0, k_{1}<0\}\\
&R^{---}:=\{k-k_{1}<0, k<0, k_{1}<0\}\\
&R^{--+}:=\{k-k_{1}<0, k<0, k_{1}>0\}\\
&R^{-++}:=\{k-k_{1}<0, k>0, k_{1}>0\}.
\end{align*}

In the first case
 $R^{+++},$ observe that
 $\displaystyle{3kk_{1}(k-k_{1})\geq\frac{3}{2}k^{2}\Longleftrightarrow k\geq\frac{k_{1}^{2}}{k_{1}-\frac{1}{2}}}.$
  The fact $k>k_{1}$ implies that the right side of the last expression is always true. The other cases are similar.
Finally, note that \eqref{ineq-52} is consequence of \eqref{ineq-51} just dividing it by $|3k|.$
\end{proof}

\begin{lem}\label{estimate2}
For any $k\in \mathbb{Z},$ $\alpha> 0$ and $\mu \in \mathbb{R},$ let $\phi(k)=-k^{3}-2\mu k+\alpha k|k|.$ 
For all $k,k_{1}\in \mathbb{Z}$  with $k\neq 0,\;k_{1}\neq 0,$  $k\neq k_{1},$ and $\max\{|k|, |k_{1}|, |k-k_{1}|\}\geq
\max\Bigl\{1, \frac{4\alpha}{3}\Bigr\},$
  there exists a constant $C_{\alpha}>0$
depending only on $\alpha$ such that,\;
$\displaystyle{|E(k,k_{1})|\geq 3C_{\alpha}|kk_{1}(k-k_{1})|},$\;
where $$E(k,k_{1}):=(\tau-\phi(k))-(\tau_{1}-\phi(k_{1}))-(\tau-\tau_{1}-\phi(k-k_{1})).$$
\end{lem}

\begin{proof}

Observe that
$E(k,k_{1})=-\alpha k|k|+\alpha k_{1}|k_{1}|+\alpha(k- k_{1})|k-k_{1}|+3kk_{1}(k-k_{1}).$
Again, the proof follows by estraight forward calculations considering the same six cases of Lemma \ref{estimate1}. We verify three cases, others are similar.
\begin{itemize}
  \item [Case  \textbf{1)}] In the region $R^{+++},$ one has  $k\geq k_{1}+1\geq2$ and  $\max\{|k|, |k_{1}|, |k-k_{1}|\}=k.$ Thus
 $ E(k,k_{1})=-\alpha kk+\alpha k_{1}k_{1}+\alpha(k- k_{1})^{2}+3kk_{1}(k-k_{1})
  =3k_{1}(k-k_{1})\left(k-\frac{2\alpha}{3}\right),$
  and $|E(k,k_{1})|=3k_{1}(k-k_{1})\left|k-\frac{2\alpha}{3}\right|\geq3k_{1}(k-k_{1})C_{\alpha}k=3|kk_{1}(k-k_{1})|C_{\alpha}.$

  \item [Case \textbf{2)}] In $R^{++-},$ note that $k-k_{1}\geq2.$ In this case, $\max\{|k|, |k_{1}|, |k-k_{1}|\}=k-k_{1}.$
 $E(k,k_{1})=-\alpha kk+\alpha k_{1}(-k_{1})+\alpha(k- k_{1})^{2}+3kk_{1}(k-k_{1})
  =3kk_{1}\left((k-k_{1})-\frac{2\alpha}{3}\right).$
  Thus, $|E(k,k_{1})|=3k(-k_{1})\left|(k-k_{1}) -\frac{2\alpha}{3}\right|\geq3k(-k_{1})(k-k_{1}) C_{\alpha}=3|kk_{1}(k-k_{1})|C_{\alpha}.$

  \item [Case  \textbf{5)}] In $R^{--+},$  $k\leq -1,$ $k_{1}\geq 1,$ $k-k_{1}\leq -2,$ and $\max\{|k|, |k_{1}|, |k-k_{1}|\}=-(k-k_{1}).$
  $E(k,k_{1})=-\alpha k(-k)+\alpha k_{1}k_{1}-\alpha(k- k_{1})^{2}+3kk_{1}(k-k_{1})
  =3kk_{1}\left((k-k_{1})+\frac{2\alpha}{3}\right).$
  Thus,\; $|E(k,k_{1})|=3(-k)k_{1} \left|-(k-k_{1})-\frac{2\alpha}{3} \right|\geq 3(-k)k_{1}|k-k_{1}|C_{\alpha}
  =3|kk_{1}(k-k_{1})|C_{\alpha}.$
\end{itemize}
\end{proof}

\begin{rem}\label{resonance property}
 Lemmas \ref{estimate1} and \ref{estimate2} imply the so-called non-resonance property for
Benjamin equation, it means, $|E(k,k_{1})|\geq \frac{3}{2}C_{\alpha}k^{2},$ provided that
$\max\{|k|, |k_{1}|, |k-k_{1}|\}\geq
\max\Bigl\{1, \frac{4\alpha}{3}\Bigr\}.$
\end{rem}
\begin{rem}\label{resonance property1}
It  follows from Lemma \ref{estimate2} that if $\max\{|k|, |k_{1}|, |k-k_{1}|\}\geq
\max\Bigl\{1, \frac{4\alpha}{3}\Bigr\},$ then one of the following cases may occur
\begin{itemize}
  \item [i)]$|\tau-\phi(k)|>\frac{3}{8}C_{\alpha}k^{2},$
  \item [ii)]$|\tau_{1}-\phi(k_{1})|>\frac{3}{8}C_{\alpha}k^{2},$
  \item [iii)]$|\tau-\tau_{1}-\phi(k-k_{1})| >\frac{3}{8}C_{\alpha}k^{2}.$
\end{itemize}
\end{rem}

Using similar arguments as in Bourgain  \cite{13} (see also \cite{Chen and Xiao, Terence Tao}), we obtain the following key bilinear estimate.
\begin{thm}\label{BilinearEstimate}(Bilinear Estimate)
Let $u,v:\mathbb{T}\times \mathbb{R}\rightarrow \mathbb{R}$ be  functions in $X_{s,\frac{1}{3}},$ and $X_{s,\frac{1}{2}}$.
Assume that the mean $[u(\cdot,t)]=[v(\cdot,t)]=0$ for each $t\in \mathbb{R},$ $s\geq0,$ $\alpha>0.$  Then
$$\|\partial_{x}(uv)\|_{Z_{s,-\frac{1}{2}}}\leq C_{\alpha,s}\left(\|u\|_{X_{s,\frac{1}{2}}}
\|v\|_{X_{s,\frac{1}{3}}}+\|u\|_{X_{s,\frac{1}{3}}}\|v\|_{X_{s,\frac{1}{2}}}\right).$$
\end{thm}
\begin{proof} We prove this in two steps.

\noindent
	{\textbf{Step 1.}} First we estimate the $X_{s,-\frac{1}{2}}$ norm. Using duality and Plancherel, we get
	\begin{equation}\label{bilestim1}
	\begin{split}
	\left\| \partial_{x}(uv)\right\|_{X_{s,-\frac{1}{2}}}\!\!
	&=\sup_{\substack{w\in X_{-s,\frac{1}{2}} \\ \|w\|_{X_{-s,\frac{1}{2}}=1}}}\left|\sum_{k \in \mathbb{Z}}\int\limits_{\mathbb{R}}
	\widehat{\partial_{x}(uv)}(k,\tau ) \widehat{w}(k,\tau)d\tau\right|\\
	&\leq\!\!\sup_{\substack{w\in X_{-s,\frac{1}{2}} \\ \|w\|_{X_{-s,\frac{1}{2}}=1}}}\!\!\!\!\left(\sum_{\substack{k \in \mathbb{Z}\\ k\neq 0}}\sum_{\substack{k_{1} \in \mathbb{Z}\\k_{1}\neq 0}}\; \int\limits_{\mathbb{R}^2} |k|
	|\widehat{u}(k_{1},\tau_{1})|| \widehat{v}(k-k_{1},\tau-\tau_{1})| |\widehat{w}(k,\tau)|d\tau_{1}d\tau\!\!\right).
	\end{split}
	\end{equation}
	
	 Since $[u(\cdot,t)]=[v(\cdot,t)]=0,$ $k=0,$ $k_{1}=0$ and $k-k_{1}=0$ do not contribute to the sum.
	Now, we move to estimate 
	\begin{equation}\label{aest}
	\begin{split}
	I&:=\sum_{\substack{k \in \mathbb{Z}\\ k\neq 0}}\sum_{\substack{k_{1} \in \mathbb{Z}\\k_{1}\neq 0}}\int\limits_{\mathbb{R}} \int\limits_{\mathbb{R}} |k|
	|\widehat{u}(k_{1},\tau_{1})|| \widehat{v}(k-k_{1},\tau-\tau_{1})| |\widehat{w}(k,\tau)|d\tau_{1}d\tau\\
	&=\sum_{\substack{k, k_{1} \in \mathbb{Z}\\ k k_{1}(k-k_{1})\neq 0}}\int\limits_{\mathbb{R}^{2}} 
	\frac{|k||k_{1}|^{s} \langle \tau_{1}-\phi(k_{1}) \rangle^{\frac{1}{2}} |\widehat{u}(k_{1},\tau_{1})| |k-k_{1}|^{s} \langle\tau-\tau_{1} -\phi(k-k_{1})\rangle^{\frac{1}{2}}}
	{ |k_{1}|^{s}\langle\tau_{1}-\phi(k_{1}) \rangle^{\frac{1}{2}}|k-k_{1}|^{s} \langle\tau-\tau_{1} -\phi(k-k_{1}) \rangle^{\frac{1}{2}}}\\
	&\qquad  \qquad \qquad  \qquad \qquad\times \frac{ |\widehat{v}(k-k_{1},\tau-\tau_{1})|\langle k\rangle^{-s} \langle\tau-\phi(k)\rangle^{\frac{1}{2}} 
	|\widehat{w}(k,\tau)| }
	{ \langle k\rangle^{-s} \langle\tau-\phi(k)\rangle^{\frac{1}{2}}}
	d\tau_{1}d\tau.
	\end{split}
	\end{equation}
	
		Let $u,v,w:\mathbb{T}\times \mathbb{R}\rightarrow  \mathbb{R}$ with $[u(\cdot,t)]=[v(\cdot,t)]=0$. We define
	\begin{equation}\label{bilestim4}
		\begin{split}
		c_{u}(k_{1},\tau_{1})&:=\left(1+|k_{1}|\right)^{s} \langle\tau_{1}-\phi(k_{1})\rangle^{\frac{1}{2}} |\widehat{u}(k_{1},\tau_{1})|,\\
		c_{v}(k-k_{1},\tau-\tau_{1})&:= \left(1+|k-k_{1}|\right)^{s} \langle\tau-\tau_{1}-\phi(k-k_{1}) \rangle^{\frac{1}{2}} |\widehat{v}(k-k_{1}, \tau-\tau_{1})|,\\
		c_{w}(k,\tau)&:=\langle k\rangle^{-s} \langle\tau-\phi(k)\rangle^{\frac{1}{2}} |\widehat{w}(k,\tau)|,
		\end{split}
	\end{equation}
	for all $k,k_{1}, k-k_{1}\in \mathbb{Z}\backslash \{0\}$ and $t, \tau, \tau_{1}\in \mathbb{R}.$
	Note that $c_{u}(0,\tau_{1})=c_{v}(0,\tau-\tau_{1})=0.$ From inequality \eqref{aest} and definition \eqref{bilestim4}, we obtain
\begin{equation*}
		\begin{split}
		I
		&\leq \sum_{\substack{k, k_{1} \in \mathbb{Z}\\ k k_{1}(k-k_{1})\neq 0}}\int\limits_{\mathbb{R}^{2}} 
		\frac{|k|c_{u}(k_{1},\tau_{1})
			c_{v}(k-k_{1},\tau-\tau_{1}) \langle k\rangle^{s}c_{w}(k,\tau)}
		{|k_{1}|^{s}|k-k_{1}|^{s} \langle\tau_{1}-\phi(k_{1})\rangle^{\frac{1}{2}}
			\langle\tau-\tau_{1}-\phi(k-k_{1}) \rangle^{\frac{1}{2}} \langle\tau-\phi(k)\rangle^{\frac{1}{2}}}
		d\tau_{1}d\tau.
		\end{split}
		\end{equation*}
	
	 From \eqref{ineq-52} there exists  $C_{s}>0$ such that
	$\frac{\langle k\rangle^{s}}{|k_{1}|^{s}|k-k_{1}|^{s}} 
	\leq C_{s}.$
	Therefore, separating the small frequencies from the large ones, we obtain
	\begin{equation}\label{bilestim7}
		\begin{split}
      I&\leq C_{s}
		\sum_{\substack{k, k_{1} \in \mathbb{Z}\\ k k_{1}(k-k_{1})\neq 0}}\int\limits_{\mathbb{R}^{2}} 
		\frac{|k|c_{u}(k_{1},\tau_{1})
			c_{v}(k-k_{1},\tau-\tau_{1}) c_{w}(k,\tau)}
		{\langle\tau_{1}-\phi(k_{1})\rangle^{\frac{1}{2}}
			\langle\tau-\tau_{1}-\phi(k-k_{1}) \rangle^{\frac{1}{2}} \langle\tau-\phi(k)\rangle^{\frac{1}{2}}} d\tau_{1}d\tau\\
		&\leq C_{s,\alpha}\sum_{\substack{k, k_{1} \in \mathbb{Z}\\ k k_{1}(k-k_{1})\neq 0\\ \max\{|k|, |k_{1}|,|k-k_{1}|\}\leq \max\Bigl\{1,\frac{4\alpha}{3}\Bigr\}}} 
		\int\limits_{\mathbb{R}^{2}} 
		\frac{c_{u}(k_{1},\tau_{1})
			c_{v}(k-k_{1},\tau-\tau_{1}) c_{w}(k,\tau)}
		{\langle\tau_{1}-\phi(k_{1})\rangle^{\frac{1}{2}}
			\langle\tau-\tau_{1}-\phi(k-k_{1}) \rangle^{\frac{1}{2}} \langle\tau-\phi(k)\rangle^{\frac{1}{2}}} d\tau_{1}d\tau\\	
		&\qquad+\sum_{\substack{k, k_{1} \in \mathbb{Z}\\ k k_{1}(k-k_{1})\neq 0\\ \max\{|k|, |k_{1}|,|k-k_{1}|\}\geq \max\Bigl\{1,\frac{4\alpha}{3}\Bigr\}}}             
		\int\limits_{\mathbb{R}^{2}} 
		\frac{|k|c_{u}(k_{1},\tau_{1})
			c_{v}(k-k_{1},\tau-\tau_{1}) c_{w}(k,\tau)}
		{\langle\tau_{1}-\phi(k_{1})\rangle^{\frac{1}{2}}
			\langle\tau-\tau_{1}-\phi(k-k_{1}) \rangle^{\frac{1}{2}} \langle\tau-\phi(k)\rangle^{\frac{1}{2}}} d\tau_{1}d\tau\\	
		\end{split}
	\end{equation}
	
	In view of Remark \ref{resonance property1} we must study three different  cases.
	
	\noindent {\bf Case 1.} {\bf  $|\tau-\phi(k)|>\frac{3}{8}C_{\alpha}k^{2}:$}\; In this case, from
		\eqref{bilestim7}, we have
				\begin{equation}\label{bilestim8}
		{\small
			\begin{split}
			I
			&\leq C_{s,\alpha}\sum_{\substack{k, k_{1} \in \mathbb{Z}\\ k k_{1}(k-k_{1})\neq 0\\ \max\{|k|, |k_{1}|,|k-k_{1}|\}\leq \max\Bigl\{1,\frac{4\alpha}{3}\Bigr\}}} 
			\int\limits_{\mathbb{R}^{2}} 
			\frac{c_{u}(k_{1},\tau_{1})
				c_{v}(k-k_{1},\tau-\tau_{1}) c_{w}(k,\tau)}
			{\langle\tau_{1}-\phi(k_{1})\rangle^{\frac{1}{2}}
				\langle\tau-\tau_{1}-\phi(k-k_{1}) \rangle^{\frac{1}{2}} } d\tau_{1}d\tau\\	
			&\qquad+\sum_{\substack{k, k_{1} \in \mathbb{Z}\\ k k_{1}(k-k_{1})\neq 0\\ \max\{|k|, |k_{1}|,|k-k_{1}|\}\geq \max\Bigl\{1,\frac{4\alpha}{3}\Bigr\}}}             
			\!\!\!\!\!\int\limits_{\mathbb{R}^{2}} 
			\frac{|k|c_{u}(k_{1},\tau_{1})
				c_{v}(k-k_{1},\tau-\tau_{1}) c_{w}(k,\tau)}
			{\langle\tau_{1}-\phi(k_{1})\rangle^{\frac{1}{2}}
				\langle\tau-\tau_{1}-\phi(k-k_{1}) \rangle^{\frac{1}{2}} (1+\frac{3C_{\alpha}}{8}k^{2})^{\frac{1}{2}}} d\tau_{1}d\tau\\
			&\leq C_{s,\alpha}\sum_{k\in \mathbb{Z}}
			\int\limits_{\mathbb{R}} \left(
			\sum_{k_{1} \in \mathbb{Z}}
			\int\limits_{\mathbb{R}} 
			\frac{c_{u}(k_{1},\tau_{1})
				c_{v}(k-k_{1},\tau-\tau_{1}) }
			{\langle\tau_{1}-\phi(k_{1})\rangle^{\frac{1}{2}}
				\langle\tau-\tau_{1}-\phi(k-k_{1}) \rangle^{\frac{1}{2}} }d\tau_{1} \right)c_{w}(k,\tau) d\tau.	
			\end{split}}
		\end{equation}
		
		We define functions	 $F,\,G:\mathbb{T}\times\mathbb{R}\rightarrow \mathbb{C}$  by
		$$\displaystyle{\widehat{F}(m,\lambda)=\frac{c_{u}(m,\lambda)}{(1+|\lambda-\phi(m)|)^{\frac{1}{2}}}},\;\;\text{and}
		\;\;\displaystyle{\widehat{G}(m,\lambda)=\frac{c_{v}(m,\lambda)}{(1+|\lambda-\phi(m)|)^{\frac{1}{2}}}}.$$
		It means
		$$F(x,t)=\sum_{m\in\mathbb{Z}}\int_{\mathbb{R}}e^{i(mx+\lambda t)}\frac{c_{u}(m,\lambda)}{(1+|\lambda-\phi(m)|)^{\frac{1}{2}}}\;d\lambda,$$
		and
		$$G(x,t)=\sum_{m\in\mathbb{Z}}\int_{\mathbb{R}}e^{i(mx+\lambda t)}\frac{c_{v}(m,\lambda)}{(1+|\lambda-\phi(m)|)^{\frac{1}{2}}}\;d\lambda.$$	

From  \eqref{bilestim1}, \eqref{bilestim8} and  Plancherel, we obtain
$${
\begin{split}
\displaystyle{\|\partial_{x}(uv)\|_{X_{s,-\frac{1}{2}}}}
&\leq C_{s,\alpha} \sup_{\substack{w\in X_{-s,\frac{1}{2}} \\ \|w\|_{X_{-s,\frac{1}{2}}=1}}}  \Big(\sum_{k_{1} \in \mathbb{Z}}\int\limits_{\mathbb{R}}
(\widehat{F}\ast\widehat{G})(k,\tau)\; c_{w}(k,\tau) \;d\tau\Big)\\
&\leq C_{s,\alpha} \sup_{\substack{w\in X_{-s,\frac{1}{2}} \\ \|w\|_{X_{-s,\frac{1}{2}}=1}}} \Big(\left\|\widehat{F.G}(k,\tau)\right\|_{l_{k}^{2}L_{\tau}^{2}(\mathbb{R})}\|w\|_{X_{-s,\frac{1}{2}}}\Big)\\
&\leq
\displaystyle{C_{s,\alpha}\left\|F\right\|_{L^{4}(\mathbb{T}\times\mathbb{R})}\left\|G\right\|_{L^{4}(\mathbb{T}\times\mathbb{R})}}\\
&\leq \displaystyle{C_{s,\alpha}\left\|F\right\|_{X_{0,\frac{1}{3}}}\left\|G\right\|_{X_{0,\frac{1}{3}}}},
\end{split}}$$
where we applied Cauchy-Schwartz, and Theorem \ref{fundaimme} in the last two inequalities. 
 Note that
$${\scriptsize
		\begin{split}
		\displaystyle{\left\|F\right\|_{X_{0,\frac{1}{3}}}}
		&=\displaystyle{\left(\sum_{k=-\infty}^{\infty} \int_{\mathbb{R}}
		\langle\tau-\phi(k)\rangle^{\frac{2}{3}} \frac{|c_{u}(k,\tau)|^{2}}{1+|\tau-\phi(k)|} \;d\tau\right)^{\frac{1}{2}}}
	\sim \displaystyle{\left(\sum_{k=-\infty}^{\infty} \int_{\mathbb{R}} \langle k\rangle^{2s}\langle\tau-\phi(k)\rangle^{\frac{2}{3}} |\widehat{u}(k,\tau)|^{2}\;d\tau\right)^{\frac{1}{2}}}
		=\displaystyle{\left\|u\right\|_{X_{s,\frac{1}{3}}}}.
		\end{split}}$$
		In a similar way, we show that $\displaystyle{\left\|G\right\|_{X_{0,\frac{1}{3}}}=\left\|v\right\|_{X_{s,\frac{1}{3}}}}.$
The immersion $X_{s,\frac{1}{2}}\hookrightarrow
 X_{s,\frac{1}{3}}$ yields,
		$$\displaystyle{\left\|\partial_{x}(u v)\right\|_{X_{s,-\frac{1}{2}}}}\leq
		C_{s,\alpha}\left\|u\right\|_{X_{s,\frac{1}{2}}}\left\|v\right\|_{X_{s,\frac{1}{3}}}$$
		where $C_{s,\alpha}$ is a positive constant depending only on $s$ and $\alpha.$
		
		\noindent{\bf Case 2.  $|\tau_{1}-\phi(k_{1})|>\frac{3}{8}C_{\alpha}k^{2}:$} In this case, 
		\eqref{bilestim7} and Cauchy-Schwartz imply
\begin{equation}\label{bilestim11}
{\small
	\begin{split}
	I
	&\leq C_{s,\alpha}\sum_{\substack{k, k_{1} \in \mathbb{Z}\\ k k_{1}(k-k_{1})\neq 0\\ \max\{|k|, |k_{1}|,|k-k_{1}|\}\leq \max\Bigl\{1,\frac{4\alpha}{3}\Bigr\}}} 
	\int\limits_{\mathbb{R}^{2}} 
	\frac{c_{u}(k_{1},\tau_{1})
		c_{v}(k-k_{1},\tau-\tau_{1}) c_{w}(k,\tau)}
	{\langle\tau-\tau_{1}-\phi(k-k_{1}) \rangle^{\frac{1}{2}} \langle\tau-\phi(k)\rangle^{\frac{1}{2}}} d\tau_{1}d\tau\\	
	&\qquad+\sum_{\substack{k, k_{1} \in \mathbb{Z}\\ k k_{1}(k-k_{1})\neq 0\\ \max\{|k|, |k_{1}|,|k-k_{1}|\}\geq \max\Bigl\{1,\frac{4\alpha}{3}\Bigr\}}}             
	\int\limits_{\mathbb{R}^{2}} 
	\frac{|k|c_{u}(k_{1},\tau_{1})
		c_{v}(k-k_{1},\tau-\tau_{1}) c_{w}(k,\tau_{2})}
	{(1+\frac{3}{8}C_{\alpha}k^{2})^{\frac{1}{2}}
		\langle\tau-\tau_{1}-\phi(k-k_{1}) \rangle^{\frac{1}{2}} \langle\tau-\phi(k)\rangle^{\frac{1}{2}}} d\tau_{1}d\tau\\	
	&\leq C_{s,\alpha}
	\sum_{k\in\mathbb{Z}}\int\limits_{\mathbb{R}}
	\frac{1}{( 1+ |\tau-\phi(k)| )^{\frac{1}{2}}}
	\left(\sum_{k_{1}\in\mathbb{Z}}\int\limits_{\mathbb{R}}c_{u}(k_{1},\tau_{1})\;
	\frac{c_{v}(k-k_{1},\tau-\tau_{1})}{ \langle\tau-\tau_{1}-\phi(k-k_{1}) \rangle^{\frac{1}{2}}}\;d\tau_{1}\right)
	c_{w}(k,\tau)d_{\tau}\\
	&\leq C_{s,\alpha}\left(\sum_{k\in\mathbb{Z}} \int\limits_{\mathbb{R}}( 1+ |\tau-\phi(k)| )^{-1} |(\widehat{H_{u}}\ast\widehat{G}) (k,\tau)|^{2}d\tau\right)^{\frac{1}{2}} \left(\sum_{k\in \mathbb{Z}} \int\limits_{\mathbb{R}}|c_{w}(k,\tau)|^{2} \right)^{\frac{1}{2}},
	\end{split}}
\end{equation}		
		where $H_{f}:\mathbb{T}\times\mathbb{R}\rightarrow \mathbb{C}$ is a function defined by
		$\widehat{H_{f}}(m,\lambda)=c_{f}(m,\lambda).$
		It means,
		$$H_{f}(x,t)=\sum_{m\in\mathbb{Z}}\int_{\mathbb{R}}
		e^{i(mx+\lambda t)}c_{f}(m,\lambda)\;d\lambda.$$
From relations \eqref{bilestim1}-\eqref{aest}, \eqref{bilestim11} and  $\left(1+|\tau-\phi(k)|\right)^{-1}
<\left(1+|\tau-\phi(k)|\right)^{-\frac{2}{3}},$
we have
$${
\begin{split}
\displaystyle{\left\|\partial_{x}(uv)\right\|_{X_{s,-\frac{1}{2}}}}
&\leq C_{s,\alpha}
\sup_{\substack{w\in X_{-s,\frac{1}{2}} \\ \|w\|_{X_{-s,\frac{1}{2}}=1}}} \left[ \left(\sum_{k=-\infty}^{\infty}\int_{\mathbb{R}}
\left(1+|\tau-\phi(k)|\right)^{-\frac{2}{3}}
\left|\widehat{H_{u} \cdot G}(k,\tau)\right|^{2} \;d\tau\right)^{\frac{1}{2}} \|w\|_{X_{-s,\frac{1}{2}}}\right]\\
&\leq\displaystyle{ C_{s,\alpha}\left(\sum_{k=-\infty}^{\infty}\int_{\mathbb{R}}
	\langle\tau-\phi(k)\rangle^{-\frac{2}{3}}
	\left|\widehat{H_{u} \cdot G}(k,\tau)\right|^{2}\;d\tau\right)^{\frac{1}{2}}}\\
&\leq C_{s,\alpha} \|H_{u}\cdot G\|_{L^{\frac{4}{3}}(\mathbb{T}\times\mathbb{R})}\\
&\leq C_{s,\alpha}\displaystyle{\|H_{u}\|_{L^{2}(\mathbb{T}\times\mathbb{R})}
	\| G\|_{L^{4}(\mathbb{T}\times\mathbb{R})}},
\end{split}}$$
where we  applied Corollary \ref{dualestim}, and Holder inequality in the last two  inequalities.
		From Theorem \ref{fundaimme}, we have  $\| G\|_{L^{4}(\mathbb{T}\times\mathbb{R})}\leq \| G\|_{X_{0,\frac{1}{3}}}.$
		On the other hand, 
$${
		\begin{split}
		\displaystyle{\|H_{u}\|_{L^{2}(\mathbb{T} \times\mathbb{R})}}
		&\sim\displaystyle{ \left(\sum_{k=-\infty}^{\infty}\int_{\mathbb{R}}
			\langle k\rangle^{2s}\langle\tau-\phi(k)\rangle^{2\cdot\frac{1}{2}}|\widehat{u}(k,\tau)|^{2}\;d\tau\right)^{\frac{1}{2}}}
		=\displaystyle{\|u\|_{X_{s,\frac{1}{2}}}}.
		\end{split}}$$
		
		Therefore, 
		$\displaystyle{\left\|\partial_{x}(u v)\right\|_{X_{s,-\frac{1}{2}}}}
			\leq C_{s,\alpha}\displaystyle{\|u\|_{X_{s,\frac{1}{2}}} \| G\|_{X_{0,\frac{1}{3}}}}
			\leq C_{s,\alpha}\displaystyle{\|u\|_{X_{s,\frac{1}{2}}} \| v\|_{X_{s,\frac{1}{3}}}}.$
		
		\noindent{\bf Case 3.  $|\tau-\tau_{1}-\phi(k-k_{1})|>\frac{3}{8}C_{\alpha}k^{2}:$}
		Observe that, this case is similar to the second one, just substituting $H_{v}$ in the place of $H_{u}$ and $F$ in the place of $G.$ Thus, we obtain
		\begin{align*}
			\displaystyle{\left\|\partial_{x}(u v) \right\|_{X_{s,-\frac{1}{2}}}}
			&\leq C_{s,\alpha}\displaystyle{\|H_{v}\|_{L^{2}(\mathbb{T} \times \mathbb{R})} \| F\|_{X_{0,\frac{1}{3}}}}
			\leq C_{s,\alpha}\displaystyle{\|v\|_{X_{s,\frac{1}{2}}}
			 \| u\|_{X_{s,\frac{1}{3}}}}.
		\end{align*}

\noindent
{\textbf{Step 2.}} 	Now we estimate the $Y_{s,-1}$ norm. Using duality we have,
	\begin{equation}\label{bilestim14}
		\begin{split}
			\displaystyle{\|\partial_{x}(uv)\|_{Y_{s,-1}}}
			&\sim
			\displaystyle{\left\|\; \|(1+| k|)^{s}(1+ |\tau - \phi(k)|)^{-1}\; \widehat{\partial_{x}(uv)}(k,\tau)
			\|_{L^{1}_{\tau}(\mathbb{R})} \right\|_{l^{2}_{k}}}\\
			&=\displaystyle{\sup_{\substack{a_{k}\in l^{2}_{k},\;a_{k}\geq0\\ \|a_{k}\|_{l^{2}_{k}}=1}}
				\sum_{k\in \mathbb{Z}}a_{k} \left(\int\limits_{\mathbb{R}}
				(1+| k|)^{s}(1+ |\tau - \phi(k)|)^{-1}|\widehat{\partial_{x}(uv)} (k,\tau)|\;d\tau\right)}.
	\end{split}
\end{equation}

We move to estimate
\begin{equation}\label{bilestim15}
	II:=\displaystyle{(1+| k|)^{s}(1+ |\tau- \phi(k)|)^{-1} |\widehat{\partial_{x}(uv)}(k,\tau)|}.
\end{equation}

Note that
\begin{equation}\label{bilestim16}
		\begin{split}
		II&\leq\displaystyle{(1+| k|)^{s}|k|(1+ |\tau - \phi(k)|)^{-1}}
		\displaystyle{\left( \sum_{  \substack{ k_{1} \in \mathbb{Z}\\ k_{1}\neq 0 }} \int\limits_{\mathbb{R}}
		|\widehat{u}(k_{1},\tau_{1})||\widehat{v}(k-k_{1},\tau-\tau_{1})|\;d\tau_{1}\right)}\\
			&\leq\displaystyle{|k|(1+ |\tau- \phi(k)|)^{-1}}
			\displaystyle{\left( \sum_{  \substack{ k_{1} \in \mathbb{Z}\\ k_{1}\neq 0 }} \int\limits_{\mathbb{R}}
				\frac{(1+| k|)^{s} c_{u}(k_{1},\tau_{1}) c_{v}(k-k_{1},\tau-\tau_{1})}
				{| k_{1}|^{s}\langle\tau_{1} - \phi(k_{1})\rangle^{\frac{1}{2}}
					| k-k_{1}|^{s}
					\langle\tau-\tau_{1} - \phi(k-k_{1})\rangle^{\frac{1}{2}} }d\tau_{1}\right)}\\
			&\leq C_{s}\left( \sum_{  \substack{ k_{1} \in \mathbb{Z}\\ k_{1}\neq 0 }} \int\limits_{\mathbb{R}}
			\frac{|k| c_{u}(k_{1},\tau_{1}) c_{v}(k-k_{1},\tau-\tau_{1})}
			{(1+ |\tau - \phi(k)|)\langle\tau_{1} - \phi(k_{1})\rangle^{\frac{1}{2}}
				\langle\tau-\tau_{1} - \phi(k-k_{1})\rangle^{\frac{1}{2}} }d\tau_{1}\right).
	\end{split}
\end{equation}

	It follows from identity \eqref{bilestim14}, definition \eqref{bilestim15} and relation \eqref{bilestim16} that
$${\small
		\begin{split}
			\|\partial_{x}(uv)\|_{Y_{s,-1}}
			&\leq C_{s}
			\sup_{\substack{a_{k}\in l^{2}_{k},\;a_{k}\geq 0\\\|a_{k}\|_{l^{2}_{k}}=1}}
			\sum_{  \substack{k, k_{1} \in \mathbb{Z}\\ kk_{1}(k-k_{1})\neq 0 }} \int\limits_{\mathbb{R}^{2}}
			\frac{a_{k} |k| c_{u}(k_{1},\tau_{1}) c_{v}(k-k_{1},\tau-\tau_{1})}
			{(1+ |\tau - \phi(k)|)\langle\tau_{1} - \phi(k_{1})\rangle^{\frac{1}{2}}
				\langle\tau-\tau_{1} - \phi(k-k_{1})\rangle^{\frac{1}{2}} }d\tau_{1}d\tau.
	\end{split}}$$	

As in {\bf Step 1.,} in view of Remark \ref{resonance property1}, here we divide the sum into small  and large frequencies to consider three different cases and obtain 
$$\|\partial_{x}(u v)\|_{Y_{s,-1}}\leq C_{s,\alpha}\displaystyle{\left\|u\right\|_{X_{s,\frac{1}{2}}} \left\|v\right\|_{X_{s,\frac{1}{3}}} },$$ in the first two cases and 
$$\displaystyle{\left\|\partial_{x}(u v)\right\|_{Y_{s,-1}}}
\leq  C_{s,\alpha}\displaystyle{\|v\|_{X_{s,\frac{1}{2}}} \| u\|_{X_{s,\frac{1}{3}}}},$$
in the third case. We omit the details (see the proof of lemma 7.42 in \cite{13}).	
\end{proof}

\begin{cor} \label{Biestimate2}
Let  $s\geq0,$ $\alpha>0,$  and  $T>0$ be given.
Assume that
 $u,v:\mathbb{T}\times \mathbb{R}\rightarrow \mathbb{R}$ are  functions in $X_{s,\frac{1}{3}}^{T}$
 and $X_{s,\frac{1}{2}}^{T}$ with
 mean $[u(\cdot,t)]=[v(\cdot,t)]=0$ for each $t\in \mathbb{R}.$ Then
$$\|\partial_{x}(uv)\|_{Z_{s,-\frac{1}{2}}^{T}}\leq C_{\alpha,s}\Big(\|u\|_{X_{s,\frac{1}{2}}^{T}}
\|v\|_{X_{s,\frac{1}{3}}^{T}}+\|u\|_{X_{s,\frac{1}{3}}^{T}} \|v\|_{X_{s,\frac{1}{2}}^{T}}\Big).$$
\end{cor}

\begin{cor} \label{Biestimate3}
Let $s\geq0,$ $\alpha>0,$  and $T>0$ be given.
Assume that
$v:\mathbb{T}\times \mathbb{R}\rightarrow \mathbb{R}$ is a function in $X^{T}_{s,\frac{1}{2}}$
with mean $[v(\cdot,t)]=0$ for each $t\in \mathbb{R}.$  Then there exist $0<\epsilon<\frac{1}{6}$ such that
$$\|\partial_{x}(v^{2})\|_{Z^{T}_{s,-\frac{1}{2}}}\leq C_{\alpha,s}T^{\epsilon}\;\|v\|^{2}_{Z^{T}_{s,\frac{1}{2}}}.$$
\end{cor}
\begin{proof}
Applying Proposition \ref{Xestimate1} with $b'=\frac{1}{3},$  $0<\epsilon<\frac{1}{6},$ and $b=\frac{1}{3}+\epsilon,$ we obtain
$$
\|u\|_{X^{T}_{s,\frac{1}{3}}}
\leq C\;T^{\frac{1}{3}+\epsilon-\frac{1}{3}}
\|u\|_{X^{T}_{s,\frac{1}{3}+\epsilon}}
\leq
C\;T^{\epsilon}
\|u\|_{X^{T}_{s,\frac{1}{2}}}.$$

From Corollary \ref{Biestimate2} and \ref{Biestimate3}, we have
$$
\|\partial_{x}(u^{2})\|_{Z^{T}_{s,-\frac{1}{2}}}
\leq C 	\|u\|_{X^{T}_{s,\frac{1}{2}}}
\|u\|_{X^{T}_{s,\frac{1}{3}}}
\leq
C\;T^{\epsilon}
\|u\|^{2}_{X^{T}_{s,\frac{1}{2}}}
\leq
C\;T^{\epsilon}
\|u\|^{2}_{Z^{T}_{s,\frac{1}{2}}}.$$
\end{proof}

\section{Propagation of Compactness, Regularity and Unique Continuation Property} \label{sec-UCP}

In this section we establish some results on propagation of compactness, regularity and unique continuation property satisfied by the solution of the Benjamin equation that are essential to  prove the global exponential stabilization.

\subsection{The Multiplication Property of the Bourgain's Space}

In the following Lemmas we establish the multiplication property
of the Bourgain space $X_{s,b}.$

\begin{lem}\label{multi1}
	If $\psi=\psi(t)\in C^{\infty}(\mathbb{R}),$ then
	$\psi v \in X^{T}_{s,b}$ for all $v\in X^{T}_{s,b}.$ Furthermore, there exists a positive constant $C=C_{\eta,T,b,\psi}$
	 such that
	 \begin{equation}\label{mult1}
	 \|\psi v\|_{X^{T}_{s,b}}\leq C
	 \| v\|_{X^{T}_{s,b}}.
	 \end{equation}

	 If $T\leq 1,$ then positive constant $C$ does not depend on the time $T.$
\end{lem}
\begin{proof}
	Let $v=v(x,t) \in X^{T}_{s,b}$ and consider $w$
	an extention to $X_{s,b}$ of $v$ such that
	$$\|w\|_{X_{s,b}}\leq
	2 \;\|v\|_{X_{s,b}}.$$
	
	Thus, $\eta_{T}(t) \psi(t) w(x,t)$
	is an extention to $X_{s,b}$ of
	$\psi(t) v(x,t).$ Note that
	\begin{equation}\label{multi2}
		\begin{split}
			\left\| \psi(t) v
			\right\|^{2}_{X^{T}_{s,b}}
			&\leq
			\left\| \eta_{T}(t) \psi(t) w
			\right\|^{2}_{X_{s,b}}\\
			&=\displaystyle{
				\sum_{k=-\infty}^{\infty}
				\int_{\mathbb{R}}
				\langle k\rangle^{2s}\langle \tau -\phi(k) \rangle^{2b}
				\left|\left(\left(\eta_{T}(\cdot)\;\psi(\cdot)
				\right)^{\wedge}\ast
				\;\widehat{w}(k,\cdot)
				\right)(\tau)\right|^{2}\;d\tau}.
		\end{split}
	\end{equation}
	
	If $b=0$ in \eqref{multi2}, we use Young Inequality to obtain
$$\left\| \psi(t) v
			\right\|^{2}_{X^{T}_{s,0}}
			\leq \displaystyle{
				\sum_{k=-\infty}^{\infty}\langle k\rangle^{2s}
				\left\|\left(\eta_{T}(\cdot)\;\psi(\cdot)
				\right)^{\wedge}(\tau)
				\right\|_{L^{1}_{\tau}(\mathbb{R})}^{2}
				\left\|
				\;\widehat{w}(k,\tau)
				\right\|_{L^{2}_{\tau}(\mathbb{R})}^{2}}.$$
	
	Since $\eta_{T}(t)\psi(t)\in C^{\infty}_{c}(\mathbb{R})$
	we have \;$\displaystyle{
			\left\|\left(\eta_{T}(\cdot)\;\psi(\cdot)
			\right)^{\wedge}(\tau)
			\right\|_{L^{1}_{\tau}(\mathbb{R})}^{2}
			<\infty}.$
	Thus, there exists a positive constant $C_{T,\eta,\psi}$
	such that
$${
	\begin{split}
			\left\| \psi(t) v
			\right\|^{2}_{X^{T}_{s,0}}
			&\leq \displaystyle{ C_{T,\eta,\psi}\;
				\sum_{k=-\infty}^{\infty}\langle k\rangle^{2s}
				\left\|
				\;\widehat{w}(k,\tau)
				\right\|_{L^{2}_{\tau}(\mathbb{R})}^{2}}
			=	C_{T,\eta,\psi}\;\left\| w
			\right\|^{2}_{X^{T}_{s,0}}
			= 2\;	C_{T,\eta,\psi}\;\left\| v
			\right\|^{2}_{X^{T}_{s,0}}.
	\end{split}}$$
	
	On the other hand, if $b>0,$ we have the following inequality
	\begin{equation}\label{multi6}
		\begin{split}
			\langle \tau-\phi(k)\rangle^{b}
			&\leq c_{b}\; \langle \tau-y-\phi(k)\rangle^{b}\;
			\langle y \rangle^{b}.
		\end{split}
	\end{equation}
	
	From \eqref{multi2}, \eqref{multi6} and Young inequality, we have
	\begin{equation}\label{multi7}
		\begin{split}
			\left\| \psi(t)v
			\right\|^{2}_{X^{T}_{s,b}}
			&\leq \displaystyle{ c_{b}
				\sum_{k=-\infty}^{\infty}
				\langle k\rangle^{2s}
				\left\| \left(
				\langle \cdot  \rangle^{b}
				\left(\eta_{T}(t)\;\psi(t)
				\right)^{\wedge}(\cdot)\ast
				\langle \cdot-\phi(k) \rangle^{b}
				\;\widehat{w}(k,\cdot)\right)(\tau)
				\right\|^{2}_{L^{2}_{\tau}(\mathbb{R})}}\\
			&\leq \displaystyle{ c_{b}
				\sum_{k=-\infty}^{\infty}
				\langle k\rangle^{2s}
				\left\|
				\langle \tau  \rangle^{b}
				\left(\eta_{T}(t)\;\psi(t)
				\right)^{\wedge}(\tau)
				\right\|^{2}_{L^{1}_{\tau}(\mathbb{R})}
				\left\| \langle \tau-\phi(k) \rangle^{b}
				\;\widehat{w}(k,\tau)
				\right\|^{2}_{L^{2}_{\tau}(\mathbb{R})}}.
		\end{split}
	\end{equation}

	Since $\eta_{T}(t)\psi(t)\in C^{\infty}_{c}(\mathbb{R})$
	we get \;$\displaystyle{\left\|\langle \tau  \rangle^{b}
			\left(\eta_{T}(\cdot)\;\psi(\cdot)
			\right)^{\wedge}(\tau)
			\right\|_{L^{1}_{\tau}(\mathbb{R})}^{2}
			<\infty}.$
	Therefore, there exists a positive constant $C_{T,\eta,b,\psi}$
	such that
$${
		\begin{split}
			\left\| \psi(t) v
			\right\|^{2}_{X^{T}_{s,b}}
			&\leq \displaystyle{C_{T,\eta,b,\psi}
				\sum_{k=-\infty}^{\infty}\langle k\rangle^{2s}
				\left\|\langle \tau-\phi(k) \rangle^{b}
				\;\widehat{w}(k,\tau)
				\right\|_{L^{2}_{\tau}(\mathbb{R})}^{2}}
			=C_{T,\eta,\psi} \left\| w
			\right\|^{2}_{X^{T}_{s,b}}
			= 2	C_{T,\eta,\psi} \left\| v
			\right\|^{2}_{X^{T}_{s,b}}.
		\end{split}}$$	
	
	Finally, if $b<0$ we infer that
$\langle \tau-\phi(k)\rangle^{b}
			\leq c_{b}\;
			\langle \tau-y-\phi(k)\rangle^{b}\;
			\langle y\rangle^{-b},$
	and similar computations as those yielding  \eqref{multi7} give the result. Note that, if $T\leq 1$ then  $\eta(t) \psi(t) w(x,t)$
	is an extention to $X_{s,b}$ of
	$\psi(t) v(x,t).$ Consequently, the constant $C$ in the estimates will only depend on $\eta,b,$ and $\psi.$
\end{proof}

As was pointed out by
Laurent  et. al \cite{14} for the KdV equation,
if $\phi=\phi(x)\in C^{\infty}(\mathbb{T}),$
then $\phi v$ may not belong to the space $X^{T}_{s,b}$
for $v \in X^{T}_{s,b}.$ For the Benjamin equation too, the same is lost in the index of regularity $s$
because the structure in space of the harmonics is not kept by the multiplication by a (smooth) function
of $x$ 
(see Example \ref{examplemulti} in the appendix).
In fact, this is reflected in next theorem. We first  prove a necessary lemma.

\begin{lem}\label{multi44}
	Let $s\in \mathbb{R}.$\;\;
	 $v\in X_{s,1}$ if and only if
	$v\in L^{2}(\mathbb{R},H^{s}_{p}(\mathbb{T})),$ and $$\partial_{t}v-\partial^{3}_{x}v
	-\alpha \mathcal{H}\partial^{2}_{x}v
	+2\;\mu \partial_{x}v \in
	L^{2}(\mathbb{R},H^{s}_{p}(\mathbb{T})).$$
In this case we have\;\;
	$\|v\|^{2}_{X_{s,1}}=
	\|v\|^{2}_{L^{2}(\mathbb{R}_{t},H^{s}_{p}(\mathbb{T}))}
	+\|\partial_{t}v-\partial^{3}_{x}v
	-\alpha \mathcal{H}\partial^{2}_{x}v
	+2\;\mu \partial_{x}v
	\|^{2}_{L^{2}(\mathbb{R}_{t},H^{s}_{p}(\mathbb{T}))}.$
\end{lem}

\begin{proof}
	Let $s\in \mathbb{R}$ be fixed. Then, applying Plancherel identity in time, we obtain
$${
	\begin{split}
    \|v\|^{2}_{X_{s,1}}
    &=\sum_{k=-\infty}^{\infty}\int_{\mathbb{R}}\langle k\rangle^{2s}
    \left(1+|\tau +k^{3}-\alpha k|k|+2\mu k|^{2}\right)|\widehat{v}(k,\tau)|^{2}\;d\tau\\
    &=\|v\|^{2}_{L^{2}(\mathbb{R}_{t},H^{s}_{p}(\mathbb{T}))}
    +\sum_{k=-\infty}^{\infty}\int_{\mathbb{R}}\langle k\rangle^{2s}
    |\widehat{\partial_{t}v}(k,t)
    -\widehat{\partial^{3}_{x}}v(k,t)
    -\alpha \widehat{\mathcal{H}\partial^{2}_{x}v}(k,t)
    +2\;\mu \widehat{\partial_{x}v}(k,t)|^{2}\;dt.
	\end{split}}$$
This proves the Lemma.
\end{proof}

Using the Fourier transform, $X_{s,b}$ may be viewed
as the weighted $L^{2}$ space
\begin{equation}\label{measure2}
L^{2}(\mathbb{R}_{\tau}\times\mathbb{Z}_{k},
\langle k\rangle^{2s}
\langle\tau-\phi(k)\rangle^{2b}
\lambda\otimes \delta),
\end{equation}
 where
$\lambda$  and $\delta$ are the Lebesgue measure over $\mathbb{R}$ and the discrete measure on $\mathbb{Z},$ respectively.

\begin{thm}\label{multi23}
	Let $-1\leq b\leq1,$ $s\in \mathbb{R},$ and
	$\varphi \in C^{\infty}(\mathbb{T}).$
	Then for any $v\in X_{s,b},$ $\varphi v
	\in X_{s-2|b|,b}.$ Similarly, the multiplication
	by $\varphi$ maps $X^{T}_{s,b},$ into
	$X^{T}_{s-2|b|,b}$ i.e, there exists a
	positive constant
	$C=C_{s,\alpha,\varphi,\mu}$ such that
	\begin{equation}\label{mult2}
	\|\varphi v\|_{X^{T}_{s-2|b|,b}}\leq C\;
	\|v\|_{X^{T}_{s,b}}.
	\end{equation}	

\end{thm}

\begin{proof}
We proceed as in \cite{14}. First, we show \eqref{mult2} for the cases $b=0$ and $b=1.$ Afther, using interpolation and duality we show \eqref{mult2} for the other cases of $b.$

	\noindent{\bf Case 1. $b=0$:} Let  $v\in\mathcal{S}(\mathbb{T}\times \mathbb{R}).$ From definition, we have
$$
\|\varphi v\|^{2}_{X_{s,0}}
	=
    	\sum_{k\in \mathbb{Z}}
    	\int_{\mathbb{R}}2\pi\;
    	( 1+|k|^{2} )^{s}
    	|\left(\varphi\;v\right)^{\wedge}(k,t)
    	|^{2}\;dt.
$$

If $s\geq 0,$ one has
\;$\displaystyle{
(1+|k|)^{s}\leq c_{s}\;
(1+|k-j|)^{s}(1+|j|)^{s}}.$ Therefore,
 the Cauchy-Schwartz inequality for $N>\frac{1}{2},$ yields
$${
\begin{split}
\displaystyle{\|\varphi v\|^{2}_{X_{s,0}}}
&\leq c_{s}
\displaystyle{
	\sum_{j=-\infty}^{\infty}(1+|j|)^{2s+2N}
	|\widehat{\varphi}
	(j)|^{2}	\sum_{k\in \mathbb{Z}}
	\int_{\mathbb{R}}
	\;(1+|k-j|)^{2s}\; |\widehat{v(x,t)}(k-j,t)|^{2}
	\;dt}.
\end{split}}$$

Using the
invariance of the
$L^{2}(\mathbb{R}_{t},H^{s}_{p}(\mathbb{T}))$ norm by translations, we get
$$\displaystyle{\|\varphi v\|^{2}_{X_{s,0}}}
\leq c_{s}
\displaystyle{\|\varphi\|^{2}_{H^{s+N}_{p}(\mathbb{T})}
	\|v\|^{2}_{X_{s,0}}}.$$

If $s<0,$  we apply
$(1+|k|)^{s} \leq c_{s} (1+|k-j|)^{s}(1+|j|)^{-s}$
and  proceed as above to obtain \eqref{mult2}.
In the general case, we use density and duality arguments to complete the proof.

\noindent{\bf Case 2.  $b=1:$}
From Lemma \ref{multi44}, we obtain
{\small
\begin{equation}\label{multi46}
\begin{split}
\|\varphi v\|^{2}_{X_{s-2,1}}\!\!\!
&=\|\varphi v
\|^{2}_{L^{2}(\mathbb{R}_{t},H^{s-2}_{p}(\mathbb{T}))}\\
&\quad
+\|\partial_{t}(\varphi v)-\partial^{3}_{x}(\varphi v)
-\!\alpha \mathcal{H}\partial^{2}_{x}(\varphi v)
+2\mu \partial_{x}(\varphi v)
-\!\alpha\varphi \mathcal{H}\partial^{2}_{x}v
+\!\alpha\varphi \mathcal{H}\partial^{2}_{x}v
\|^{2}_{L^{2}(\mathbb{R}_{t},H^{s-2}_{p}(\mathbb{T}))}\\
&\leq\|\varphi v
\|^{2}_{X_{s-2,0}}
+c\;\alpha\| \mathcal{H}\partial^{2}_{x}(\varphi v)
+\varphi \mathcal{H}\partial^{2}_{x}v
\|^{2}_{L^{2}(\mathbb{R}_{t},H^{s-2}_{p}(\mathbb{T}))}\\
&\quad
+c\; \|\partial_{t}(\varphi v)-\partial^{3}_{x}(\varphi v)
+2\;\mu \partial_{x}(\varphi v)
-\alpha\varphi \mathcal{H}\partial^{2}_{x}(v)
\|^{2}_{L^{2}(\mathbb{R}_{t},H^{s-2}_{p}(\mathbb{T}))}\\
&=:I+II+III.
\end{split}
\end{equation}}

From case $b=0,$ we obtain that there exists $A_{s,\varphi}>0$ such that
\begin{equation}\label{multi47}
   I=\|\varphi v
	\|^{2}_{X_{s-2,0}}
	\leq
	A_{s,\varphi}\;\|v
	\|^{2}_{X_{s-2,0}}
	\leq
	A_{s,\varphi}\;\|v
	\|^{2}_{X_{s,0}}.
\end{equation}

From properties of the operator $\partial^{r}_{x}$ on
Bourgain's spaces, and noting that $\mathcal{H}$ is an isomethry in $H_{p}^{s-2}(\mathbb{T}),$ we get
$${
\begin{split}
II
&\leq
c \alpha \left(
\| \partial^{2}_{x}(\varphi v)
\|^{2}_{L^{2}(\mathbb{R}_{t},H^{s-2}_{p}(\mathbb{T}))}
+\| \varphi \mathcal{H}\partial^{2}_{x}v
\|^{2}_{X_{s-2,0}}
\right)\\
&\leq
c \alpha \left(
\|\varphi v
\|^{2}_{X_{s,0}}
+c_{s,\varphi}\; \|\mathcal{H}\partial^{2}_{x}v
\|^{2}_{L^{2}(\mathbb{R}_{t},H^{s-2}_{p}(\mathbb{T}))}
\right)\\
&\leq
c \alpha \left(d_{s,\varphi}\;
\|v
\|^{2}_{X_{s,0}}
+ c_{s,\varphi}\;\|\partial^{2}_{x}v
\|^{2}_{X_{s-2,0}}
\right).
\end{split}}$$

Hence,  there exists another positive constant
 $B_{s,\alpha,\varphi}$ such that
\begin{equation}\label{multi53}
	\begin{split}
	II
	&\leq
     B_{s,\alpha,\varphi}\;
      \|v	\|^{2}_{X_{s,0}}.	
	\end{split}
\end{equation}

We estimate $III.$
From the Leibniz's rule for derivatives
one has
\begin{equation}\label{multi52}
	\begin{split}
	\partial_{t}(\varphi v)-\partial^{3}_{x}(\varphi v)
	+2\mu \partial_{x}(\varphi v)
	-\alpha\varphi  \mathcal{H}\partial^{2}_{x}v
	&=
	\varphi \left(
	\partial_{t}v-\partial^{3}_{x}v
	+2\mu \partial_{x}v
	-\alpha \mathcal{H}\partial^{2}_{x}v\right)\\
	&\quad
	-3\partial_{x}\varphi\partial^{2}_{x}v
	-3\partial^{2}_{x}\varphi\partial_{x}v
	-\partial^{3}_{x}\varphi v
	+2\;\mu \partial_{x}\varphi v.
	\end{split}
\end{equation}

Note that $-3\partial_{x}\varphi\;\partial^{2}_{x}v
-3\partial^{2}_{x}\varphi\;\partial_{x}v
-\partial^{3}_{x}\varphi\;v
+2\;\mu \partial_{x}\varphi\;v$ is an operator of second order for $v$.
Fron identity
\eqref{multi52},
 the case $b=0,$ and the fact that $\varphi\in C^{\infty}(\mathbb{T}),$ we get
\begin{equation*}
\begin{split}
III
&\leq
c \|\varphi(x)\left(
\partial_{t}v-\partial^{3}_{x}v
+2\mu \partial_{x}v
-\alpha \mathcal{H}\partial^{2}_{x}v\right)
\|^{2}_{L^{2}(\mathbb{R}_{t},H^{s-2}_{p}(\mathbb{T}))}\\
&\quad
+c\|-3\partial_{x}\varphi\;\partial^{2}_{x}v
-3\partial^{2}_{x}\varphi\;\partial_{x}v
-\partial^{3}_{x}\varphi\;v
+2\mu \partial_{x}\varphi\;v
\|^{2}_{L^{2}(\mathbb{R}_{t},H^{s-2}_{p}(\mathbb{T}))}\\
&\leq
c_{s,\varphi} \|
\partial_{t}v-\partial^{3}_{x}v
+2 |\mu| \partial_{x}v
-\alpha \mathcal{H}\partial^{2}_{x}v
\|^{2}_{X_{s-2,0}}\\
&\quad
+3c d_{s,\partial_{x}\varphi} \|v
\|^{2}_{X_{s,0}}
+3c d_{s,\partial^{2}_{x}\varphi} \|v
\|^{2}_{X_{s-1,0}}
+c d_{s,\partial^{3}_{x}\varphi} \|v
\|^{2}_{X_{s-2,0}}
+2 |\mu| c d_{s,\partial_{x}\varphi} \| v
\|^{2}_{X_{s-2,0}}.
\end{split}
\end{equation*}

 Using
 $X_{s,0}\hookrightarrow X_{s-2,0},$ and
 $X_{s,0}\hookrightarrow X_{s-1,0},$ we have that there exists $D_{s,\mu,\varphi}>0$ such that
 \begin{equation}\label{multi56}
 \begin{split}
 III
 &\leq
 D_{s,\mu,\varphi}\;\left(\|
 \partial_{t}v-\partial^{3}_{x}v
 +2\;\mu \partial_{x}v
 -\alpha \mathcal{H}\partial^{2}_{x}v
 \|^{2}_{X_{s,0}}+
  \|v
 \|^{2}_{X_{s,0}}\right)=
 D_{s,\mu,\varphi}\;
 \|v
 \|^{2}_{X_{s,1}},
 \end{split}
 \end{equation}
 where in the las step Lemma \ref{multi44} is used.
 From
 \eqref{multi46}-\eqref{multi53}, and
 \eqref{multi56}, we obtain that there exists a positive constant
 $C_{s,\alpha, \mu,\varphi}$ such that
$$
 \|\varphi(x) v\|^{2}_{X_{s-2,1}}
 \leq
 C_{s, \alpha, \mu,\varphi}\;
 \|v\|^{2}_{X_{s,1}}.$$
 This proves the case $b=1.$

 \noindent{\bf Case 3. $0<b<1$:} In this case we use interpolation. From identification \eqref{measure2}
and the Complex Interpolation Theorem
 of Stein-Weiss (see Bergh and Lofstrom
 \cite[page 115]{Bergh and Lofstrom}),
 we obtain
$\left(X_{s,0}, \;X_{s',1}\right)_{\theta,2}
 \approx X_{s(1-\theta)+s'\theta,\theta},$
 with $0<\theta<1.$ Furthermore, from the cases $b=0$ and $b=1$ we infer that the operator of
 multiplication by $\varphi \in C^{\infty}(\mathbb{T}),$ defined by
 $${\small
 	\begin{split}
 	T:X_{s,\theta}\approx L^{2}(\mathbb{R}_{\tau}\times \mathbb{Z}_{k},
 	\langle k\rangle^{2s}
 	\langle\tau-\phi(k)\rangle^{2\theta} \lambda \otimes \delta)&
 	\longrightarrow X_{s-2\theta,\theta}\approx
 	L^{2}(\mathbb{R}_{\tau}\times \mathbb{Z}_{k},
 	\langle k\rangle^{2(s-2\theta)} \langle\tau-\phi(k)
 	\rangle^{2\theta} \lambda \otimes \delta)\\
 	&T(v)=\varphi(x) v,
 	\end{split}}$$
 satisfies
$\|Tv\|_{X_{s-2\theta,\theta}}\leq C_{s,\varphi}^{1-\theta}
 C_{\alpha,s,\varphi,\mu}^{\theta}
 \|v\|_{s,\theta}
 \leq
 C_{\alpha,s,\varphi,\mu,\theta}\;\|v\|_{s,\theta},$
 with $0<\theta<1.$ Thus, we have a $2\theta$ loss of regularity
 in the spatial variable.

 \noindent{\bf Case 4. $-1< b< 0$:} In this case we use duality,
$$ 
 \|\varphi(x)v\|_{X_{s-2|b|,b}}
 =
 	\underset{\|u\|_{X_{-s-2b,-b}}\leq 1}
 	{\sup_{u \in X_{-s-2b,-b}}}
 	\left|\int_{\mathbb{T}}\int_{\mathbb{R}}
 	u \cdot \varphi(x)v \;dt\;dx\right|.
$$

This completes the proof of the theorem.
\end{proof}

\subsection{Propagation of Compactness and Regularity}
Here we show some properties of propagation of compactness and regularity for the linear operator
\begin{equation}\label{compact1}
L:= \partial_{t}-\alpha\;\mathcal{H}\partial_{x}^{2}
-\partial_{x}^{3}+2\mu\;\partial_{x},
\end{equation}
associated to  the Benjamin equation. These propagation properties are fundamental to study global stabilizability. We begin establishing two technical lemmas.

\begin{lem}\label{compact24}
	Let $\alpha>0$ and $\mu\in\mathbb{R}.$ The operator
	\begin{equation}\label{compact25}
	L:D(L)\subseteq
	L^{2}(\mathbb{T}\times(0,T))
	\rightarrow L^{2}(\mathbb{T}\times (0,T)),\;
	\end{equation}
	defined by \eqref{compact1}
	is skew-adjoint
	on $L^{2}(\mathbb{T}\times(0,T)),$
	where
	$$D(L)=\left\{
	v\in \mathcal{D}'(\mathbb{T}\times (0,T))
	: v(x,\cdot)\in H^{1}(0,T),
	\;\text{and}\;v(\cdot,t)\in
	 H_{p}^{3}(\mathbb{T})\right\}.$$
\end{lem}

\begin{proof}
	The proof is similar to that of Proposition 3.3 in \cite{Vielma and Panthee}, so we omit the details.
\end{proof}

\begin{lem}\label{compact26}
	Let $r\in \mathbb{R}.$ The Hilbert transform $\mathcal{H}$ commutes with the operator $D^{r}$ (see \eqref{difop}) in
	$L^{2}(\mathbb{T}).$ Furthermore,
	$\mathcal{H}=-D^{-1}\partial_{x}$ in
	$L^{2}(\mathbb{T}).$ Also, the operator
	 $\partial_{x}^{r}$ commutes with the operators $D^{r}$
	 and $\mathcal{H}$ in  $L^{2}(\mathbb{T}).$
\end{lem}
\begin{proof}
	It can be easily shown using Fourier transform.
\end{proof}

\begin{prop}[Propagation of Compactness]\label{compact2}
	Let $T>0$ and $0\leq b'\leq b\leq 1$ be given (with $b>0$)
	and assume that $v_{n}\in X^{T}_{0,b}$ and
	$f_{n}\in X^{T}_{-2+2b,-b}$ satisfy
	\begin{equation}\label{compact3}
	\partial_{t}v_{n}-\alpha\;\mathcal{H}
	\partial_{x}^{2}v_{n}
	-\partial_{x}^{3}v_{n}+2\mu\;\partial_{x}v_{n}
	=f_{n},
	\end{equation}
	for $n=1,2,3,\cdots.$
	Suppose that there exists  $C>0$ such that
	\begin{equation}\label{compact4}
	\|v_{n}\|_{X^{T}_{0,b}}\leq C,
	\end{equation}
	for all $n\geq 1,$ and that
	\begin{equation}\label{compact5}
	\|v_{n}\|_{X^{T}_{-2+2b,-b}}
	+\|f_{n}\|_{X^{T}_{-2+2b,-b}}
	+\|v_{n}\|_{X^{T}_{-1+2b',-b'}}\longrightarrow 0,
	\end{equation}
	as $n\longrightarrow \infty.$
	Additionally, assume that for some nonempty open set
	$\omega \subset \mathbb{T}$
	\begin{equation}\label{compact6}
	v_{n}\longrightarrow 0,\;\;
	\text{strongly in}\;\;
	L^{2}((0,T);L^{2}(\omega)).
	\end{equation}
	
	Then, there exists a subsequence
	$\left\{v_{n_{j}}\right\}_{j\in \mathbb{N}}$ of
	$\{v_{n}\}_{n\in \mathbb{N}}$ such that
	\begin{equation}\label{compact7}
	v_{n_{j}}\longrightarrow 0,\;\;
	\text{strongly in}\;\;
	L^{2}_{loc}((0,T);L^{2}(\mathbb{T})),\;\;
	\text{as}\;j\longrightarrow \infty.
	\end{equation}
\end{prop}

\begin{proof}
	Let $K\subset (0,T)$ be compact and   $\psi \in C_{c}^{\infty}((0,T))$ be such that
	$0\leq \psi(t)\leq 1$ and $\psi(t)=1$ in $K.$
	Then,
\begin{equation*}
	\|v_{n}\|_{L^{2}(K,L^{2}(\mathbb{T}))}
	\leq \int_{0}^{T}\psi(t)
	\|v_{n}\|^{2}_{L^{2}(\mathbb{T})}\;dt
	=\int_{0}^{T}\psi(t)\;\left(v_{n},
	v_{n}\right)_{L^{2}(\mathbb{T})}\;dt.
\end{equation*}
	
Since $\mathbb{T}$ is compact there exists a finite
set of points, say $x_{0}^{i}\in \mathbb{T},$  $i=1,\cdots,N,$ such that we can construct a partition of the unity on $\mathbb{T}$
involving functions of the form
$\chi_{i}(\cdotp -x_{0}^{i})$ with $\chi_{i}(\cdotp)\in C_{c}^{\infty}(\omega).$ 
Specifically, there exists $N\in \mathbb{N}$ such that
 \begin{equation}\label{compact12}
	 \left\{
	 \begin{array}{lcl}
	 0\leq\chi_{i}(x-x_{0}^{i})\leq 1, & \mbox{for all} &  x\in \mathbb{T}\;\;\text{and}
	 \;\;i=1,2,..,N\\
	\chi_{i}(\cdotp)\in C_{c}^{\infty}(\omega) &           \mbox{for} &   i=1,\cdots,N.      \\
	 \sum\limits_{i=1}^{N}\chi_{i}(\cdot-x_{0}^{i})=1
	  & \mbox{on} & \mathbb{T}.
	 \end{array}
	 \right.
\end{equation}
	
	 Therefore,
$${\small
\begin{split}
	 \|v_{n}\|_{L^{2}(K,L^{2}(\mathbb{T}))}
	 &\leq \int_{0}^{T}\left(\psi(t)
	 \left(\sum_{i=1}^{N}\chi_{i}(x-x_{0}^{i})\right) v_{n},
	 v_{n}\right)_{L^{2}(\mathbb{T})} dt
	 =\sum_{i=1}^{N}
	 \left(\psi(t) \chi_{i}(x-x_{0}^{i})
	  v_{n},
	 v_{n}\right)_{L^{2}(\mathbb{T}\times (0,T))}.
\end{split}}$$
	
Thus, it is sufficient  to show that for any $\chi(\cdotp)\in C_{c}^{\infty}(\omega)$ and any $x_{0}\in \mathbb{T}$
there exists a subsequence
$\left\{v_{n_{j}}\right\}_{j\in \mathbb{N}}$  such that
$$\left(\psi(t)\;\chi(x-x_{0})
\;v_{n_{j}},
v_{n_{j}}\right)_{L^{2}(\mathbb{T}\times (0,T))}
\longrightarrow 0,\;\;\text{as}\;j\longrightarrow
\infty.$$

For this, consider $\phi(x)=\chi(x)-\chi(x-x_{0}),$ where
  $\chi\in C^{\infty}_{c}(\omega)$ and $x_{0}\in \mathbb{T}.$
From Lemma~\ref{compact8} there exists
 $\varphi \in C^{\infty}(\mathbb{T})$ such that
 $\partial_{x}\varphi(x)=\chi(x)-\chi(x-x_{0}),$ for all
 $x\in \mathbb{T}.$ Consequently,
$$ {\small
 \begin{split}
 \left(\psi(t) \chi(x-x_{0})
  v_{n},v_{n}\right)_{L^{2}(\mathbb{T}\times (0,T))}
  &=
 \left(\psi(t) \chi(x)v_{n},
 v_{n}\right)_{L^{2}(\mathbb{T}\times (0,T))}
 - \left(\psi(t) \partial_{x}\varphi(x)
 v_{n},v_{n}\right)_{L^{2}(\mathbb{T}\times (0,T))}.
 \end{split}}$$

 From \eqref{compact6}, we have that
$$\left|\left(\psi(t)\;\chi(x)
 \;v_{n},
 v_{n}\right)_{L^{2}(\mathbb{T}\times (0,T))}\right|
\leq \|\psi\|_{C_{c}^{\infty}((0,T))}\;
 \|\chi\|_{C_{c}^{\infty}(\omega)}\;
 \|v_{n}\|^{2}_{L^{2}((0,T);L^{2}(\omega))}
 \longrightarrow 0,$$
as $n\longrightarrow \infty$. So, we only need to show that
there exists a subsequence
$\left\{v_{n_{j}}\right\}_{j\in \mathbb{N}}$  such that
 \begin{equation}\label{compact19}
 \left|\left(\psi(t)\;\partial_{x}\varphi(x)
 \;v_{n_{j}},
 v_{n_{j}}\right)_{L^{2}(\mathbb{T}\times (0,T))}\right|
 \longrightarrow 0,\;\;\text{as}\;j\longrightarrow
 \infty.
 \end{equation}

In what follows, we show \eqref{compact19}. Taking consideration of definition of $D^r$ in \eqref{difop} and passing to the frequency space, it is easy to verify that
 \begin{equation}\label{compact70}
 \begin{split}
 \left(\psi(t)\;\partial_{x}\varphi(x)
 \;v_{n},
 v_{n}\right)_{L^{2}(\mathbb{T}\times (0,T))}
 &= \left(\psi(t)\;\partial_{x}\varphi(x)
(-\partial_{x}^{2})D^{-2} \;v_{n},
 v_{n}\right)_{L^{2}(\mathbb{T}\times (0,T))}\\
 &\quad+
  \left(\psi(t)\;\partial_{x}\varphi(x)
  \;\widehat{v_{n}}(0,t),
 v_{n}\right)_{L^{2}(\mathbb{T}\times (0,T))}.
 \end{split}
 \end{equation}

   First, we prove
  \begin{equation}\label{compact80}
 \lim_{n\longrightarrow \infty}
 \left|\left(\psi(t)\;\partial_{x}\varphi(x)\;
 (-\partial_{x}^{2})D^{-2}v_{n}
 ,\;v_{n}
 \right)_{L^{2}(\mathbb{T}\times (0,T))}\right|
 =0.
 \end{equation}

 In fact, from \eqref{compact1} and \eqref{compact3}, we have
 $Lv_{n}=f_{n},$\; for $n=1,2,3,\cdots.$
  Set $B:=\varphi(x) D^{-2}$ and
 $A:=\psi(t) B.$
 For $\epsilon>0,$ let
 \begin{equation}\label{compact22}
 A_{\epsilon}:=\psi(t)\;B_{\epsilon},
 \end{equation}
be a regularization of $A,$ where
 \begin{equation}\label{compact23}
 B_{\epsilon}:=B\;e^{\epsilon\; \partial_{x}^{2}},\;\;
 \text{with}\;\;e^{\epsilon\; \partial_{x}^{2}}\;\;\text{
 	defined by}\;\;
e^{\epsilon\; \partial_{x}^{2}}v(\cdot)=
\left(e^{-\epsilon\;k^{2}}\widehat{v}(k)
\right)^{\vee}(\cdot).
 \end{equation}

 Define $\displaystyle{ \alpha_{n,\epsilon}:=
  \left([A_{\epsilon},L]v_{n},\;v_{n}
 \right)_{L^{2}(\mathbb{T}\times (0,T))}}.$
Using Lemma \ref{compact24} we obtain that
 \begin{equation}\label{compact31}
 \begin{split}
 	\alpha_{n,\epsilon}&:=
 \left(f_{n},\;A_{\epsilon}^{\ast}v_{n}
 \right)_{L^{2}(\mathbb{T}\times (0,T))}
 +\left(A_{\epsilon}v_{n},\;f_{n}
 \right)_{L^{2}(\mathbb{T}\times (0,T))}.
\end{split}
\end{equation}

We infer from \eqref{compact22}, \eqref{compact23}, \eqref{mult1} and \eqref{mult2} that for any $r\in \mathbb{R},$ and
$0\leq b\leq1,$  there exists a positive constant $C$
(independent of T, if $T\leq 1$) such that
 \begin{equation}\label{compact321}
 	\begin{split}
 		\|A_{\epsilon}^{\ast}v\|_{X_{r-2+2|b|,b}^{T}}
 		&\leq C\;\|v\|_{X_{r,b}^{T}}.
 	\end{split}
 \end{equation}

 Using  $0<b\leq 1,$ we get the immersion
$X_{0,b}^{T}\hookrightarrow X_{4-4b,b}^{T}.$ From \eqref{compact321}, \eqref{compact4} and \eqref{compact5}, we obtain
\begin{equation}\label{compact32}
\begin{split}
 \left|\left(f_{n},\;A_{\epsilon}^{\ast}v_{n}
\right)_{L^{2}(\mathbb{T}\times (0,T))}\right|
&\leq
\|f_{n}\|_{X_{-2+2b,-b}^{T}}\;
\|A_{\epsilon}^{\ast}v_{n}\|_{X_{2-2b,b}^{T}}\\
&\leq C\;
\|f_{n}\|_{X_{-2+2b,-b}^{T}}\;
\|v_{n}\|_{X_{4-4b,b}^{T}}\\
&\leq C\;
 \|f_{n}\|_{X_{-2+2b,-b}^{T}}\;
 \|v_{n}\|_{X_{0,b}^{T}}\\
&\leq C\;
\|f_{n}\|_{X_{-2+2b,-b}^{T}}\longrightarrow 0,\;\;
\text{as}\;n\longrightarrow \infty.
\end{split}
\end{equation}

 Note that, since the positive constant $C$ in \eqref{compact32} does not depend on $\epsilon,$
 one can get
\begin{equation}\label{compact35}
 \lim_{n\longrightarrow \infty}\sup_{0<\epsilon\leq 1}
 \left|\left(f_{n},\;A_{\epsilon}^{\ast}v_{n}
 \right)_{L^{2}(\mathbb{T}\times (0,T))}\right|
=0.
\end{equation}

 Using a similar procedure, we obtain
\begin{equation}\label{compact37}
 \lim_{n\longrightarrow \infty}\sup_{0<\epsilon\leq 1}
 \left|\left(\;A_{\epsilon}v_{n},\;f_{n}
 \right)_{L^{2}(\mathbb{T}\times (0,T))}\right|
 =0.
\end{equation}

 Therefore, \eqref{compact31},  \eqref{compact35} and \eqref{compact37} imply that
 \begin{equation}\label{compact38}
 \lim_{n\longrightarrow \infty}\sup_{0<\epsilon\leq 1}
 \left|\alpha_{n,\epsilon}\right|
 =0.
 \end{equation}

 On the other hand, using that the operator $B_{\epsilon}$ commutes with derivatives in time, we obtain
$$[A_{\epsilon},L]v_{n}=
 -\psi'(t) B_{\epsilon}v_{n}+[A_{\epsilon},
 -\alpha\mathcal{ H}\partial_{x}^{2}]v_{n}+[A_{\epsilon},
 -\partial_{x}^{3}+2\mu\partial_{x}]v_{n}.$$

 Therefore,
 \begin{equation}\label{compact41}
 \begin{split}
 \alpha_{n,\epsilon}
 &= - \left(\psi'(t) B_{\epsilon}v_{n}, v_{n}
 \right)_{L^{2}(\mathbb{T}\times (0,T))}
 + \left([A_{\epsilon},
 -\alpha\mathcal{ H}\partial_{x}^{2}]v_{n}
 , v_{n}
 \right)_{L^{2}(\mathbb{T}\times (0,T))}\\
  &\quad+ \left([A_{\epsilon},
 -\partial_{x}^{3}
 +2\mu\partial_{x}]v_{n}
 , v_{n}
 \right)_{L^{2}(\mathbb{T}\times (0,T))}.
 \end{split}
 \end{equation}

  We infer from \eqref{compact23},  \eqref{mult1} and \eqref{mult2} that for
  any $s\in \mathbb{R},$ and
 $-1\leq b\leq1,$  there exists
 a positive constant $C$
 (independent of T, if $T\leq 1$) which does not depend on $\epsilon,$
 such that
 \begin{equation}\label{compact42}
 \begin{split}
 \|\psi'(t)B_{\epsilon}v\|_{X_{s+2-2|b|,b}^{T}}
 &\leq C\;\|v\|_{X_{s,b}^{T}}.
 \end{split}
 \end{equation}

 From \eqref{compact42}, \eqref{compact4}, \eqref{compact5}, and
 the fact that $0<b\leq 1,$ we obtain
\begin{align*}
 \begin{split}
 \left| \left(\psi'(t) B_{\epsilon}v_{n},\;v_{n}
 \right)_{L^{2}(\mathbb{T}\times (0,T))}\right|
 &\leq
 \|\psi'(t) B_{\epsilon}v_{n}\|_{X_{0,-b}^{T}}\;
 \|v_{n}\|_{X_{0,b}^{T}}\\
 &\leq C
 \|v_{n}\|_{X_{-2+2b,-b}^{T}}\longrightarrow 0,
 \;\;\text{as}\;n\longrightarrow \infty.
 \end{split}
\end{align*}

 Therefore,
 \begin{equation}\label{compact44}
 \lim_{n\longrightarrow \infty}\sup_{0<\epsilon\leq 1}
 \left|\left(\psi'(t)\;B_{\epsilon}v_{n},\;v_{n}
 \right)_{L^{2}(\mathbb{T}\times (0,T))}\right|
 =0.
 \end{equation}

 Also, observe that
 \begin{equation}\label{compact45}
 \begin{split}
[A_{\epsilon},\;
-\alpha\;\mathcal{ H}\partial_{x}^{2}]v_{n}
 &=-\alpha \;\psi(t)\;\varphi D^{-2}
 e^{\epsilon\partial_{x}^{2}}
 \mathcal{ H}\partial_{x}^{2}v_{n}
 +\alpha\;\psi(t)\;\mathcal{ H}\partial_{x}^{2}
 \left(\varphi D^{-2}
 e^{\epsilon\partial_{x}^{2}} v_{n}\right).
 \end{split}
 \end{equation}

 From the Leibniz's rule for derivatives, we obtain
 \begin{equation}\label{compact46}
 \begin{split}
\partial_{x}^{2}
 \left(\varphi D^{-2}
 e^{\epsilon\partial_{x}^{2}}v_{n}\right)
 &=\varphi\;\partial_{x}^{2}D^{-2}
 e^{\epsilon\partial_{x}^{2}}v_{n}
 +2\partial_{x}\varphi\;\partial_{x}D^{-2}
 e^{\epsilon\partial_{x}^{2}}v_{n}
 +\partial_{x}^{2}\varphi\;D^{-2}
 e^{\epsilon\partial_{x}^{2}}v_{n}.
 \end{split}
 \end{equation}

 Substituting \eqref{compact46} into \eqref{compact45} and using Lemma \ref{compact26}, we obtain
  \begin{equation}\label{compact47}
 \begin{split}
 [A_{\epsilon},
 -\alpha \mathcal{ H}\partial_{x}^{2}]v_{n}
  &=\alpha  \psi(t) \left\{-\varphi
 \mathcal{ H}D^{-2}\partial_{x}^{2}
 e^{\epsilon\partial_{x}^{2}}
 v_{n}
 +\mathcal{ H}
 \left(\varphi\;D^{-2}\partial_{x}^{2}
 e^{\epsilon\partial_{x}^{2}}
 v_{n}\right)\right\}\\
 &\quad
 +2\alpha \psi(t) \mathcal{ H}
 \left(\partial_{x}\varphi\;\partial_{x}D^{-2}
 e^{\epsilon\partial_{x}^{2}}v_{n}\right)
 +\alpha \psi(t) \mathcal{ H}
 \left(\partial_{x}^{2}\varphi\;D^{-2}
 e^{\epsilon\partial_{x}^{2}}v_{n}\right).
 \end{split}
 \end{equation}

  Lemma \ref{compact26} implies that,
 \begin{equation}\label{compact48}
 \begin{split}
 -\varphi \mathcal{ H}D^{-2}\partial_{x}^{2}
 e^{\epsilon\partial_{x}^{2}} v_{n}+\mathcal{ H}
 \left(\varphi D^{-2}\partial_{x}^{2}
 e^{\epsilon\partial_{x}^{2}} v_{n}\right)
 &= \varphi
 D^{-1}\partial_{x}D^{-2}\partial_{x}^{2} e^{\epsilon\partial_{x}^{2}} v_{n} -D^{-1}\partial_{x}
 \left(\varphi D^{-2}\partial_{x}^{2}
 e^{\epsilon\partial_{x}^{2}} v_{n}\right)\\
 &= \varphi
 D^{-1}\partial_{x}D^{-2}\partial_{x}^{2}
 e^{\epsilon\partial_{x}^{2}} v_{n} -D^{-1}\left(\varphi\;
 \partial_{x} D^{-2}\partial_{x}^{2}
 e^{\epsilon\partial_{x}^{2}} v_{n}\right)\\
 &\quad-D^{-1}\left(\partial_{x}\varphi
 D^{-2}\partial_{x}^{2} e^{\epsilon\partial_{x}^{2}}
 v_{n}\right)\\
  &= -[D^{-1},\varphi] \partial_{x}D^{-2}\partial_{x}^{2}
 e^{\epsilon\partial_{x}^{2}} v_{n}
 -D^{-1}\left(\partial_{x}\varphi D^{-2}\partial_{x}^{2}
 e^{\epsilon\partial_{x}^{2}} v_{n}\right).
 \end{split}
 \end{equation}

 Substituting \eqref{compact48} into \eqref{compact47}, we have
 \begin{equation}\label{compact49}
 \begin{split}
 [A_{\epsilon},-\alpha \mathcal{ H}\partial_{x}^{2}]v_{n}
 &= \alpha  \psi(t) \left\{-[D^{-1}, \varphi]
 \partial_{x}D^{-2}\partial_{x}^{2}
 e^{\epsilon\partial_{x}^{2}} v_{n}
 -D^{-1}\left(\partial_{x}\varphi D^{-2}\partial_{x}^{2}
 e^{\epsilon\partial_{x}^{2}}v_{n}\right) \right\}\\
 &\quad
 +2\alpha \psi(t) \mathcal{ H}
 \left(\partial_{x}\varphi \partial_{x}D^{-2}
 e^{\epsilon\partial_{x}^{2}}v_{n}\right)
 +\alpha \psi(t) \mathcal{ H}
 \left(\partial_{x}^{2}\varphi D^{-2}
 e^{\epsilon\partial_{x}^{2}}v_{n}\right).
 \end{split}
 \end{equation}

 Therefore,
  \begin{equation}\label{compact50}
 \begin{split}
  \left([A_{\epsilon},-\alpha\mathcal{ H}\partial_{x}^{2}]v_{n}
 ,v_{n} \right)_{L^{2}(\mathbb{T}\times (0,T))}
 &=  -\left(\alpha  \psi(t) [D^{-1},\varphi]
 \partial_{x}D^{-2}\partial_{x}^{2}
 e^{\epsilon\partial_{x}^{2}}v_{n},
 v_{n}\right)_{L^{2}(\mathbb{T}\times (0,T))}\\
 &\quad- \left(\alpha \psi(t) D^{-1}\left(\partial_{x}\varphi\;
   D^{-2}\partial_{x}^{2} e^{\epsilon\partial_{x}^{2}}
   v_{n}\right),v_{n} \right)_{L^{2}(\mathbb{T}\times (0,T))}\\
 &\quad+2\alpha \left(\psi(t) \mathcal{ H}
  \left(\partial_{x}\varphi\; \partial_{x}D^{-2}
 e^{\epsilon\partial_{x}^{2}}v_{n}\right),v_{n}
 \right)_{L^{2}(\mathbb{T}\times (0,T))}\\
&\quad + \alpha \left( \psi(t) \mathcal{ H}
   \left(\partial_{x}^{2}\varphi \;D^{-2}
   e^{\epsilon\partial_{x}^{2}}v_{n}\right), v_{n}
 \right)_{L^{2}(\mathbb{T}\times (0,T))}.
 \end{split}
 \end{equation}

  Applying Cauchy-Schwartz, \eqref{mult1}, \eqref{compact4}, Lemmas \ref{prop3}, \ref{prop2} and using that $0\leq b'\leq b\leq 1$ be given (with $b>0$),  we obtain
  that there exists a positive constant $C=C_{T}$ ($C,$ if $T\leq1$) which does not
  depend on $\epsilon$ such that
 \begin{equation}\label{compact51}
 {\small
 \begin{split}
 \left|\left(\alpha  \psi(t) [D^{-1},\;\varphi]
 \partial_{x}D^{-2}\partial_{x}^{2} e^{\epsilon\partial_{x}^{2}}
 v_{n},\; v_{n} \right)_{L^{2}(\mathbb{T}\times (0,T))}\right|
 &\leq
 C \| [D^{-1},\;\varphi] \partial_{x}D^{-2}\partial_{x}^{2}
 e^{\epsilon\partial_{x}^{2}} v_{n}
\|_{L^{2}(\mathbb{T}\times (0,T))} \| v_{n}\|_{X_{0,b}^{T}}\\
 &\leq
C \left(\sum_{k=-\infty}^{\infty}\int_{0}^{T}
\langle k\rangle^{2(-2)}
\left| \left(\partial_{x}D^{-2}\partial_{x}^{2}
e^{\epsilon\partial_{x}^{2}} v_{n}\right)^{\wedge}(k,\tau)
\right|^{2} \;d\tau\right)^{\frac{1}{2}}\\
&\leq C
\left\|\partial_{x}D^{-2}\partial_{x}^{2}
e^{\epsilon\partial_{x}^{2}}v_{n}\right\|^{\frac{1}{2}}
_{X_{-2,-b'}^{T}} \left\|\partial_{x}D^{-2}\partial_{x}^{2}
e^{\epsilon\partial_{x}^{2}}v_{n}\right\|^{\frac{1}{2}}
_{X_{-2,b'}^{T}}\\
&\leq C
\left\|v_{n}\right\|^{\frac{1}{2}}_{X_{-1,-b'}^{T}}
\left\|v_{n}\right\|^{\frac{1}{2}}_{X_{-1,b'}^{T}}\\
&\leq
C\; \left\|v_{n}\right\|^{\frac{1}{2}}_{X_{-1+2b',-b'}^{T}}
\left\|v_{n}\right\|^{\frac{1}{2}}_{X_{0,b}^{T}}\\
&\leq
C\; \left\|v_{n}\right\|^{\frac{1}{2}}_{X_{-1+2b',-b'}^{T}}
\longrightarrow 0\;\;\text{as}\;n\longrightarrow \infty,	
\end{split}}
\end{equation}
where in the last two inequalities we use  \eqref{compact4}, \eqref{compact5} and  the immersions
$X_{0,b}^{T}\hookrightarrow X_{-1,b'}^{T},$ and
$X_{-1+2b',-b'}^{T}\hookrightarrow X_{-1,-b'}^{T}.$
Note that the loss of regularity in \eqref{compact51}
is too large if one uses the estimates with the same
$b$. Therefore, we have to use the index $b'$ instead.
Consequently,
 \begin{equation}\label{compact56}
\lim_{n\longrightarrow \infty}\sup_{0<\epsilon\leq 1}
\left|\left(
\alpha \;\psi(t)\;
[D^{-1},\;\varphi]
\partial_{x}D^{-2}\partial_{x}^{2}
e^{\epsilon\partial_{x}^{2}}
v_{n},\;v_{n}
\right)_{L^{2}(\mathbb{T}\times (0,T))}\right|
=0.
\end{equation}

From \eqref{mult1}, \eqref{mult2}, \eqref{compact4}, \eqref{compact5}, and Lemma \ref{prop2}, we have that there exists a positive constant
$C=C_{T}$ ($C,$ if $T\leq 1$) which does not depend on $\epsilon$ such that
$${\small
\begin{split}
\left|\left(\alpha \psi(t) D^{-1}\left(\partial_{x}\varphi\;
D^{-2}\partial_{x}^{2} e^{\epsilon\partial_{x}^{2}}
v_{n}\right),\;v_{n}\right)_{L^{2}(\mathbb{T}\times (0,T))}\right|
&\leq
C \left\|D^{-1}\left(\partial_{x}\varphi\;D^{-2}\partial_{x}^{2}
e^{\epsilon\partial_{x}^{2}} v_{n}\right) \right\|_{X_{1-2b',b'}^{T}} \left\|v_{n}\right\|
_{X_{-1+2b',-b'}^{T}}\\
&\leq
C \left\|D^{-2}\partial_{x}^{2}e^{\epsilon\partial_{x}^{2}}
v_{n} \right\|_{X_{-2b'+2|b'|,b'}^{T}}
\left\|v_{n}\right\|_{X_{-1+2b',-b'}^{T}}\\
&\leq
C \left\|v_{n}\right\|_{X_{0,b'}^{T}}
\left\|v_{n}\right\|_{X_{-1+2b',-b'}^{T}}\\
&\leq
C \left\|v_{n}\right\|_{X_{-1+2b',-b'}^{T}}
\longrightarrow 0\;\;\text{as}\;n\longrightarrow \infty.
\end{split}}$$

Therefore,
\begin{equation}\label{compact57}
\lim_{n\longrightarrow \infty}\sup_{0<\epsilon\leq 1}
\left|\left(
\alpha \;\psi(t)\;
D^{-1}\left(\partial_{x}\varphi\;
D^{-2}\partial_{x}^{2}
e^{\epsilon\partial_{x}^{2}}
v_{n}\right),\;v_{n}
\right)_{L^{2}(\mathbb{T}\times (0,T))}\right|
=0.
\end{equation}

Similarly, using that the Hilbert transform  $\mathcal{ H}$ is an
isometry in $L^{2}_{p}(\mathbb{T})$, we have
$${\small
\begin{split}
\left|\left(2\alpha \;\psi(t)\;\mathcal{ H}
\left(\partial_{x}\varphi\;\partial_{x}D^{-2}
e^{\epsilon\partial_{x}^{2}}v_{n}\right),\;v_{n}
\right)_{L^{2}(\mathbb{T}\times (0,T))}\right|
&\leq 2\alpha \left\|\psi(t) \mathcal{ H}
\left(\partial_{x}\varphi\;\partial_{x}D^{-2}
e^{\epsilon\partial_{x}^{2}}v_{n}\right) \right\|_{X_{0,-b'}^{T}}
 \left\|v_{n}\right\|_{X_{0,b'}^{T}}\\
&\leq
C \left\|v_{n}\right\|_{X_{-1+2b',-b'}^{T}}
\longrightarrow 0\;\;\text{as}\;n\longrightarrow \infty,
\end{split}}$$
where the positive constant
$C=C_{T}$ ($C,$ if $T\leq 1$) does not depend on $\epsilon.$
Hence,
\begin{equation}\label{compact59}
\lim_{n\longrightarrow \infty}\sup_{0<\epsilon\leq 1}
\left|\left(
2\alpha \;\psi(t)\;\mathcal{ H}
\left(\partial_{x}\varphi\;\partial_{x}D^{-2}
e^{\epsilon\partial_{x}^{2}}v_{n}\right),\;v_{n}
\right)_{L^{2}(\mathbb{T}\times (0,T))}\right|
=0.
\end{equation}

With similar arguments, we get
\begin{equation}\label{compact61}
\lim_{n\longrightarrow \infty}\sup_{0<\epsilon\leq 1}
\left|\left(\alpha
\psi(t)\;\mathcal{ H}
\left(\partial_{x}^{2}\varphi\;D^{-2}
e^{\epsilon\partial_{x}^{2}}v_{n}\right)
,\;v_{n}
\right)_{L^{2}(\mathbb{T}\times (0,T))}\right|
=0.
\end{equation}

From \eqref{compact50} and \eqref{compact56}-\eqref{compact61}, we obtain that
 \begin{equation}\label{compact62}
 \lim_{n\longrightarrow \infty}\sup_{0<\epsilon\leq 1}
 \left|\left([A_{\epsilon},\;
 -\alpha\;\mathcal{ H}\partial_{x}^{2}]v_{n}
 ,\;v_{n}
 \right)_{L^{2}(\mathbb{T}\times (0,T))}\right|
 =0.
 \end{equation}

 Therefore,  \eqref{compact38}, \eqref{compact41},
 \eqref{compact44} and \eqref{compact62}, imply that
$$\lim_{n\longrightarrow \infty}\sup_{0<\epsilon\leq 1}
 \left|\left([A_{\epsilon},\;
 -\partial_{x}^{3}
 +2\mu\partial_{x}]v_{n}
 ,\;v_{n}
 \right)_{L^{2}(\mathbb{T}\times (0,T))}\right|
 =0.$$

 In particular,
 \begin{equation}\label{compact64}
 \lim_{n\longrightarrow \infty}
 \left([A,\;
 -\partial_{x}^{3}
 +2\mu\partial_{x}]v_{n}
 ,\;v_{n}
 \right)_{L^{2}(\mathbb{T}\times (0,T))}
 =0.
 \end{equation}

Using the Leibniz's rule for derivatives, we note that
\begin{equation}\label{compact66}
\begin{split}
\left([A,\;-\partial_{x}^{3}+2\mu\partial_{x}]v_{n}
,\;v_{n} \right)_{L^{2}(\mathbb{T}\times (0,T))}
&= \left(3\psi(t) \partial_{x}\varphi\;
 \partial_{x}^{2}D^{-2}v_{n}
 ,\;v_{n} \right)_{L^{2}(\mathbb{T}\times (0,T))}\\
 &\quad+ \left(3\psi(t) \partial_{x}^{2}\varphi\;
 \partial_{x}D^{-2}v_{n} ,\;v_{n}
 \right)_{L^{2}(\mathbb{T}\times (0,T))}\\
 &\quad-
  \left(\psi(t)\left(-\partial_{x}^{3}\varphi
  +2\mu \partial_{x}\varphi\right) D^{-2}v_{n}
 ,\;v_{n}\right)_{L^{2}(\mathbb{T}\times (0,T))}.
\end{split}
\end{equation}

Relations \eqref{compact64}, \eqref{compact66} and estimates similar to those in the proof of Proposition 3.5 in
\cite{14} shows \eqref{compact80}. \\

  Second, we prove that there exists a subsequence
  $\left\{v_{n_{j}}\right\}_{j\in \mathbb{N}}$ of
  $\{v_{n}\}_{n\in \mathbb{N}}$ such that
 \begin{equation}\label{compact81}
 \lim_{j\longrightarrow \infty}
 \left|\left(\psi(t)\;\partial_{x}\varphi
 \;\widehat{v_{n_{j}}}(0,t),
 v_{n_{j}}\right)_{L^{2}(\mathbb{T}\times (0,T))}\right|
 =0.
 \end{equation}

 Indeed, observe that
$$\|\widehat{v_{n}}(0,t)\|_{L^{2}(0,T)}
 \leq C \left(\int_{0}^{T}\int_{\mathbb{T}}
 |v_{n}(x,t)|^{2}\;dx\;dt
 \right)^{\frac{1}{2}}
 \leq C
 \|v_{n}\|_{X_{0,0}^{T}}
 \leq C
 \|v_{n}\|_{X_{0,b}^{T}}
 \leq C.$$

 Since $L^{2}(0,T)$ is a reflexive Banach space,
 then from weak compactness, there exists a $w\in L^{2}(0,T)$ and
 a subsequence
 $\left\{v_{n_{j}}\right\}_{j\in \mathbb{N}}$ of
 $\{v_{n}\}_{n\in \mathbb{N}}$ such that
 \begin{equation}\label{compact86}
\widehat{v_{n_{j}}}(0,t)\rightharpoonup w.
 \end{equation}

On the other hand, considering
the function
$$g_{n_{j}}(t):=\int_{\mathbb{T}}
\psi(t)\;\partial_{x}\varphi(x)v_{n_{j}}(x,t)\;dx,
\;\;\text{for}\;t\in (0,T),$$
 we note that
$${
 \begin{split}
 \|g_{n_{j}}(t)\|_{(L^{2}(0,T))'}&:=
 \underset{\|g\|_{L^{2}(0,T)}\leq 1}
 {\sup_{g \in L^{2}(0,T)}}\left|
 \int_{0}^{T}g(t)\;\int_{\mathbb{T}}
 \psi(t)\;\partial_{x}\varphi(x)v_{n_{j}}(x,t)\;dx
 \;dt\right|\\
 &\leq C\|g\|_{L^{2}(0,T)}
 \|\psi(t)\;\partial_{x}\varphi(x)\|^{\frac{1}{2}}
 _{X_{2-2b,b}^{T}}
 \;\|v_{n_{j}}(x,t)\|^{\frac{1}{2}}_{X_{-2+2b,-b}^{T}}
 \longrightarrow 0,
 \end{split}}$$
 as $j\longrightarrow \infty.$
 Thus, $g_{n_{j}}\longrightarrow 0$ (strongly) in
 $\left(L^{2}(0,T)\right)'.$
 Therefore, \eqref{compact86} implies
\begin{equation}\label{compact83}
\begin{split}
\left(\psi(t) \partial_{x}\varphi(x)
\;\widehat{v_{n_{j}}}(0,t),
 v_{n_{j}}\right)_{L^{2}(\mathbb{T}\times (0,T))}
&=
\int_{0}^{T} \widehat{v_{n_{j}}}(0,t)\;g_{n_{j}}(t)\;dt \\
&=\langle\widehat{v_{n_{j}}}(0,t),\;g_{n_{j}}(t) \rangle
\longrightarrow \langle w,\;0 \rangle=0,
\end{split}
\end{equation}
as $j\longrightarrow \infty,$ which proves \eqref{compact81}.
From \eqref{compact70}-\eqref{compact80}, and \eqref{compact81},
we obtain \eqref{compact19}. This completes the proof of the proposition.
\end{proof}

\begin{rem}
	If we assume additionally  that $[v_{n}]=0$
	for all $n\in \mathbb{N}$,  then the result of  in Proposition is \ref{compact2} valid for the original sequence $\left\{v_{n}\right\}_{n\in \mathbb{N}}.$
	\end{rem}

Now, we study  the propagation of regularity for the operator $L$ defined in
\eqref{compact1}.

\begin{prop}[Propagation of Regularity]\label{rglarty}
	Let $T>0,$ $0\leq b\leq 1,$ $r\geq 0$
	and $f\in X_{r,-b}^{T}$ be given.
	Let $v\in X_{r,b}^{T}$ solves
	\begin{equation}\label{rglarty1}
	Lv:=\partial_{t}v-
	\alpha\;\mathcal{ H}\partial_{x}^{2}v
	-\partial_{x}^{3}v
	+2\mu\;\partial_{x}v
	=f.
	\end{equation}

	If there exists a nonempty open set
	$\omega$ of $\mathbb{T}$ such that
\begin{equation}\label{rglarty2}
v \in L^{2}_{loc}((0,T);H^{r+\rho}(\omega)),
\end{equation}	
for some $\rho$ with
\begin{equation}\label{rglarty3}
0<\rho\leq \min\left\{1-b,\;\frac{1}{2}\right\},
\end{equation}
then $v \in L^{2}_{loc}((0,T);H^{r+\rho}(\mathbb{T})).$
\end{prop}

\begin{proof}
	Let $s=r+\rho.$
	Let $\Omega$ be a compact subset of the interval $(0,T),$
	and $\psi(t)\in C^{\infty}_{c}(0,T),$ such that
	$0\leq \psi(t)\leq 1$ and
	$\psi(t)=1$ in $\Omega.$
	Observe that
$${
\begin{split}
\left\|v\right\|^{2}_{L^{2}(\Omega,H^{s}(\mathbb{T}))}
&\leq\int_{0}^{T} \psi(t) \|v\|^{2}_{H^{s}(\mathbb{T})}\;dt\\
&\leq c_{s}\Big(\|v\|_{L^{2}(\mathbb{T}\times (0,T))}^{2}
+\int_{0}^{T}\underset{k \neq 0}{\sum_{k=-\infty}^{+\infty}}
|k|^{2s} \psi(t) |\widehat{v}(k,t)|^{2}\;dt\Big)\\
&= c_{s}\left(\|v\|_{L^{2}(\mathbb{T}\times (0,T))}^{2}
+\left(\psi(t) D^{2s-2}\partial_{x}^{2}v,\;v\right)
_{L^{2}(\mathbb{T}\times(0,T))}\right).
\end{split}}$$	
where the operator $D$ is defined in \eqref{difop}.
Thus, we only need to show that there exists a positive constant
$C$ such that
$$
\big|\left(\psi(t) D^{2s-2}\partial_{x}^{2}v,\;v\right)
_{L^{2}(\mathbb{T}\times(0,T))}\big|
\leq C.$$

Note that, with a similar argument as in the proof of Theorem \ref{compact2} (see \eqref{compact12}), there exist $x_{0}^{i}\in \mathbb{T},$  $i=1,\cdots,N,$ such that we can construct a partition of unity on $\mathbb{T}$ involving functions of the form $\chi^{2}_{i}(\cdotp -x_{0}^{i})$ with  $\chi^{2}_{i}(\cdotp) \in C^{\infty}_{c}(\omega).$
Therefore,
$$
 \big|\left( \psi(t) D^{2s-2}\partial_{x}^{2}v, v\right)
_{L^{2}(\mathbb{T}\times(0,T))}\big|
\leq \sum_{i=1}^{N} \big|\left( \psi(t)
D^{2s-2}\chi^{2}_{i}(x-x_{0}^{i})
\partial_{x}^{2}v, v\right)
_{L^{2}(\mathbb{T}\times(0,T))}\big|.
$$

So, it is sufficient to prove that for any $\chi^{2}(\cdotp)\in C_{c}^{\infty}(\omega)$ and  any $x_{0} \in
\mathbb{T},$ there exists a positive constant $C$ such that
\begin{equation}\label{rglarty10}
	\begin{split}
	\big|\left(
	\psi(t)\;
	D^{2s-2}\chi^{2}(x-x_{0})
	\partial_{x}^{2}v,\;v\right)
	_{L^{2}(\mathbb{T}\times(0,T))}\big|
	&\leq C.
	\end{split}
\end{equation}

In fact, from Lemma \ref{compact8} there exists
	$\varphi \in C^{\infty}(\mathbb{T})$ such that
	$\partial_{x}\varphi(x)=\chi^{2}(x)
	-\chi^{2}(x-x_{0})$ for all
	$x\in \mathbb{T}.$ Consequently,
\begin{equation}\label{rglarty11}
\begin{split}
\left|\left(\psi(t) D^{2s-2}\chi^{2}(x-x_{0})
\partial_{x}^{2}v,\;v\right)_{L^{2}(\mathbb{T}\times(0,T))}\right|
&\leq
\left|\left(\psi(t) D^{2s-2}\chi^{2}(x) \partial_{x}^{2}v, \;v\right)_{L^{2}(\mathbb{T}\times(0,T))}\right|\\
&\quad
+\left|\left(\psi(t) D^{2s-2}\partial_{x}\varphi(x)\;
\partial_{x}^{2}v,\;v\right)_{L^{2}(\mathbb{T}\times(0,T))}\right|.
\end{split}
\end{equation}	
	
Now, we move to bound the RHS of \eqref{rglarty11}. Define $\displaystyle{ v_{n}:=e^{\frac{1}{n}\partial_{x}^{2}}v=E_{n}v=
 \left(e^{-\frac{1}{n}k^{2}}\widehat{v}(k,t)\right)^{\vee}}$
and $f_{n}:=E_{n}f=E_{n}Lv,$ for $n=1,2,3,\cdots.$
Passing to the frequency space, it is easy to verify
 that $E_{n}$ commutes with $L,$ i.e.,
$$ f_{n}:=E_{n}f=E_{n}Lv=LE_{n}v=Lv_{n}.$$

From hypothesis and the definition of $E_{n},$ we obtain that there exists
$C>0$ independent on $n$ such that
\begin{equation}\label{rglarty15}
\begin{split}
\|v_{n}\|_{X_{r,b}^{T}}\leq C,\;\;\;\;\text{and}\;\;\;\;
\|f_{n}\|_{X_{r,-b}^{T}}\leq C, \;\;
\text{for all}\;\;n\geq 1.
\end{split}
\end{equation}	
	
Set $B=D^{2s-2}\varphi,$ and $A=\psi(t)B.$			
We infer from \eqref{mult1} and \eqref{mult2} that for any $r\in \mathbb{R},$ and $0\leq b\leq1,$  there exists
a positive constant $C$ (independent of T, if $T\leq 1$)
such that
\begin{equation}\label{rglarty17}
\begin{split}
	\|Av\|_{X_{r-2|b|-2s+2,b}^{T}}
	&\leq C\;\|v\|_{X_{r,b}^{T}}.
\end{split}
\end{equation}

	With similar calculations as in the proof of the Proposition \ref{compact2} (see
	\eqref{compact41}), we obtain
	\begin{equation}\label{rglarty18}
	\begin{split}
	\left(f_{n},\;A^{\ast}v_{n}
	\right)_{L^{2}(\mathbb{T}\times (0,T))}
	&+\left(Av_{n},\;f_{n}
	\right)_{L^{2}(\mathbb{T}\times (0,T))}\\
	&=
	- \left(\psi'(t)\;Bv_{n},\;v_{n}
	\right)_{L^{2}(\mathbb{T}\times (0,T))}
	+ \left([\alpha\mathcal{ H}\partial_{x}^{2},\;
	A]v_{n}
	,\;v_{n}
	\right)_{L^{2}(\mathbb{T}\times (0,T))}\\
	&\quad+ \left([A,\;
	-\partial_{x}^{3}
	+2\mu\partial_{x}]v_{n}
	,\;v_{n}
	\right)_{L^{2}(\mathbb{T}\times (0,T))}.
	\end{split}
	\end{equation}

 Using that $\rho\leq 1-b,$ \eqref{rglarty15} and \eqref{rglarty17}, we get that there exists $C>0$ independent on $n$ such that
 	\begin{equation}\label{rglarty19}
 	\begin{split}
 	\left|\left(Av_{n},\;f_{n}
 	\right)_{L^{2}(\mathbb{T}\times (0,T))}\right|
 	&\leq
 	\left\|Av_{n}
 	\right\|_{X_{-r,b}^{T}}
 	\left\|f_{n}\right\|_{X_{r,-b}^{T}}
  	\leq C\;
 \left\|v_{n}
 \right\|_{X_{r,b}^{T}}
 \leq C,	
 \end{split}
 \end{equation}
\begin{equation}\label{rglarty20}
\begin{split}
\left|\left(f_{n},\;A^{\ast}v_{n}
\right)_{L^{2}(\mathbb{T}\times (0,T))}\right|
&\leq
\left\|Af_{n}\right\|_{X_{-r,-b}^{T}}
\left\|v_{n}\right\|_{X_{r,b}^{T}}
\leq C\;
\left\|f_{n}
\right\|_{X_{r,-b}^{T}}
\leq C,	
\end{split}
\end{equation}	
and
\begin{equation}\label{rglarty21}
\begin{split}
\left|\left(\psi'(t)\;Bv_{n},v_{n}
\right)_{L^{2}(\mathbb{T}\times (0,T))}\right|
&\leq
\left\|\psi'(t)\;Bv_{n}
\right\|_{X_{-r,-b}^{T}}
\left\|v_{n}\right\|_{X_{r,b}^{T}}
\leq C\;
\left\|v_{n}
\right\|_{X_{r,b}^{T}}
\leq C.
\end{split}
\end{equation}	

Also, using the Leibniz's rule for derivatives and Lemma \ref{compact26}, a simple calculation yields
\begin{equation}\label{rglarty23}
\begin{split}
\mathcal{ H}\partial_{x}^{2}Av_{n}
-A\mathcal{ H}\partial_{x}^{2}v_{n}
&=-\psi(t) D^{2s-2}\left(
[D^{-1},\varphi]\partial_{x}^{3}v_{n}\right)
-\psi(t)D^{2s-3}
\left(\partial_{x}
\varphi\;\partial_{x}^{2}
v_{n}\right)\\
&\quad
+2\mathcal{ H}\psi(t)D^{2s-2}\left(
\partial_{x}\varphi\;\partial_{x}
v_{n}\right)
+\mathcal{ H}\psi(t)D^{2s-2}\left(
\partial_{x}^{2}\varphi\;
v_{n}\right).
\end{split}
\end{equation}

Now, using \eqref{rglarty23} we obtain
\begin{equation}\label{rglarty25}
\begin{split}
\left([\alpha\mathcal{ H}\partial_{x}^{2},\;
A]v_{n}
,\;v_{n}
\right)_{L^{2}(\mathbb{T}\times (0,T))}
&=\alpha\left(\mathcal{ H}\partial_{x}^{2}Av_{n}
-A\mathcal{ H}\partial_{x}^{2}v_{n}
,\;v_{n}
\right)_{L^{2}(\mathbb{T}\times (0,T))}\\
&=
-\alpha\left(
\psi(t)\;D^{2s-2}\left(
[D^{-1},\varphi]\partial_{x}^{3}v_{n}\right)
,\;v_{n}
\right)_{L^{2}(\mathbb{T}\times (0,T))}\\
&\quad
-\alpha\left(
\psi(t)D^{2s-3}
\left(\partial_{x}
\varphi\;\partial_{x}^{2}
v_{n}\right)
,\;v_{n}
\right)_{L^{2}(\mathbb{T}\times (0,T))}\\
&\quad
+2\alpha\left(
\mathcal{ H}\psi(t)D^{2s-2}\left(
\partial_{x}\varphi\;\partial_{x}
v_{n}\right)
,\;v_{n}
\right)_{L^{2}(\mathbb{T}\times (0,T))}\\
&\quad
+\alpha\left(
\mathcal{ H}\psi(t)D^{2s-2}\left(
\partial_{x}^{2}\varphi\;
v_{n}\right)
,\;v_{n}
\right)_{L^{2}(\mathbb{T}\times (0,T))}.
\end{split}
\end{equation}

Using $\rho\leq \frac{1}{2},$ \eqref{mult1}, \eqref{mult2}, Lemma \ref{prop3}, and that   $\mathcal{ H}$ is an
isometry in $H^{-r}_{p}(\mathbb{T}),$ we can obtain  $C>0$ independent on $n$, such that
\begin{equation}\label{rglarty26}
\begin{split}
\alpha\big|&\left( \psi(t) D^{2s-2}\left(
[D^{-1},\varphi]\partial_{x}^{3}v_{n}\right),\;v_{n}
\right)_{L^{2}(\mathbb{T}\times (0,T))}\big|\\
&\qquad\qquad\qquad\qquad\qquad\qquad\leq
\alpha \left\|\psi(t) D^{2s-2}\left(
[D^{-1},\varphi]\partial_{x}^{3}v_{n}\right)
\right\|_{X_{-r,0}^{T}} \left\|v_{n}\right\|_{X_{r,0}^{T}}\\
&\qquad\qquad\qquad\qquad\qquad\qquad\leq
C \left\|[D^{-1},\varphi]\partial_{x}^{3}v_{n}
\right\|_{L^{2}((0,T);H^{-r+2s-2}(\mathbb{T}))} \left\|v_{n}\right\|_{X_{r,b}^{T}}\\
&\qquad\qquad\qquad\qquad\qquad\qquad\leq
C\left\|\partial_{x}^{3}v_{n} \right\|_{L^{2}((0,T);H^{-r+2s-4}(\mathbb{T}))}\\
&\qquad\qquad\qquad\qquad\qquad\qquad\leq
C \left\|v_{n}\right\|_{X_{-r+2s-1,0}^{T}}\\
&\qquad\qquad\qquad\qquad\qquad\qquad\leq
C \left\|v_{n}\right\|_{X_{r,b}^{T}}
\leq
C,
\end{split}
\end{equation}
\begin{equation}\label{rglarty27}
\begin{split}
\alpha\big|&\left(\psi(t)D^{2s-3}\left(\partial_{x}
\varphi\;\partial_{x}^{2} v_{n}\right),\;v_{n}
\right)_{L^{2}(\mathbb{T}\times (0,T))}\big|\\
&\qquad\qquad\qquad\qquad\qquad\leq
\alpha \left\|\psi(t)D^{2s-3}\left(\partial_{x}
\varphi\;\partial_{x}^{2} v_{n}\right)
\right\|_{L^{2}((0,T);H^{-r}(\mathbb{T}))}
\left\|v_{n}\right\|_{L^{2}((0,T);H^{r}(\mathbb{T}))}\\
&\qquad\qquad\qquad\qquad\qquad\leq
C\left\|\partial_{x}\varphi\;\partial_{x}^{2}v_{n}
\right\|_{X_{-r+2s-3,0}^{T}} \left\|v_{n}\right\|_{X_{r,0}^{T}}\\
&\qquad\qquad\qquad\qquad\qquad\leq
C\left\|v_{n}\right\|_{X_{-r+2s-1,0}^{T}}\left\|v_{n}\right\|
_{X_{r,b}^{T}}\\
&\qquad\qquad\qquad\qquad\qquad\leq
C\left\|v_{n}\right\|_{X_{r,b}^{T}}
\leq
C.
\end{split}
\end{equation}
In similar manner, one can get
\begin{equation}\label{rglarty28}
2 \alpha \big|\left(\mathcal{ H}\psi(t)D^{2s-2}\left(
\partial_{x}\varphi\;\partial_{x}v_{n}\right),\;v_{n}
\right)_{L^{2}(\mathbb{T}\times (0,T))}\big|
\leq
C\left\|\partial_{x}\varphi\;\partial_{x}v_{n}
\right\|_{X_{-r+2s-2,0}^{T}}
\left\|v_{n}\right\|_{X_{r,0}^{T}}
\leq
C,
\end{equation}
and 
\begin{equation}\label{rglarty29}
\begin{split}
\alpha \left|\left(\mathcal{ H}\psi(t)D^{2s-2}\left(
\partial_{x}^{2}\varphi\;v_{n}\right),\;v_{n}
\right)_{L^{2}(\mathbb{T}\times (0,T))}\right|
&\leq
C \left\|v_{n}\right\|_{X_{-r+2s-2,0}^{T}}
\left\|v_{n}\right\|_{X_{r,b}^{T}}\\
&\leq
C \left\|v_{n}\right\|_{X_{r,b}^{T}}
\leq
C.
\end{split}
\end{equation}

From \eqref{rglarty25}, \eqref{rglarty26}, \eqref{rglarty27}, \eqref{rglarty28}, and
\eqref{rglarty29}, we infer that
\begin{equation}\label{rglarty30}
\begin{split}
\left|\left([\alpha\mathcal{ H}\partial_{x}^{2},\;
A]v_{n}
,\;v_{n}
\right)_{L^{2}(\mathbb{T}\times (0,T))}\right|
\leq C.
\end{split}
\end{equation}

It follows from \eqref{rglarty18}, \eqref{rglarty19}, \eqref{rglarty20},
\eqref{rglarty21},  and \eqref{rglarty30}  that
\begin{equation}\label{rglarty31}
\begin{split}
\left|\left([A,\;
-\partial_{x}^{3}
+2\mu\partial_{x}]v_{n}
,\;v_{n}
\right)_{L^{2}(\mathbb{T}\times (0,T))}
\right|
\leq C,
\end{split}
\end{equation}
where $C>0$ does not depend on $n.$

 Using the Leibniz's rule, we note that
\begin{equation}\label{rglarty33}
\begin{split}
\left([A,-\partial_{x}^{3}+2\mu\partial_{x}]v_{n},v_{n}
\right)_{L^{2}(\mathbb{T}\times (0,T))}
&= \left(3\psi(t) D^{2s-2}\partial_{x}
\varphi\;\partial_{x}^{2}v_{n},v_{n}
\right)_{L^{2}(\mathbb{T}\times (0,T))}\\
&\quad
+ \left(3\psi(t) D^{2s-2}\partial_{x}^{2}\varphi\;
\partial_{x}v_{n},v_{n}\right)
_{L^{2}(\mathbb{T}\times (0,T))}\\
&\quad-
\left(\psi(t) D^{2s-2}\left(-\partial_{x}^{3}\varphi
+2\mu \partial_{x}\varphi\right)v_{n},v_{n}
\right)_{L^{2}(\mathbb{T}\times (0,T))}.
\end{split}
\end{equation}

It follows from \eqref{rglarty31}, \eqref{rglarty33} and  similar estimates
as those in the proof of Proposition 3.6
in  \cite{14}
that there exists $C>0$ independent of $n$ such that
\begin{equation}\label{rglarty36}
\begin{split}
\left|\left(\psi(t)\;D^{2s-2}\partial_{x}
\varphi(x)\;\partial_{x}^{2}v_{n}
,\;v_{n}
\right)_{L^{2}(\mathbb{T}\times (0,T))}
\right|
&\leq C,\;\;\;\text{for any}\;n\geq 1.
\end{split}
\end{equation}
Therefore, letting $n\longrightarrow \infty$  we get that the second term on the right side of  \eqref{rglarty11} is bounded.

Finally, estimates similar to those in the proof of Proposition 3.6 in  \cite{14} shows that there exists $C>0$ independent of $n$ such that
\begin{equation}\label{rglarty50}
\begin{split}
\left|\left(\psi(t)\;D^{2s-2}\chi^{2}\;
\partial_{x}^{2}v_{n}
,\;D^{s}v_{n}
\right)_{L^{2}(\mathbb{T}\times (0,T))}\right|
&\leq
C.
\end{split}
\end{equation}

Letting $n\longrightarrow \infty$ we get that the first term on the right side of inequality \eqref{rglarty11} is bounded. Thus, \eqref{rglarty11}, \eqref{rglarty36}, and \eqref{rglarty50} imply \eqref{rglarty10} and completes the proof.
\end{proof}

\begin{cor}\label{rglarty51}
	Let $\mu\in \mathbb{R},$ $\alpha>0.$ 
	Let $v\in X^{T}_{0,\frac{1}{2}}$ be a solution
	of
	\begin{equation}\label{rglarty52}
	\begin{split}
\partial_{t}v
-\partial_{x}^{3}v
-\alpha\;\mathcal{H}\partial_{x}^{2}v
+2\mu\;\partial_{x}v
+2v \partial_{x}v=0,\;\;\text{on}\;\;
(0,T),
	\end{split}
	\end{equation}
	with $[u]=0.$
	Assume that $v\in C^{\infty}(\omega\times (0,T)),$
	where $\omega$ is a nonempty open set
	in $\mathbb{T}.$ Then
	$v\in C^{\infty}(\mathbb{T}\times (0,T)).$
	
\end{cor}

\begin{proof}
This result is a direct consequence of
Corollary \ref{Biestimate3},  and an iterated application of Proposition \ref{rglarty}  with $f=-2v \partial_{x}v$ (see
Corollary 3.7 in \cite{14}).
\end{proof}

\subsection{Unique Continuation Property} In this subsection we prove the unique continuation property for the Benjamin equation. We start with a result proved in \cite{Linares Rosier}.
\begin{lem}[{\cite[Lemma 2.9]{Linares Rosier}}]\label{UCP}
	Let $s\in \mathbb{R}$ and let
$h(x)=\sum_{k\geq 0}\widehat{h}(k)e^{ikx}$
	be such that $h\in H^{s}(\mathbb{T})$
	and $h=0$ in $(a,b)\subset \mathbb{T}.$
	Then $h\equiv 0.$
\end{lem}

The following is the main result of this subsection. 

\begin{prop}\label{UCP2}
		Let $\mu\in \mathbb{R},$ $\alpha>0$, 
	$c(t)\in L^{2}(0,T)$
	and $v\in L^{2}((0,T);L^{2}_{0}(\mathbb{T}))$ be a solution
	of
	\begin{equation}\label{UCP3}	
	\left \{
	\begin{array}{l l}
	\partial_{t}v
	-\partial^{3}_{x}v
	-\alpha \mathcal{H}\partial^{2}_{x}v
	+2\mu\partial_{x} v
	+2v\partial_{x} v=0,&  t>0,\;\;\text{on}\; \mathbb{T}
	\times(0,T)\\
	v(x,t)=c(t), & \text{for almost every} \;\;(x,t)\in (a,b)\times(0,T),
	\end{array}
	\right.
	\end{equation}
	for some $T>0$ and
	$0\leq a< b\leq 2\pi.$
	Then $v(x,t)=0$ for almost every
	$(x,t)\in\mathbb{T}\times(0,T).$
\end{prop}

\begin{proof}
	Since 	$v(x,t)=c(t)$ for a.e $(x,t)\in (a,b)\times(0,T),$ we have that
	\begin{equation}\label{UCP4}
	\partial_{x}v(x,t)=
	\partial_{x}^{2}v(x,t)=
	(2v\partial_{x}v)(x,t)=0
	\;\;\text{for a.e}\;\; (x,t)\in (a,b)\times(0,T).
	\end{equation}
	
	Thus, from the first relation in \eqref{UCP3}, we infer that $v$ satisfies
$$\partial_{t}v-\alpha \mathcal{H}\partial^{2}_{x}v=0
	\;\;\text{in}\; (a,b)\times(0,T).$$
	
	Thus, the second relation in \eqref{UCP3} implies that
\begin{equation}\label{UCP6}
\alpha \mathcal{H}\partial^{2}_{x}v
=\partial_{t}v=c'(t)
\;\;\text{in}\; (a,b)\times(0,T).
\end{equation}

Using \eqref{UCP4} and \eqref{UCP6} we deduce that, for almost every $t\in (0,T),$ it holds that
\begin{equation}\label{UCP7}
\begin{split}
&\partial_{x}^{3}v(\cdot,t)\in H^{-3}(\mathbb{T})\\
&\partial_{x}^{3}v(\cdot,t)=0, \;\;\text{in}\; (a,b).\\
 &\mathcal{H}\partial^{3}_{x}v(\cdot,t)=
\partial_{x}\mathcal{H}\partial^{2}_{x}v(\cdot,t)	=0
, \;\;\text{in}\; (a,b).
\end{split}
\end{equation}
	
	Pick a time $t$ as above, and set
	$h(x)=\partial_{x}^{3}v(x,t),\;\;\text{for}\;\; x\in \mathbb{T}.$
	Decompose $h$ as
	$$h(x)=\sum_{k\in \mathbb{Z}}
	\widehat{h}(k)e^{ikx},$$
	where the convergence of the series being in
	$H^{-3}(\mathbb{T}).$ Observe  that
	\begin{equation}
	\begin{split}
	\left(ih-Hh\right)(x)&=\sum_{k\in \mathbb{Z}}(ih-Hh)^{\wedge}(k)e^{ikx}
	=i\sum_{k\in \mathbb{Z}}(1+\text{sgn}(k))\hat{h}(k)e^{ikx}
	=2i\sum_{k>0}\hat{h}(k)e^{ikx}.
	\end{split}
	\end{equation}
	
	From \eqref{UCP7}, we have
	$$0=ih(x)-\mathcal{ H}h(x)=2i\sum_{k> 0}
	\widehat{h}(k)e^{ikx},\;\;\text{for all}\;x\in(a,b).$$

Therefore,
	$\sum\limits_{k> 0}
	\widehat{h}(k)e^{ikx}=0, \;
	\text{in}\;(a,b).$
	From Lemma \ref{UCP}, we obtain that
	$\sum\limits_{k> 0}
	\widehat{h}(k)e^{ikx}=0,$
	in $\mathbb{T}.$
	Since $h$ is real-valued, we also have that
	$\widehat{h}(-k)=\overline{\widehat{h}(k)},$
	for all $k\in \mathbb{Z}.$ Thus,
		$$\sum_{k> 0}
	\widehat{h}(-k)e^{-ikx}=0, \;\;
	\text{in}\;\;\mathbb{T}.$$
	
	Consequently, for a.e. \,$t\in (0,T),$
	$\;\partial_{x}^{3}v(\cdot,t)=0\;\;
	\text{in}\;\mathbb{T}.$
	Then, for a.e. $t\in (0,T),$
$\;\partial_{x}^{2}v(\cdot,t)=c_{1}(t)
\;\;\text{in}\;\mathbb{T}$. But  from
\eqref{UCP4} we obtain that $c_{1}(t)=0,$ for a.e. $t\in (0,T).$ Thus,
for a.e. $t\in (0,T),$
$\;\partial_{x}^{2}v(\cdot,t)=0\;\;
\text{in}\;\mathbb{T}.$ Arguing in a similar way
we obtain that
for a.e. $t\in (0,T),$
$\;\partial_{x}v(\cdot,t)=0\;\;
\text{in}\;\mathbb{T}.$
Thus, for a.e.  $t\in (0,T),$
\begin{equation}\label{UCP8}
v(x,t)=c(t)\;\;
\text{in}\;\mathbb{T}.
\end{equation}

 Substituting \eqref{UCP8} in the first relation of \eqref{UCP3}, we
 obtain that $c'(t)=0$ for a.e. $t\in (0,T).$
 Therefore,
 $v(x,t)=c(t)=cte=:\beta$ a.e. in $\mathbb{T}\times(0,T).$

 Finally, using  $v\in L^{2}_{0}(\mathbb{T}),$
we have $[v]=0$. Consequently,  we obtain 
 $v(x,t)=\beta=0$ a.e. in $\mathbb{T}\times(0,T).$
\end{proof}

\begin{cor}\label{UCP10}
	Let $T>0,$ $\mu\in \mathbb{R},$ and  $\alpha>0$  be given.
	Assume that $\omega$ is a nonempty open set in $\mathbb{T}$
	and let $v\in X^{T}_{0,\frac{1}{2}}$ be a solution
	of
	\begin{equation}\label{UCP9}
		\left \{
		\begin{array}{l l}
		\partial_{t}v
		-\partial^{3}_{x}v
		-\alpha \mathcal{H}\partial^{2}_{x}v
		+2\mu\partial_{x} v
		+2v\partial_{x} v=0,&  t>0,\;\;\text{on}\; \mathbb{T}
		\times(0,T)\\
		v(x,t)=c, & \text{on} \;\;\omega
		\times(0,T),
		\end{array}
		\right.
	\end{equation}
	where $c\in \mathbb{R}$ and $[v]=0.$
	Then $v(x,t)=c=0$ on $\mathbb{T}\times(0,T).$
\end{cor}

\begin{proof}
	Using Corollary \ref{rglarty51}, we infer that
	$v\in C^{\infty}(\mathbb{T}\times(0,T)).$
	It follows that $v(x,t)=c$ on $\mathbb{T}\times(0,T)$
	by Proposition \ref{UCP2} and from the fact that $[v]=0,$ we obtain $c=0.$
\end{proof}

\section{Local control for the  Benjamin equation}\label{section 4}
In this section we obtain the main results regarding  controllability of nonlinear Benjamin equation
\eqref{nonlinearintro3}
in $H_{0}^{s}(\mathbb{T})$ with $s\geq0.$ Observe that the result obtained in this section will imply Theorem \ref{Smalldatacontrol}. We proceed as in  \cite{10},
 by rewriting the system \eqref{nonlinearintro3} in its equivalent integral equation form,
$$u(t)=\displaystyle{U_{\mu}(t)u_{0}+
\int_{0}^{t}U_{\mu}(t-\tau)G(h)(\tau)\;d\tau -\int_{0}^{t}U_{\mu}(t-\tau)(2u\partial_{x}u)(\tau)\;d\tau},$$
where $U_{\mu}(t)$ is the semigroup 
defined in \eqref{semgru2}.
Let us choose $h=\Phi_{\mu}(u_{0},u_{1}+\omega(T,u)),$
where $\Phi_{\mu}$ is the bounded linear operator given in the Remark \ref{sol4}, and define
$$w(T,u):=\int_{0}^{T}U_{\mu}(T-\tau)(2u\partial_{x}u)(\tau)d\tau.
$$

According to Remark \ref{sol4},
the linear system \eqref{linearintro3} is exactly controllable in any positive time $T.$ Therefore, for given  $u_{0},\;u_{1}\in H^{s}_{0}(\mathbb{T}),$  we have
$${\normalsize
\begin{split}
u(t)&=U_{\mu}(t)u_{0}+\int_{0}^{t}U_{\mu} (t-\tau)(G(\Phi_{\mu}(u_{0},u_{1}+w(T,u))))(\tau) d\tau -\int_{0}^{t}U_{\mu}(t-\tau) (2u\partial_{x}u)(\tau) d\tau,
\end{split}}$$
and $u(0)=u_{0},\;\;u(T)=u_{1}.$
This suggests that we should consider the map
\begin{equation}\label{gaop1}
{
\begin{split}
\Gamma(u)&:=U_{\mu}(t)u_{0}+
	\int_{0}^{t}U_{\mu}(t-\tau)(G(\Phi_{\mu}(u_{0},u_{1}+\omega(T,u))))(\tau)\;d\tau -\int_{0}^{t}U_{\mu}(t-\tau)(2u\partial_{x}u)(\tau)\;d\tau,
\end{split}}
\end{equation}
and show that $\Gamma$ is a contraction in an appropriate space. The fixed point $u$ of $\Gamma$ is a mild solution
of IVP \eqref{nonlinearintro3} with $h=\Phi_{\mu}(u_{0},u_{1}+w(T,u))$ and satisfies $u(x,T)=u_{1}(x).$
To complete this argument, we use the Bourgain's space associated to the Benjamin equation and show that  $\Gamma$ is a contraction mapping. This is the content of the following result.

\begin{thm}[Small Data Control]\label{Smalldatacontrol2}
Let $T>0,$ $s\geq0,$ $\mu \in \mathbb{R},$ and $\alpha>0$   be given. Then there exists a $\delta>0$ such that
for any $u_{0},u_{1}\in H_{0}^{s}(\mathbb{T})$ with $[u_{0}]=[u_{1}]=0$  and
$$\|u_{0}\|_{H_{0}^{s}(\mathbb{T})}\leq \delta,\;\;\;\;\;\|u_{1}\|_{H_{0}^{s}(\mathbb{T})}\leq \delta,$$
one can find a control $h\in L^{2}([0,T];H_{0}^{s}(\mathbb{T}))$ such that the IVP
\eqref{nonlinearintro3}
has a unique solution $u\in C([0,T];H_{0}^{s}(0,2\pi))$ satisfying
$$u(x,0)=u_{0}(x), \;\;\;\;u_{1}(x,T)=u_{1}(x),\;\;\;\text{for all}\;x\in \mathbb{T}.$$
\end{thm}

\begin{proof} Let $T>0$ be given. For $s\geq0$ we will show that there exists $M>0$ such that $\Gamma$ defined by \eqref{gaop1} is a contraction on the ball
$B_{M}(0):=\Big\{u\in Z^{T}_{s,\frac{1}{2}}:\;[u]=0,\;\|u\|_{Z^{T}_{s,\frac{1}{2}}}\leq M\Big\}.$\\

In fact, using  Corollary \ref{linstimateZ}, Theorem \ref{intestimateZ3},
and  Corollary \ref{Biestimate3}, we obtain
\begin{equation}\label{princontrac1}
\begin{split}
\| \Gamma (u)\|_{Z^{T}_{s,\frac{1}{2}}}
&\leq c_{1}\displaystyle{ \left(\| u_{0} \|_{H^{s}_{0}(\mathbb{T})}+
\| G(\Phi_{\mu}(u_{0},u_{1}+w(T,u))) \|_{Z^{T}_{s,-\frac{1}{2}}}+
\| u \|^{2}_{Z^{T}_{s,\frac{1}{2}}}\right)},
\end{split}
\end{equation}
where $c_{1}$ is a positive constant depending on $s,T,$ and $\alpha.$  Using that  $G$ (see \eqref{EQ1}) is a bounded operator, the immersion $X^{T}_{s,0}\hookrightarrow X^{T}_{s,-\frac{1}{2}}$ and Remark \ref{sol4}, we get
\begin{equation}\label{princontrac3}
\begin{split}
\| G(\Phi_{\mu}(u_{0},u_{1}+w(T,u))) \|_{X^{T}_{s,-\frac{1}{2}}}
&\leq    c  \| G( \Phi_{\mu}(u_{0},u_{1}+\omega(T,u))) \|_{X^{T}_{s,0}}\\
&\leq    c  \|  \Phi_{\mu}(u_{0},u_{1}+\omega(T,u)) \|_{L^{2}([0,T], H^{s}_{0}(\mathbb{T}))}\\
&\leq c_{2}\left(\| u_{0} \|_{H_{0}^{s}(\mathbb{T})} +
\| u_{1} \|_{H_{0}^{s}(\mathbb{T})} + \| w(T,u) \|_{H_{0}^{s}(\mathbb{T})}\right),
\end{split}
\end{equation}
where $c_{2}>0$ depends on $s,T,$ and $g.$
Using   Proposition \ref{contimbedded}, Theorem \ref{intestimateZ3}
and Corollary \ref{Biestimate3}, we obtain
\begin{equation}\label{princontrac4}
\begin{split}
\| w(T,u)) \|_{H_{0}^{s}(\mathbb{T})}
& \leq \sup_{0\leq t \leq T} \left\|\displaystyle{ \int_{0}^{t}U_{\mu}(t-\tau)(2u\partial_{x}u)(\tau)d\tau} \right\|_{H_{0}^{s}(\mathbb{T})}\\
& \leq c \left\| \displaystyle{ \int_{0}^{t}U_{\mu}(t-\tau)(\partial_{x}(u^{2}))(\tau)d\tau} \right\|_{Z^{T}_{s,\frac{1}{2}}}\\
& \leq c_{3}\| u \|_{Z^{T}_{s,\frac{1}{2}}}^{2},
\end{split}
\end{equation}
where $c_{3}>0$ depends on $s,\,\alpha$ and $T.$
From \eqref{princontrac3} and \eqref{princontrac4},  we have
\begin{equation}\label{part1}
	\begin{split}
\| G(\Phi_{\mu}(u_{0},u_{1}+w(T,u))) \|_{X^{T}_{s,-\frac{1}{2}}}&\leq
c_{2}\left(\| u_{0} \|_{H_{0}^{s}(\mathbb{T})} +
\| u_{1} \|_{H_{0}^{s}(\mathbb{T})} + c_{3} \| u \|_{Z^{T}_{s,\frac{1}{2}}}^{2}\right).
\end{split}
\end{equation}

From Corollary \ref{Yestimate1} with $b=-1$ and $0<\epsilon\leq\frac{1}{2},$ we get
\begin{equation}\label{princontrac5}
\begin{split}
\| G(\Phi_{\mu}(u_{0},u_{1}+w(T,u))) \|_{Y^{T}_{s,-1}}
& \leq c(\epsilon)  \| G(\Phi_{\mu}(u_{0},u_{1}+w(T,u))) \|_{X^{T}_{s,-1+\frac{1}{2}+\epsilon}}\\
& \leq c(\epsilon)   \| G(\Phi_{\mu}(u_{0},u_{1}+w(T,u))) \|_{X^{T}_{s,0}}.
\end{split}
\end{equation}

From inequality \eqref{princontrac5} and the same calculations as above, we obtain
\begin{equation}\label{princontrac6}
\begin{split}
\| G(\Phi_{\mu}(u_{0},u_{1}+w(T,u))) \|_{Y^{T}_{s,-1}}
&\leq  c_{4}(\epsilon)\left(\| u_{0} \|_{H_{0}^{s}(\mathbb{T})} +
\| u_{1} \|_{H_{0}^{s}(\mathbb{T})} +c_{3}\;\| u \|_{Z^{T}_{s,\frac{1}{2}}}^{2} \right),
\end{split}
\end{equation}
where $c_{4}(\epsilon)>0$  depends
on $s,\,\alpha,\,T$ and $g.$ From \eqref{part1} and \eqref{princontrac6}, we infer that
\begin{equation}\label{princontrac7}
\begin{split}
\| G(\Phi_{\mu}(u_{0},u_{1}+w(T,u))) \|_{Z^{T}_{s,-\frac{1}{2}}}
&\leq (c_{2}+c_{4})\left(\| u_{0} \|_{H_{0}^{s}(\mathbb{T})} +
\| u_{1} \|_{H_{0}^{s}(\mathbb{T})}\right)
+ (c_{2}c_{3}+c_{4}c_{3})\| u \|_{Z^{T}_{s,\frac{1}{2}}}^{2}.
\end{split}
\end{equation}

Combining  \eqref{princontrac1} and \eqref{princontrac7}, we obtain that there exists $C=C_{s, \epsilon, \alpha, g, T}>0$ such that
	\begin{equation}\label{contra-9}
\| \Gamma (u)\|_{Z^{T}_{s,\frac{1}{2}}}
 \leq C (\| u_{0} \|_{H^{s}_{0}(\mathbb{T})} +\| u_{1} \|_{H^{s}_{0}(\mathbb{T})})
 + C \| u \|_{Z^{T}_{s,\frac{1}{2}}}^{2}.
 \end{equation}

Choosing $\delta > 0$ and $M > 0$ such that
\begin{equation}\label{eme}
	CM<\frac{1}{4}\;\;\;\;\text{and}
	\;\;\;\;2C\delta+CM^{2}\leq M,
\end{equation}
we obtain from \eqref{contra-9}  that $\| \Gamma (u)\|_{Z^{T}_{\frac{1}{2},s}}\leq M,$ for each
$u\in B_{M}(0),$ provided that $\| u_{0} \|_{H^{s}_{0}(\mathbb{T})}\leq \delta$ and
$\| u_{1} \|_{H^{s}_{0}(\mathbb{T})}\leq \delta.$

Furthermore,  for all $\; u,v \in B_{M}(0),$
 with a similar computations as above, we can obtain
$$
\| \Gamma (u)-\Gamma (v) \|_{Z^{T}_{s,\frac{1}{2}}}
\leq C \| u-v \|_{Z^{T}_{s,\frac{1}{2}}} \| u+v \|_{Z^{T}_{s,\frac{1}{2}}}
\leq \frac{1}{2} \| u-v \|_{Z^{T}_{s,\frac{1}{2}}}.
$$

Thus the map $\Gamma$ is a contraction on $B_{M}(0)$ provided that $\delta$ and $M$ are chosen according
to (\ref{eme}) with $\| u_{0} \|_{H^{s}_{0}(\mathbb{T})}\leq \delta,\;\;\text{and}\;\;
\| u_{1} \|_{H^{s}_{0}(\mathbb{T})}\leq \delta.$
\end{proof}

Note that Theorem \ref{Smalldatacontrol} is a direct
consequence of Theorem \ref{Smalldatacontrol2}.

\section{Stabilization of the Benjamin Equation}\label{section 5}
In this section we study the stabilization problem for the Benjamin equation in $H_{0}^{s}(\mathbb{T}),$ with $s\geq0.$  Consider the IVP,
\begin{equation}\label{atabilizationNL1}
\left \{
\begin{array}{l l}
\partial_{t}u-\partial^{3}_{x}u
-\alpha \mathcal{H}\partial^{2}_{x}u+2\mu\partial_{x} u+2u\partial_{x}u =-K_{\lambda}u,&  t>0,\;\;x\in \mathbb{T}\\
u(x,0)=u_{0}(x), & x\in \mathbb{T},
\end{array}
\right.
\end{equation}
 with $[u]=0,$ $\lambda\geq 0.$
 The feedback control law $K_{\lambda}$ is as defined in \eqref{feedback}. We first check that the system \eqref{atabilizationNL1} is globally well-posed in $H_{0}^{s}(\mathbb{T})$ for any $s\geq0.$
Let $U_{\mu}(t)$
be the group defined in \eqref{semgru2} that describes solution $u$ of the linear IVP associated to \eqref{atabilizationNL1}.
The following estimate is needed.
\begin{lem}\label{LWP7}
	For any $0<\epsilon<1$ there exists a positive constant
	$C(\epsilon)$ such that
	\begin{align*}
	\displaystyle{\left\|
		\int_{0}^{t}U_{\mu}(t-\tau)
		(K_{\lambda}v)(\tau)\;d\tau
		\right\|_{Z^{T}_{s,\frac{1}{2}}}}
&	\leq C(\epsilon)\;T^{1-\epsilon}
	\displaystyle{\|
		v\|_{Z^{T}_{s,\frac{1}{2}}}}.
	\end{align*}
\end{lem}

\begin{proof} The proof follows from Theorem \ref{intestimateZ3},  Propositions \ref{Xestimate1},  \ref{st36}, Corollary \ref{Yestimate1},
with  similar arguments as in the proof of Lemma 4.2 in \cite{14}.
\end{proof}

\begin{thm}\label{LWP}
Let	$s\geq 0,$ $\lambda\geq 0,$  $\alpha>0$  and
$\mu \in \mathbb{R},$ be given.
For any $u_{0}\in H_{0}^{s}(\mathbb{T})$ there exists
a maximal time of existence $T^{\ast}>0$ and a unique
solution $u\in C([0,T^{\ast});H_{0}^{s}(\mathbb{T}))$ to the IVP \eqref{atabilizationNL1}
such that u satisfies the following properties:
\begin{itemize}
	\item [i)]  For every interval $[0,T]\subset [0,T^{\ast}),$
	$\displaystyle{u\in Z^{T}_{s,\frac{1}{2}}
	\cap C([0,T];L_{0}^{2}(\mathbb{T}))}.$
	
		\item [ii)] (Blow-up Alternative) If $T^{\ast}<+\infty,$ then $\displaystyle{\lim_{t\longrightarrow T^{\ast}}
	\|u(t)\|_{H_{0}^{s}(\mathbb{T})}=+\infty}.$
	
	\item [iii)] $u$ depends continuously on the
	initial data in the following sense: If
	$\lim\limits_{n \rightarrow \infty} u_{n,0}= u_{0}$ in $H_{0}^{s}(\mathbb{T})$
	and if $u_{n}$ is the corresponding maximal solution
	of the IVP \eqref{atabilizationNL1} with initial data
	$u_{n,0},$ then $\lim\limits_{n \rightarrow \infty} u_{n}= u$ in
		$Z^{T}_{s,\frac{1}{2}}$,
	for every interval $[0,T]\subset [0,T^{\ast})$. In particular,
$\lim\limits_{n \rightarrow \infty} u_{n}= u$ in
$C([0,T];H_{0}^{s}(\mathbb{T}))$,
for every interval $[0,T]\subset [0,T^{\ast})$.	
\end{itemize}
	Furthermore, denoting  $S(t)u_{0}$ the unique solution $u$
	of the IVP \eqref{atabilizationNL1} corresponding to
	the initial data $u_{0},$  the operator
    $S(t):H_{0}^{s}(\mathbb{T})\longrightarrow
	Z^{T}_{s,\frac{1}{2}},$ defined by
	\begin{equation}\label{LWP3}
     S(t)u_{0}=u
	\end{equation}
	is continuous on every interval $[0,T]
	\subset [0,T^{\ast}).$
\end{thm}
\begin{proof}
 We rewrite the IVP \eqref{atabilizationNL1}
	in its integral form and
	for given $u_{0}\in H_{0}^{s}(\mathbb{T}),$ $0<T<1,$  we define the map
$$\Gamma(v)=\displaystyle{U_{\mu}(t)u_{0}-
		\int_{0}^{t}U_{\mu}(t-\tau)
		(2v\partial_{x}v)(\tau)\;d\tau
		-\int_{0}^{t}U_{\mu}(t-\tau)
		(K_{\lambda}v)(\tau)\;d\tau}.$$
	
Observe that,
$${
\begin{split}
\Gamma(v_{1})-\Gamma(v_{2})
	&=\displaystyle{
	\int_{0}^{t}U_{\mu}(t-\tau)
				[\partial_{x}((v_{2}-v_{1})(v_{2}+v_{1}))](\tau)\;d\tau
				+\int_{0}^{t}U_{\mu}(t-\tau)
				[K_{\lambda}(v_{2}-v_{1})](\tau)\;d\tau}.
\end{split}}$$
	
	It follows then from Corollary \ref{linstimateZ}, Theorem \ref{intestimateZ3}, Corollary \ref{Biestimate3}, and  Lemma \ref{LWP7} that there exists
	some positive constants $C_{1},$ $C_{2},$ $C_{3},$ 
	$0<\theta<\frac{1}{6},$ and $0<\epsilon<1$ such that
\begin{equation}\label{LWP14}
\begin{split}
\|\Gamma(v)\|_{Z^{T}_{s,\frac{1}{2}}}
&\leq
\;\displaystyle{
	C_{1}\;\|u_{0}\|_{H^{s}_{0}(\mathbb{T})}
	+C_{2}\;T^{\theta}\;
	\left\|v\right\|^{2}_{Z^{T}_{s,\frac{1}{2}}}
	+
	C_{3}\;T^{1-\epsilon}\;
	\left\|v\right\|_{Z^{T}_{s,\frac{1}{2}}}},
\end{split}
\end{equation}
\begin{equation}\label{LWP16}
\begin{split}
\|\Gamma(v_{1})-\Gamma(v_{2})\|_{Z^{T}_{s,\frac{1}{2}}}
&\leq
\;\displaystyle{
	C_{2}\;T^{\theta}\;
	\left\|v_{2}-v_{1}\right\|_{Z^{T}_{s,\frac{1}{2}}}
	\left\|v_{2}+v_{1}\right\|_{Z^{T}_{s,\frac{1}{2}}}
	+
	C_{3}\;T^{1-\epsilon}\;
	\left\|v_{2}-v_{1}\right\|_{Z^{T}_{s,\frac{1}{2}}}},
\end{split}
\end{equation}
for any $v,v_{1},v_{2}\in
Z^{T}_{s,\frac{1}{2}}\cap L^{2}([0,T];L^{2}_{0}(\mathbb{T})).$
Pick $M=2\;C_{1}\;\|u_{0}\|_{H_{0}^{s}(\mathbf{T})},$ and
$T>0$ such that,
	\begin{equation}\label{LWP17}
	\begin{split}
	2\;C_{2}\;M\;T^{\theta}+
	C_{3}\;T^{1-\epsilon}&\leq\frac{1}{2}.
	\end{split}
	\end{equation}

Note that, if we choose $0<\epsilon<1$ such that
$0<\theta<1-\epsilon,$ then $T^{1-\epsilon}<T^{\theta}$
and the time $T>0$ can be taken as
\begin{equation}\label{LWP18}
\begin{split}
T=T(\|u_{0}\|_{H_{0}^{s}(\mathbb{T})})
&=\Big(\frac{1}{8C_{1}C_{2}
\|u_{0}\|_{H_{0}^{s}(\mathbb{T})}+2C_{3}}
\Big)^{\frac{1}{\theta}}.
\end{split}
\end{equation}

Therefore, from \eqref{LWP14} and \eqref{LWP16}, we infer that for any $v,v_{1},v_{2}\in B_{M}(0),$
$\displaystyle{\|\Gamma(v)\|_{Z^{T}_{s,\frac{1}{2}}}
	\leq M},$ and
$$\|\Gamma(v_{1})-\Gamma(v_{2})\|_{Z^{T}_{s,\frac{1}{2}}}
\leq
\;\displaystyle{
	\frac{1}{2}
	\left\|v_{2}-v_{1}\right\|_{Z^{T}_{s,\frac{1}{2}}}}.$$

Therefore, $\Gamma$ is a contraction map in the closed ball
$B_{M}(0)$ and its unique fixed point $u$ is the desired solution of \eqref{atabilizationNL1} in the space
 $Z^{T}_{s,\frac{1}{2}}\cap
L^{2}([0,T];L^{2}_{0}(\mathbb{T})).$ It follows from
the Proposition~\ref{contimbedded} that
$u\in C([0,T];H_{0}^{s}(\mathbb{T}))$ with
$${
	\begin{split}
\|u\|_{L^{\infty}([0,T];H_{0}^{s}(\mathbb{T}))}
&\leq C_{4}
\|u\|_{Z^{T}_{s,\frac{1}{2}}}
\leq 2C_{1}C_{4}
\|u_{0}\|_{H_{0}^{s}(\mathbb{T})},
\end{split}}$$
for some $C_{4}>0.$

Now we turn our attention to prove the blow-up alternative. We use the ideas given in \cite{Pavlovic and Tzirakis}.
Define
\begin{equation}\label{LWP22}
\begin{split}
T^{\ast}&:=\sup\big\{T>0\;:\;\exists\,!\; u\in Z^{T}_{s,\frac{1}{2}} \;\text{solution of  \eqref{atabilizationNL1} on [0,T]}\big\}.
\end{split}
\end{equation}

Assume $T^{\ast}<+\infty$ and
$\displaystyle{\lim_{t\longrightarrow T^{\ast}}
\|u(t)\|_{H_{0}^{s}(\mathbb{T})}<+\infty}.$
Then there exist a sequence $t_{j}\longrightarrow T^{\ast^{-}}$
and a positive constant $R$ such that $\|u(t_{j})\|_{H_{0}^{s}(\mathbb{T})}\leq R.$
 In particular, for
$k\in \mathbb{N}$ such that $t_{k}$ is close to $T^{\ast}$
we have that $\|u(t_{k})\|_{H_{0}^{s}(\mathbb{T})}\leq R.$
Now we solve the IVP \eqref{atabilizationNL1} with
inicial data $u(t_{k}).$
Thus from \eqref{LWP18} we obtain that
$T(\|u(t_{k})\|_{H_{0}^{s}(\mathbb{T})})\geq T(R).$ Therefore, we can extend our solution to the interval
$[t_{k},t_{k}+T(R)].$ If we pick  $k$ such that
$t_{k}+T(R)>T^{\ast},$ then it contradicts the definition of $T^{\ast}$ in \eqref{LWP22}.

Finally, we will prove the continuous dependence on the
initial data. Let $u_{0}\in H_{0}^{s}(\mathbb{T})$ and consider
a sequence $u_{n,0}$ in $H_{0}^{s}(\mathbb{T})$ such that
$\displaystyle{\lim_{n\longrightarrow \infty}
u_{n,0}=u_{0}},$
where the limit is taken in the $H_{0}^{s}(\mathbb{T})$ norm.
Let $u$ and $u_{n}$ be the maximal solutions of the IVP
\eqref{atabilizationNL1} in the spaces
$ C([0,T^{\ast});H_{0}^{s}(\mathbb{T}))$ and
$ C([0,T_{n}^{\ast});H_{0}^{s}(\mathbb{T}))$  with initial data $u_{0}$
 and $u_{n,0}$ respectively. For $n$ sufficiently large
 we have that $\|u_{n,0}\|_{H_{0}^{s}(\mathbb{T})}
 \leq 2 \|u_{0}\|_{H_{0}^{s}(\mathbb{T})}$. So, by the local
 theory there exists
$\displaystyle{T_{0}=T_{0}(\|u_{0}\|_{H_{0}^{s}(\mathbb{T})})
<T(\|u_{n,0}\|_{H_{0}^{s}(\mathbb{T})})},$
fulfilling \eqref{LWP17} such that $u$ and $u_{n}$ are defined in $[0,T_{0}]$ for any $n>N_{0}$. Observe that
$${
	\begin{split}
	\Gamma(u)-\Gamma(u_{n})
	&=\displaystyle{
		U_{\mu}(u_{0}-u_{n,0})+
		\int_{0}^{t}U_{\mu}(t-\tau)
		[\partial_{x}(u_{n}^{2}-u^{2})](\tau)\;d\tau}
		+\int_{0}^{t}U_{\mu}(t-\tau)
		[K_{\lambda}(u_{n}-u)](\tau)\;d\tau.
	\end{split}}$$

 With a similar procedure that led to
 \eqref{LWP16}, we get
 $${
 \begin{split}
 \|u-u_{n}\|_{Z^{T_{0}}_{s,\frac{1}{2}}}
 &\leq
 \;\displaystyle{C_{1}
 	\left\|u_{0}-u_{n,0}\right\|_{H_{0}^{s}(\mathbb{T})}}
 +
 \displaystyle{
 	\left(2 C_{2} T_{0}^{\theta} M 
 	+
 	C_{3} T_{0}^{1-\epsilon} \right)
 	\left\|u_{n}-u\right\|_{Z^{T_{0}}_{s,\frac{1}{2}}}}\\
 &\leq
 \displaystyle{C_{1}  
 	\left\|u_{0}-u_{n,0}\right\|_{H_{0}^{s}(\mathbb{T})}}
 +\displaystyle{
 	\frac{1}{2}
 	\left\|u-u_{n}\right\|_{Z^{T_{0}}_{s,\frac{1}{2}}}}.
 \end{split}}$$

  Thus
 \begin{equation}\label{LWP26}
 \|u-u_{n}\|_{Z^{T_{0}}_{s,\frac{1}{2}}}
 \leq
 \;\displaystyle{2C_{1}
 	\left\|u_{0}-u_{n,0}\right\|_{H_{0}^{s}(\mathbb{T})}}.
 \end{equation}

 From \eqref{LWP26} we infer
$$\lim_{n\longrightarrow \infty}
 \|u-u_{n}\|_{L^{\infty}([0,T_{0}],H_{0}^{s}(\mathbb{T}))}
 \leq \;C_{4}\;\lim_{n\longrightarrow \infty}
 \|u-u_{n}\|_{Z^{T_{0}}_{s,\frac{1}{2}}}=0.$$

 Iterating this property to cover the compact subset
 $[0,T]$ of $[0,T^{\ast})$ we finish the proof.
Also, the continuous dependence shows that the operator $S(t)$ defined in
\eqref{LWP3} is continuous. This completes the proof of the theorem.
\end{proof}

Next we study  the global existence of solutions to the IVP
\eqref{atabilizationNL1}.

\begin{thm}\label{GWP}
	Let	$s\geq 0,$ $\lambda\geq 0,$  $\alpha>0,$  and
	$\mu \in \mathbb{R}$ be given.
	For any $u_{0}\in H_{0}^{s}(\mathbb{T})$
 and for any $T>0$ there exists a unique
 solution $u\in Z^{T}_{s,\frac{1}{2}}\cap C([0,T];L_{0}^{2}(\mathbb{T}))$
  to the IVP \eqref{atabilizationNL1}. Furthermore, the following estimates hold
	\begin{equation}\label{GWP1}
	\begin{split}
	\|u\|_{Z^{T}_{s,\frac{1}{2}}}&
	\leq \beta_{T,s}(\|u_{0}\|_{L_{0}^{2}(\mathbb{T})})\;
	\|u_{0}\|_{H_{0}^{s}(\mathbb{T})},
	\end{split}
	\end{equation}
	where $\beta_{T,s}:\mathbb{R}^{+}\longrightarrow 
\mathbb{R}^{+}                                                                                                                                    $ is a nondecresing continuous function
	depending only on $T$ and $s.$ In particular,
	$u\in  C([0,T];H_{0}^{s}(\mathbb{T}))$ and
$\|u\|_{L^{\infty}([0,T];H_{0}^{s}(\mathbb{T}))}
	\leq C_{4} \beta_{T,s}(\|u_{0}\|_{L_{0}^{2}(\mathbb{T})})\;
	\|u_{0}\|_{H_{0}^{s}(\mathbb{T})},$
	where $C_{4}$ is a positive constant.
	Moreover, denoting $S(t)u_{0}$ the unique solution $u$
	of the IVP \eqref{atabilizationNL1} corresponding to
	the initial data $u_{0},$  the operator
	$S(t):H_{0}^{s}(\mathbb{T})\longrightarrow
 Z^{T}_{s,\frac{1}{2}},$ defined by \eqref{LWP3}
	is continuous in the interval $[0,T].$
\end{thm}

\begin{proof}
	First, we assume that $s=0$. Multiplying the equation
\eqref{atabilizationNL1} by $u$ and integrating in space we obtain
  \begin{equation}\label{GWP3}
\begin{split}
\frac{1}{2}\frac{d}{dt'}\left(
\|u(\cdot,t')\|^{2}_{L_{0}^{2}(\mathbb{T})}
\right)
&=-\left(GG^{\ast}L_{\lambda}^{-1}u(\cdot,t'),
u(\cdot,t')\right)_{L_{0}^{2}(\mathbb{T})},\;\;\text{for all}\;t'
\geq 0.
\end{split}
\end{equation}

Now integrating in the time variable in $(0,t),$  and using the properties of operators $G$ and $L_{\lambda}^{-1}$ we infer that
\begin{equation}\label{GWP4}
\begin{split}
\frac{1}{2}
\|u(\cdot,t)\|^{2}_{L_{0}^{2}(\mathbb{T})}
-\frac{1}{2}
\|u_{0}\|^{2}_{L_{0}^{2}(\mathbb{T})}
&=-\int_{0}^{t}
\left(GL_{\lambda}^{-1}u(\cdot,t'),
Gu(\cdot,t')\right)_{L_{0}^{2}(\mathbb{T})}
\;dt'\\
&\leq\int_{0}^{t}\|G\|^{2}\|L_{\lambda}^{-1}\|
\|u(\cdot,t')\|^{2}_{L_{0}^{2}(\mathbb{T})}
\;dt'\\
\end{split}
\end{equation}
for all $t\geq 0.$	 Gronwall's inequality implies that
\begin{equation}\label{GWP6}
\begin{split}
\|u(\cdot,t)\|_{L_{0}^{2}(\mathbb{T})}
&\leq
\|u_{0}\|_{L_{0}^{2}(\mathbb{T})}
e^{C\;t},\;\;
\text{for all}\; t\geq 0,\;\text{where}\; C=\|G\|^{2}\|L_{\lambda}^{-1}\|.
\end{split}
\end{equation}
 From the first line of \eqref{GWP4}, we note that
\begin{equation}\label{GWP7}
\begin{split}
\|u(\cdot,t)\|_{L_{0}^{2}(\mathbb{T})}
&\leq
\|u_{0}\|_{L_{0}^{2}(\mathbb{T})},\;\;
\text{when}\; \lambda=0\; \text{and}\; t\geq0.
\end{split}
\end{equation}

It follows that equation \eqref{atabilizationNL1} is globally well-posed in
$L_{0}^{2}(\mathbb{T})$ by the blow-up alternative.  An standard continuation argument shows the estimate \eqref{GWP1} with $s=0$.

 Next, we suppose $s=3$. 
 In fact, we will prove that for any $T>0$ and any $u_{0}\in H_{0}^{3}(\mathbb{T})\subset L_{0}^{2}(\mathbb{T})$ the solution of the IVP \eqref{atabilizationNL1}  belongs to the space
    $u\in Z^{T}_{3,\frac{1}{2}}\cap
  C([0,T];H_{0}^{3}(\mathbb{T})).$

  For this, let $T>0$ and $u_{0}\in H_{0}^{3}(\mathbb{T})\subset L_{0}^{2}(\mathbb{T}).$ Then, the local solution $u$ of the IVP
 \eqref{atabilizationNL1} belongs to the space
 $u\in Z^{T_{1}}_{0,\frac{1}{2}}\cap
  C([0,T_{1}];L_{0}^{2}(\mathbb{T})),$ where $T_{1}$ is the
  time of local existence given by relation \eqref{LWP18} in Theorem \ref{LWP}, with $s=0.$ Then $u$ satisfies
 \begin{equation}\label{GWP11}
  	\begin{split}
  	 \|u\|_{L^{\infty}([0,T_{1}];L_{0}^{2}(\mathbb{T}))}
  	 &\leq
  	C_{4}\|u\|_{Z^{T_{1}}_{0,\frac{1}{2}}}\leq
  	2C_{4}C_{1}	\|u_{0}\|_{L^{2}_{0}(\mathbb{T})}.		
  	\end{split}
 \end{equation}

   Define
  $v=\partial_{t}u$, so that $[v]=0$ and $v$ satisfies
  \begin{equation}\label{atabilizationNL2}
  \left \{
  \begin{array}{l l}
  \partial_{t}v-\partial^{3}_{x}v
  -\alpha \mathcal{H}\partial^{2}_{x}v+2\mu\partial_{x} v+2\;\partial_{x}(uv) =-K_{\lambda}v,&  0<t\leq T_{1},\;\;x\in \mathbb{T}\\
  v(x,0)=v_{0}=u'''_{0}
  +\alpha \mathcal{H}u''_{0}-2\mu\; u'_{0}
  -2u_{0}u'_{0}-K_{\lambda}u_{0}, & x\in \mathbb{T}.
  \end{array}
  \right.
  \end{equation}

 Note that, applying the Gagliardo-Niremberg's inequality (see the Theorem 3.70 in \cite{Aubin}), we obtain that there exists $c_{1}>0$ such that
 \begin{equation}\label{GWP9}
 	\begin{split}
 	2\|u_{0}\;u_{0}'\|_{L^{2}_{0}(\mathbb{T})}&\leq
 	2\|u_{0}\|_{L^{2}_{0}(\mathbb{T})}
 	\|u_{0}'\|_{L^{\infty}(\mathbb{T})}\\
 	&\leq	2 c_{1}\|u_{0}\|_{L^{2}_{0}(\mathbb{T})}\;
 	\|u'''_{0}\|^{\frac{1}{2}}_{L^{2}_{0}(\mathbb{T})}\;
 	\|u_{0}\|^{\frac{1}{2}}_{L^{2}_{0}(\mathbb{T})}\\
 	&\leq 2 c_{1}\|u_{0}\|_{L^{2}_{0}(\mathbb{T})}\;
 	\|u_{0}\|_{H^{3}_{0}(\mathbb{T})}.			
 	\end{split}
 \end{equation}

 Therefore,  $v_{0}\in L_{0}^{2}(\mathbb{T}),$ with
  \begin{equation}\label{GWP10}
 	\begin{split}
 	\|v_{0}\|_{L^{2}_{0}(\mathbb{T})}
 	&\leq \left(c_{2}+2c_{1}\|u_{0}\|_{L^{2}_{0}(\mathbb{T})}
 	\right) \|u_{0}\|_{H^{3}_{0}(\mathbb{T})},			
 	\end{split}
 \end{equation}
 where $c_{2}>0$ depends on $\alpha, \mu, \lambda, g$ and $\delta.$
  On the other hand,
 considering the map
$$\Gamma(w)=\displaystyle{U_{\mu}(t)v_{0}-
 	2\int_{0}^{t}U_{\mu}(t-\tau)
 	(\partial_{x}(u.w))(\tau)\;d\tau
 	-\int_{0}^{t}U_{\mu}(t-\tau)
 	(K_{\lambda}w)(\tau)\;d\tau},$$
 using Corollary \ref{Biestimate2}, and doing the same calculations as those conducing to \eqref{LWP14},  yield
 $${\normalsize
 	\begin{split}
 	\|\Gamma(w)\|_{Z^{T_{2}}_{0,\frac{1}{2}}}
 	&\leq
 	\;\displaystyle{
 		C_{1}\;\|v_{0}\|_{L^{2}_{0}(\mathbb{T})}
 		+2C_{2}\;T_{2}^{\theta}\;
 		\left\|u\right\|_{Z^{T_{2}}_{0,\frac{1}{2}}}
 		\left\|w\right\|_{Z^{T_{2}}_{0,\frac{1}{2}}}
 		+
 		C_{3}\;T_{2}^{1-\epsilon}\;
 		\left\|w\right\|_{Z^{T_{2}}_{0,\frac{1}{2}}}}\\
 	&\leq
 \;\displaystyle{
 	C_{1}\;\|v_{0}\|_{L^{2}_{0}(\mathbb{T})}
 	+\left(4C_{1}C_{2}\;T_{2}^{\theta}\;
 	\left\|u_{0}\right\|_{L^{2}_{0}(\mathbb{T})}
 	+
 	C_{3}\;T_{2}^{1-\epsilon}\right)
 	\left\|w\right\|_{Z^{T_{2}}_{0,\frac{1}{2}}}}.
 	\end{split}}$$

 Note that,
 $${
 	\begin{split}
 	\Gamma(w_{1})-\Gamma(w_{2})
 	&=\displaystyle{
 		-2\int_{0}^{t}U_{\mu}(t-\tau)
 		[\partial_{x}(u\cdot(w_{1}-w_{2}))](\tau)\;d\tau
 		-\int_{0}^{t}U_{\mu}(t-\tau)
 		[K_{\lambda}(w_{1}-w_{2})](\tau)\;d\tau}.
 	\end{split}}$$

 Thus,
 $${\normalsize
 	\begin{split}
 	\|\Gamma(w_{1}-w_{2})\|_{Z^{T_{2}}_{0,\frac{1}{2}}}
 	&\leq
 	\left(4C_{1}C_{2}\;T_{2}^{\theta}\;
 	\left\|u_{0}\right\|_{L^{2}_{0}(\mathbb{T})}
 	+
 	C_{3}\;T_{2}^{1-\epsilon}\right)
 	\left\|w_{1}-w_{2}\right\|_{Z^{T_{2}}_{0,\frac{1}{2}}},
 	\end{split}}$$
 for any $w,w_{1},w_{2}\in
 Z^{T_{2}}_{0,\frac{1}{2}}\cap L^{2}([0,T_{2}];L^{2}_{0}(\mathbb{T})).$ Therefore, taking
 $T_{2}=T_{1}(\|u_{0}\|_{L_{0}^{2}(\mathbb{T})})$ (note that $T_{2}$ can be taken bigger that $T_{1},$ but we take $T_{2}=T_{1}$ in order to guarantee the existence of solutions for systems \eqref{atabilizationNL1} and \eqref{atabilizationNL2} simultaneously), we obtain that the map $\Gamma$ is a contraction in a closed ball
 $\displaystyle{
 \tilde{B}_{M}(0)=\Big\{w\in Z^{T_{1}}_{0,\frac{1}{2}}:
 [w]=0,\;\|w\|_{Z^{T_{1}}_{0,\frac{1}{2}}}
 \leq M\Big\}}$ with $M= 2C_{1}
\|v_{0}\|_{L_{0}^{2}(\mathbb{T})}.$
  Its unique fixed point $v$ is the desired solution of \eqref{atabilizationNL2} in the space
 $Z^{T_{1}}_{0,\frac{1}{2}}\cap
 L^{2}([0,T_{1}];L^{2}_{0}(\mathbb{T})).$ Thus,
$\|v\|_{Z^{T_{1}}_{0,\frac{1}{2}}}
 \leq 2C_{1}\;
 \|v_{0}\|_{L_{0}^{2}(\mathbb{T})}.$
From Proposition \ref{contimbedded} we infer that
 $v\in C([0,T_{1}];L_{0}^{2}(\mathbb{T}))$ with
 \begin{equation}\label{GWP18}
 \begin{split}
 \|v\|_{L^{\infty}([0,T_{1}];L_{0}^{2}(\mathbb{T}))}&
 \leq C_{4}
 \|v\|_{Z^{T_{1}}_{0,\frac{1}{2}}}
 \leq 2C_{4}C_{1}
 \|v_{0}\|_{L_{0}^{2}(\mathbb{T})}.
 \end{split}
 \end{equation}

 From equation \eqref{atabilizationNL1}, we have\;\;
$\partial^{3}_{x}u=v
 -\alpha \mathcal{H}\partial^{2}_{x}u+2\mu\partial_{x} u+2\;u\;\partial_{x}u +K_{\lambda}u.$

 Consequently,
  \begin{equation}\label{GWP20}
  { \footnotesize
 \|\partial^{3}_{x}u\|_{L_{0}^{2}(\mathbb{T})}\leq\|v
 \|_{L_{0}^{2}(\mathbb{T})}
 +\alpha\| \mathcal{H}\partial^{2}_{x}u\|_{L_{0}^{2}(\mathbb{T})}
 +2|\mu|\|\partial_{x} u\|_{L_{0}^{2}(\mathbb{T})}
 +2\;\|u\;\partial_{x}u\|_{L_{0}^{2}(\mathbb{T})}
  +\|K_{\lambda}u\|_{L_{0}^{2}(\mathbb{T})}}.
 \end{equation}

 The analogous computations as those leading to
 (5.27) and (5.30) in \cite{Vielma and Panthee}, yield  for any $\epsilon>0,$
 	\begin{equation}\label{GWP21}
 	\begin{split}
 	2|\mu|\|\partial_{x}u(\cdot,t)\|_{L^{2}_{0}(\mathbb{T})}
 	&\leq	c_{\mu}
 	\epsilon \|u(\cdot,t)\|_{L^{2}_{0}(\mathbb{T})}
 	+\frac{c_{\mu}}{4\epsilon}\|\partial_{x}^{3}u(\cdot,t)\|
 	_{L^{2}_{0}(\mathbb{T})},			
 	\end{split}
 \end{equation}
 	\begin{equation}\label{GWP22}
 	\begin{split}
 	\alpha\;\|\mathcal{H}\partial^{2}_{x}u(\cdot,t)
 	\|_{L^{2}_{0}(\mathbb{T})}&\leq	
 	c_{\alpha}\;\epsilon^{\frac{3}{2}}
 	\|u(\cdot,t)\|_{L^{2}_{0}(\mathbb{T})}
 	+\frac{c_{\alpha}}{\epsilon^{\frac{1}{2}}}
 	\|\partial_{x}^{3}u(\cdot,t)\|
 	_{L^{2}_{0}(\mathbb{T})}.			
 	\end{split}
 	\end{equation}

  The similar computations as those leading to
 \eqref{GWP9}, but using Cauchy's inequality with $\epsilon>0,$ yield
  \begin{equation}\label{GWP23}
 	\begin{split}
 	2\|u(\cdot,t)\;\partial_{x}u(\cdot,t)
 	\|_{L^{2}_{0}(\mathbb{T})}
 	&\leq
    \;c_{3}\epsilon\|u(\cdot,t)\|^{3}_{L^{2}_{0}(\mathbb{T})}
 	+\frac{c_{3}}{4\epsilon}\|\partial_{x}^{3}u(\cdot,t)
 	\|_{L^{2}_{0}(\mathbb{T})}.
 	\end{split}
 \end{equation}

 We already know that
 \begin{equation}\label{GWP24}
 \begin{split}
 \|K_{\lambda}u(\cdot,t)
 \|_{L^{2}_{0}(\mathbb{T})}
 &\leq
 \;c_{4}\;\|u(\cdot,t)\|_{L^{2}_{0}(\mathbb{T})}.	
 \end{split}
 \end{equation}

From \eqref{GWP11},  \eqref{GWP10} and  \eqref{GWP18}-\eqref{GWP24}, we get that for $0<t\leq T_{1}$
$${
 	\begin{split}
 	\left(1-\frac{c_{\alpha}}{\epsilon^{\frac{1}{2}}}
 	-\frac{(c_{\mu}+c_{3})}{4\epsilon}
 	\right)
 	\|\partial^{3}_{x}u\|_{L_{0}^{2}(\mathbb{T})}
 	&\leq 2 C_{4} C_{1} \left(c_{2}+2c_{1}\|u_{0}\|_{L^{2}_{0}(\mathbb{T})} \right) \|u_{0}\|_{H^{3}_{0}(\mathbb{T})}\\
 	&\quad +\left(c_{\alpha}\epsilon^{\frac{3}{2}}+c_{\mu}\epsilon
 	+c_{4}\right)2C_{4}C_{1}\|u_{0}\|_{H_{0}^{3}(\mathbb{T})}\\
 	&\quad +c_{3}\epsilon 2C_{4}C_{1}\| u_{0}\|^{2}_{L_{0}^{2}(\mathbb{T})}
 	\| u_{0}\|_{H_{0}^{3}(\mathbb{T})}.
 	\end{split}}$$
Taking $\epsilon$ large enough, we can conclude that there exists $C>0$ such that
$${
 	\begin{split}
 	\|\partial^{3}_{x}u\|_{L_{0}^{2}(\mathbb{T})}
 	&\leq C
 	\left(1+\|u_{0}\|_{L_{0}^{2}(\mathbb{T})}
 	+\|u_{0}\|^{2}_{L_{0}^{2}(\mathbb{T})}\right)
 	\|u_{0}\|_{H_{0}^{3}(\mathbb{T})},
 	\end{split}}$$
 Consequently,
  	\begin{equation}\label{GWP27}
 \begin{split}
 \|u\|_{L^{\infty}([0,T_{1}];H_{0}^{3}(\mathbb{T}))}
 &\leq
 \beta_{T_{1},3}(\|u_{0}\|_{L_{0}^{2}(\mathbb{T})})
 \|u_{0}\|_{H_{0}^{3}(\mathbb{T})},
 \end{split}
 \end{equation}
 where $\beta_{T_{1},3}$ is a nondecresing continuous function
 depending only on $T_{1}.$
 
 Next, if we assume that the maximal time of existence
 $T^{\ast}>0$ of the solution $u$ of the IVP \eqref{atabilizationNL1} with initial data in $H_{0}^{3}(\mathbb{T})$ is finite, then
from \eqref{GWP6} we have that
$\displaystyle{\lim_{t\longrightarrow T^{\ast}}
 \|u(t)\|_{L_{0}^{2}(\mathbb{T})}<+\infty}.$
 Then there exist a sequence $t_{j}\longrightarrow T^{\ast^{-}}$
 and a positive constant $R$ such that $\|u(t_{j})\|_{L_{0}^{2}(\mathbb{T})}\leq R.$
 In particular for
 $k\in \mathbb{N}$ such that $t_{k}$ is close to $T^{\ast}$
 we have that $\|u(t_{k})\|_{H_{0}^{3}(\mathbb{T})}\leq R.$
 Now we solve the equation \eqref{atabilizationNL1} with
 initial data $u(t_{k})$ in $H_{0}^{3}(\mathbb{T}).$
 Then from \eqref{LWP18} we obtain that
 $T(\|u(t_{k})\|_{L_{0}^{2}(\mathbb{T})})\geq T(R).$ Therefore, applying  similar arguments as those leading to \eqref{GWP27}  we can extend our solution to the interval
 $[t_{k},t_{k}+T(R)].$ If we pick  $k$ such that
 $t_{k}+T(R)>T^{\ast},$ then it contradicts the definition of $T^{\ast}$ in \eqref{LWP22}.

 Consequently, we can iterate the procedure leading to \eqref{GWP27} in order  to cover the compact interval
 $[0,T],$ thus we obtain that $u\in
 C([0,T]; H_{0}^{3}(\mathbb{T}))$ and satisfies \eqref{GWP1} with $s=3.$
 This completes the  proof of case $s=3$.

 Finally,  observe that a similar result
 can be obtained for $s\in 3\mathbb{N}^{\ast}.$  The global well-posedness for other values of $s,$ follows by nonlinear interpolation (see \cite{Tartar, Bonna Scott}). This completes the proof.
\end{proof}

 Next, we prove a local exponential stability result when appliying the feedback law $f=-K_{\lambda}u.$ For this, we need to observe that the
 system \eqref{atabilizationNL1} can be rewritten as
$$\partial_{t}u=A_{\mu}u-2\;u\;\partial_{x}u-K_{\lambda}u, \;\; t>0,\;\;x\in \mathbb{T},$$
 where	$A_{\mu}=\alpha \mathcal{H}\partial_{x}^{2}+
 \partial_{x}^{3}-2\mu\partial_{x}.$ Let
 $T_{\lambda}(t)=e^{(\alpha \mathcal{H}\partial_{x}^{2}+
 	\partial_{x}^{3}-2\mu\partial_{x}-K_{\lambda})t}$ be the $C_{0}$-semigroup on $H_{0}^{s}(\mathbb{T})$ with infinitesimal generator $A_{\mu}-K_{\lambda}.$ The system
 \eqref{atabilizationNL1} can be rewritten in an equivalent integral form
 \begin{equation}\label{LEF4}
 \begin{split}
 u(t)&=\displaystyle{T_{\lambda}(t)u_{0}-
 	\int_{0}^{t}T_{\lambda}(t-\tau)
 	(2\;u\;\partial_{x}u)(\tau)\;d\tau}.
 \end{split}
 \end{equation}
 Now, we need to extend some estimates for the
 $C_{0}$-semigroup $\left\{T_{\lambda}(t)\right\}.$

 \begin{lem}\label{LEF5}
 	Let  $s\geq0,$ $\lambda\geq 0,$ and $T>0$ be given. Assume that
 	$\mu\in \mathbb{R},$ and $\alpha>0.$  Then, there exists a constant $C>0$, such that
 	\begin{equation} \label{est-lambda1}
 	\left\|
 			T_{\lambda}(t)\phi
 			\right\|_{Z^{T}_{s,\frac{1}{2}}}
 		\leq C\;
 	\|\phi\|_{H_{0}^{s}(\mathbb{T})},\; {\textit{for any}}\; \phi\in H_{0}^{s}(\mathbb{T}),
 		\end{equation}
\begin{equation} \label{est-lambda2}	
 	\Big\|
 			\int_{0}^{t}T_{\lambda}(t-\tau)
 			\partial_{x}(u\cdot v)(\tau)\;d\tau
 			\Big\|_{Z^{T}_{s,\frac{1}{2}}}
 		\leq C\;
 		\|u\|_{Z^{T}_{s,\frac{1}{2}}}
 			\|v\|_{Z^{T}_{s,\frac{1}{2}}},\; {\textit{for any}}\;u,v\in Z^{T}_{s,\frac{1}{2}},
 \end{equation}
 	where the constant $C$ does not depend on $T$
 	if $T\in[0,1].$
 \end{lem}

 \begin{proof}   By definition of $T_{\lambda},$
\begin{equation}\label{LEF9}
 	\begin{split}
 	u(t)&=T_{\lambda}(t)\phi
 	\end{split}
\end{equation}
 	is a solution to
 	\begin{equation}\label{LEF8}
 	\left \{
 	\begin{array}{l l}
 	\partial_{t}u-\partial^{3}_{x}u
 	-\alpha \mathcal{H}\partial^{2}_{x}u+2\mu\partial_{x} u +K_{\lambda}u =0,&  t>0,\;\;x\in \mathbb{T}\\
 	u(x,0)=\phi(x), & x\in \mathbb{T}.
 	\end{array}
 	\right.
 	\end{equation}

 	On the other hand, using the Duhamel's formula, we can write \eqref{LEF8} as
 	\begin{equation}\label{LEF11}
 	\begin{split}
 	u(t)&=\displaystyle{U_{\mu}(t)\phi-
 		\int_{0}^{t}U_{\mu}(t-\tau)
 		[K_{\lambda}u](\tau)\;d\tau}.
 	\end{split}
 	\end{equation}
 	
 	Then, from \eqref{LEF9} and \eqref{LEF11}, we infer that
 	\begin{equation}\label{LEF12}
 	\begin{split}
 	T_{\lambda}(t)\phi&=\displaystyle{U_{\mu}(t)\phi-
 		\int_{0}^{t}U_{\mu}(t-\tau)
 		(K_{\lambda}T_{\lambda}(\tau)\phi)\;d\tau}.
 	\end{split}
 	\end{equation}
 	
 	Therefore, using the Lemma \ref{LWP7}, we obtain
 $${\normalsize
 	\begin{split}
 	\|T_{\lambda}(t)\phi\|_{Z^{T}_{s,\frac{1}{2}}}
 	&\leq
 	\;\displaystyle{
 		C_{1}\;\|\phi\|_{H^{s}_{0}(\mathbb{T})}
 		+C(\epsilon)\;T^{1-\epsilon}\;
 		\left\|T_{\lambda}(\tau)\phi
 		\right\|_{Z^{T}_{s,\frac{1}{2}}}},
 	\end{split}}$$
 	for some $0<\epsilon<1.$
 	Thus, for $T_{0}$ sufficiently small such that $1-C(\epsilon)\;T_{0}^{1-\epsilon}>0$ we get that there exists
 	a positive constant $C=C(T_{0})$ such that
$$\|T_{\lambda}(t)\phi\|_{Z^{T_{0}}_{s,\frac{1}{2}}}
 	\leq
 	\;\displaystyle{
 		C(T_{0})\;\|\phi\|_{H^{s}_{0}(\mathbb{T})}}.$$
 	
 	For $T\geq T_{0},$ the result follows from an easy induction
 	and the fact that
 	\begin{equation}\label{LEF15}
 	\begin{split}
 	\|T_{\lambda}(t)\phi\|_{Z^{T}_{s,\frac{1}{2}}}
 	&\leq
 	\|T_{\lambda}(t)\phi\|_{Z^{[0,T_{0}]}_{s,\frac{1}{2}}}+
 	\|T_{\lambda}(t)\phi\|_{Z^{[T_{0},2T_{0}]}_{s,\frac{1}{2}}}+
 	\cdots+
 	\|T_{\lambda}(t)\phi\|_{Z^{[(k-1)T_{0},T]}_{s,\frac{1}{2}}},
 	\end{split}
 	\end{equation}	
 	for some $k \in \mathbb{Z}.$\\
 	
 	Now, we move to prove \eqref{est-lambda2}. Note that, from \eqref{LEF12} 
 $${
 		\begin{split}
 		\int_{0}^{t}T_{\lambda}(t-\tau)f(\tau)\;d\tau
 		&=\int_{0}^{t}U_{\mu}(t-\tau)f(\tau)\;d\tau
 		-\int_{0}^{t}\int_{0}^{t-\tau}U_{\mu}(t-\tau-s)
 		(K_{\lambda}T_{\lambda}(s)f(\tau))\;ds\;d\tau.
 		\end{split}}$$
 	
 	Performing a change of variable $s=-\tau+\theta$ and  changing the order of integration, we obtain
 	\begin{equation}\label{LEF21}
 		\begin{split}
 		\int_{0}^{t}T_{\lambda}(t-\tau)f(\tau)\;d\tau
 		&=\int_{0}^{t}U_{\mu}(t-\tau)f(\tau)\;d\tau
 		-\int_{0}^{t}U_{\mu}(t-\theta)\int_{0}^{\theta}
 		[K_{\lambda}T_{\lambda}(\theta-\tau)f(\tau)]\;d\tau\;d\theta.
 		\end{split}
 	\end{equation}
 	
 	From  Fubini's theorem, we infer
 	\begin{equation}\label{LEF22}
 	{\footnotesize
 		\begin{split}
 		\int_{0}^{\theta}
 		[K_{\lambda}T_{\lambda}(\theta-\tau)f(\tau)]d\tau
 		&=K_{\lambda}\Big(
 		\int_{0}^{\theta}[T_{\lambda}(\theta-\tau)f(\tau)]
 		d\tau\Big).
 		\end{split}}
 	\end{equation}
 	
 	It follows from \eqref{LEF21} and \eqref{LEF22} that
 	\begin{equation}\label{LEF23}
 		\begin{split}
 		\int_{0}^{t}T_{\lambda}(t-\tau)f(\tau)d\tau
 		&=\int_{0}^{t}U_{\mu}(t-\tau)f(\tau)d\tau
 		-\int_{0}^{t}U_{\mu}(t-\theta)K_{\lambda}\Big(
 		\int_{0}^{\theta}[T_{\lambda}(\theta-\tau)f(\tau)]
 		d\tau\Big)d\theta.
 		\end{split}
 	\end{equation}

 We conclude the proof  by using \eqref{LEF23} and  similar
 arguments as those in Lemma 4.4 \cite{14}.
 \end{proof}

\begin{thm}\label{LEF1}
	Let $0<\lambda'<\lambda$ and $s\geq0$ be given. Assume
	$\mu\in \mathbb{R}$ and $\alpha>0.$  Then there exists $\delta>0$ such that for any
	$u_{0}\in H_{0}^{s}(\mathbb{T})$ with
	$\|u_{0}\|_{H_{0}^{s}(\mathbb{T})}\leq \delta,$ the corresponding solution $u$ of the IVP
	\eqref{atabilizationNL1} satisfies
$$\|u(\cdot,t)
	\|_{H_{0}^{s}(\mathbb{T})}\leq
	C\;e^{-\lambda't}\|u_{0}
	\|_{H_{0}^{s}(\mathbb{T})},\;\;\;\text{for all}\;t\geq 0,$$
	where $C>0$ is a constant that does not depend on $u_{0}.$
\end{thm}

\begin{proof}
Using Theorem \ref{st37} and Lemma \ref{LEF5} we can complete the proof as in \cite[Theorem 4.3]{14}, so we omit the details.
\end{proof}

 The stability result presented  in Theorem \ref{LEF1} is local.
 We will extend it to a global stability result. In order to do that, the following observability inequality is needed.

\begin{prop}\label{Ob1}
	Let  $s\geq0,$ $\lambda\geq 0,$ $\mu\in \mathbb{R},$  $\alpha>0,$ 
	 $T>0,$ and $R_{0}$   be given.
	 Then, there exists a constant $\beta>1$ such that
	 for any $u_{0}\in L^{2}(\mathbb{T})$ satisfying
	 \begin{equation}\label{Ob2}
	 \begin{split}
	 \|u_{0}\|_{L_{0}^{2}(\mathbb{T})}&\leq R_{0},
	 \end{split}
	 \end{equation}
	 the corresponding solution $u$ of the IVP
	 \eqref{atabilizationNL1} satisfies
	 \begin{equation}\label{Ob3}
	 \begin{split}
	 \|u_{0}\|^{2}_{L_{0}^{2}(\mathbb{T})}&\leq
	 \beta \int_{0}^{T}\|Gu\|^{2}_{L_{0}^{2}(\mathbb{T})}
	 (t)\;dt.
	 \end{split}
	 \end{equation}
\end{prop}

\begin{proof}
 We argue by contradiction, assuming that 
	\eqref{Ob3} is not true, then for any $n\geq 1,$ equation
	\eqref{atabilizationNL1} admits a solution
	$u_{n}$ satisfying
	\begin{equation}\label{Ob7}
	\begin{split}
	u_{n}\in Z^{T}_{0,\frac{1}{2}}\cap C([0,T];L^{2}_{0}(\mathbb{T}))&,
	\end{split}
	\end{equation}
	\begin{equation}\label{Ob4}
	\begin{split}
	\|u_{n}(0)\|_{L_{0}^{2}(\mathbb{T})}&\leq R_{0},
	\end{split}
	\end{equation}
	and
	 \begin{equation}\label{Ob5}
\begin{split}
 \int_{0}^{T}\|Gu_{n}\|^{2}_{L_{0}^{2}(\mathbb{T})}
(t)\;dt&<\frac{1}{n}
\|u_{0,n}\|^{2}_{L_{0}^{2}(\mathbb{T})},
\end{split}
\end{equation}	
	where $u_{0,n}=u_{n}(0).$ As
$\alpha_{n}:=\|u_{0,n}\|_{L_{0}^{2}(\mathbb{T})}\leq R_{0,}$
we can extract a subsequence of $\{\alpha_{n}\},$
still denoted by $\{\alpha_{n}\}$ such that
$\lim\limits_{n\rightarrow \infty}\alpha_{n}=\alpha.$
In what follows, we consider two cases $\alpha>0,$ and $\alpha=0$ separately.

	\noindent{\textbf{Case 1.  $\alpha>0:$}}
	 From \eqref{Ob7} and \eqref{Ob4} we obtain that the sequence
$\{\alpha_{n}\}$ is bounded in both spaces
$L^{\infty}([0,T];L^{2}(\mathbb{T}))$ and
 $X^{T}_{0,\frac{1}{2}}.$
Corollary \ref{Biestimate2} implies that the sequence
$\{\partial_{x}(u^{2}_{n})\}$ is bounded
in $X^{T}_{0,-\frac{1}{2}}.$
On the other hand, from Proposition \ref{prop1} we infer that
 $X^{T}_{0,\frac{1}{2}}\hookrightarrow X^{T}_{-1,0}$ is compact.
After, extracting a subsequence of $\{u_{n}\},$
still denoted by $\{u_{n}\},$ we may assume that
 \begin{equation}\label{Ob10}
\begin{split}
u_{n}\rightharpoonup u\;\; \text{in} \;X^{T}_{0,\frac{1}{2}},
&
\end{split}
\end{equation}
\begin{equation}\label{Ob11}
\begin{split}
u_{n}\longrightarrow u\;\; \text{in} \;X^{T}_{-1,0},&
\end{split}
\end{equation}
and
 \begin{equation}\label{Ob12}
\begin{split}
-\partial_{x}(u^{2}_{n})\rightharpoonup f\;\; \text{in} \;X^{T}_{0,-\frac{1}{2}},
&
\end{split}
\end{equation}
where $u\in X^{T}_{0,\frac{1}{2}}$ and
$f\in X^{T}_{0,-\frac{1}{2}}.$ Also, from Theorem \ref{fundaimme},
$X^{T}_{0,\frac{1}{2}}$ is continuously imbedded in
$L^{4}(\mathbb{T}\times [0,T])$ and
$$\|u_{n}^{2}\|_{L^{2}(\mathbb{T}\times [0,T])}
=\|u_{n}\|^{2}_{L^{4}(\mathbb{T}\times [0,T])}
\leq C\;\|u_{n}\|^{2}_{X^{T}_{0,\frac{1}{3}}}
\leq C\;\|u_{n}\|^{2}_{X^{T}_{0,\frac{1}{2}}}.$$

Thus, $u_{n}^{2}$ is bounded in $L^{2}(\mathbb{T}\times [0,T])$
and it follows that
$$\|\partial_{x}(u_{n}^{2})\|_{L^{2}( [0,T];
	H^{-1}(\mathbb{T}))}
=\|\partial_{x}(u_{n}^{2})\|_{X^{T}_{-1,0}}
\leq
\|u_{n}^{2}\|_{L^{2}(\mathbb{T}\times [0,T])}.$$
Therefore, $\partial_{x}(u_{n}^{2})$ is bounded in
$L^{2}( [0,T];H^{-1}(\mathbb{T}))=
X^{T}_{-1,0}.$
Applying interpolation between  $X^{T}_{0,-\frac{1}{2}},$
and $X^{T}_{-1,0}$ (see proof of
Theorem \ref{multi23}) we conclude that
$\partial_{x}(u_{n}^{2})$ is bounded in
$X^{T}_{-\theta,-\frac{(1-\theta)}{2}}=
X^{T}_{-\theta,-\frac{1}{2}+\frac{\theta}{2}},$ for $0<\theta<1.$
Since $X^{T}_{-\theta,-\frac{1}{2}+\frac{\theta}{2}}$ is compactly imbedded in $X^{T}_{-1,-\frac{1}{2}},$ for $0<\theta<1,$
one can extract a subsequence of  $\{u_{n}\},$
still denoted by $\{u_{n}\},$ such that
\begin{equation}\label{Ob15}
\begin{split}
-\partial_{x}(u^{2}_{n})\longrightarrow f\;\; \text{in} \;X^{T}_{-1,-\frac{1}{2}}.
&
\end{split}
\end{equation}	
	
It follows from \eqref{Ob5} that
\begin{equation}\label{Ob16}
\begin{split}
\int_{0}^{T}\|Gu_{n}\|^{2}_{L_{0}^{2}(\mathbb{T})}
(t)\;dt&\longrightarrow
\int_{0}^{T}\|Gu\|^{2}_{L_{0}^{2}(\mathbb{T})}
(t)\;dt	=0,
\end{split}
\end{equation}
	which implies
	$u(x,t)=c(t)=\int_{0}^{T}g(y)u(y,t)\;dy,$ on $\omega\times (0,T)$ (see \eqref{EQ1}).
Thus, passing to the limit in equation \eqref{atabilizationNL1}
	verified by $u_{n},$ we obtain
	\begin{equation}\label{Ob17}
	\left \{
	\begin{array}{l l}
	\partial_{t}u-\partial^{3}_{x}u
	-\alpha \mathcal{H}\partial^{2}_{x}u+2\mu\partial_{x} u=f,&  \text{on}\; \mathbb{T}\times (0,T)\\
	u(x,t)=c(t), & \text{on}\; \omega\times (0,T).
	\end{array}
	\right.
	\end{equation}
	
	Let
	\begin{equation}\label{Ob21}
	w_{n}:=u_{n}-u\;\;\;\text{and}\;\;
	f_{n}:=-\partial_{x}(u_{n}^{2})-f-K_{0}u_{n}.
	\end{equation}

	Note first that, \eqref{Ob16} implies
	\begin{equation}\label{Ob18}
	{\footnotesize
	\begin{split}
	\int_{0}^{T}\|Gw_{n}\|^{2}_{L_{0}^{2}(\mathbb{T})}
	(t)\;dt
	&=\int_{0}^{T}\|Gu_{n}\|^{2}_{L_{0}^{2}(\mathbb{T})}
	(t)\;dt
	+\int_{0}^{T}\|Gu\|^{2}_{L_{0}^{2}(\mathbb{T})}
	(t)\;dt\\
	&\quad-2\int_{0}^{T}\left(Gu_{n},Gu \right)_{L_{0}^{2}(\mathbb{T})}	(t)\;dt
 \longrightarrow 0,\;\;\text{as}\;n\longrightarrow \infty.
	\end{split}}
	\end{equation}
	
 From \eqref{Ob10}, we obtain that
$w_{n}\rightharpoonup 0\;\; \text{in} \;X^{T}_{0,\frac{1}{2}}.$
Furthermore, \eqref{atabilizationNL1}, and
  \eqref{Ob17}-\eqref{Ob21} imply that $w_{n}$
 satisfies
	\begin{equation}\label{Ob22}
\partial_{t}w_{n}-\partial^{3}_{x}w_{n}
-\alpha \mathcal{H}\partial^{2}_{x}w_{n}+2\mu\partial_{x}w_{n} =f_{n},\;\;  \text{on}\;\; \mathbb{T}\times (0,T).
\end{equation}
	
Observe that
	\begin{equation}\label{Ob28}
		\begin{split}
		\int_{0}^{T}\int_{\mathbb{T}}
		|Gw_{n}|^{2}\;dx\;dt
		&=
		\int_{0}^{T}\int_{\mathbb{T}}
g^{2}(x)\;w_{n}^{2}(x,t)\;dx\;dt\\
&\quad-2	\int_{0}^{T}\left(\int_{\mathbb{T}}
g(y)\;w_{n}(y,t)\;dy\right)
\left(\int_{\mathbb{T}}
g^{2}(x)\;w_{n}(x,t)\;dx\right)\;dt\\	
		&\quad +
		\int_{0}^{T}\left(\int_{\mathbb{T}}
		g(y)\;w_{n}(y,t)\;dy\right)^{2}
		\left(\int_{\mathbb{T}}
		g^{2}(x)\;dx\right)\;dt.
		\end{split}
	\end{equation}
	
At this point we need the following Lemma.

\begin{lem}\label{Ob20}
Let $\{w_{n}\}_{n\geq 1}$ be a sequence of solutions of equation \eqref{Ob22} and $g$ defined in \eqref{gcondition}.	If $w_{n}\rightharpoonup 0\;\; \text{in} \;X^{T}_{0,\frac{1}{2}},$ then there exists a subsequence of
	$\displaystyle{c_{n}(t):=
		\int_{\mathbb{T}}g(y) w_{n}(y,t) dy,}$ $t\in (0,T),$ still denoted by 	$\{c_{n}\}_{n\geq 1},$ such that
	$c_{n}\longrightarrow 0\;\; \text{in} \;L^{2}(0,T)$ as $n\longrightarrow \infty.$
\end{lem}

\begin{proof}
From hypotheses we infer that
	$w_{n}\rightharpoonup 0\;\; \text{in} \;X^{T}_{0,0}.$
	So, $\{w_{n}\}_{n\geq 1}$ is bounded in
	$X^{T}_{0,0}.$
	From \eqref{Ob22}, \eqref{Ob21} and integration by parts, we have
	\begin{equation}\label{Ob24}
		\begin{split}
		\frac{d}{dt}c_{n}(t)
		&=\int_{\mathbb{T}}g(y)
		\left(\partial^{3}_{y}w_{n}
		+\alpha \mathcal{H}\partial^{2}_{y}w_{n}-2\mu\partial_{y}w_{n} +f_{n}\right)(y)dy\\
		&=\int_{\mathbb{T}}w_{n}(y)
		\left(-\partial^{3}_{y}g
		-\alpha \mathcal{H}\partial^{2}_{y}g
		+2\mu\partial_{y}g\right)(y) +g(y)\left(-\partial_{y}(u_{n}^{2})-f\right)(y)
		-GGg(y)u_{n}(y)dy.
		\end{split}
	\end{equation}
	Integrating \eqref{Ob24} in (0,T) and  using Cauchy-Schwarz inequality on space variable, we obtain
	$${
		\begin{split}
		\left\|	\frac{d}{dt}c_{n}(t)
		\right\|_{L^{2}(0,T)}
		&\leq C\left(
		\int_{0}^{T}\left(
		\int_{\mathbb{T}}|w_{n}|^{2}dy\right)
		\left(
		\int_{\mathbb{T}}|
		-\partial^{3}_{y}g(y)
		-\alpha \mathcal{H}\partial^{2}_{y}g(y)
		+2\mu\partial_{y}g(y)|^{2}dy\right)\;dt
		\right)^{\frac{1}{2}}\\
		&\quad
		+\left|\int_{0}^{T}\int_{\mathbb{T}}\partial_{y}g(y) u_{n}^{2}(y,t)dy\;dt\right|
         +\left|\left(g(y), f(y,t)\right)_{L^{2}(\mathbb{T}\times (0,T))}\right| \\
		&\quad +C\left(
		\int_{0}^{T}\left(
		\int_{\mathbb{T}}|GGg(y)|^{2}dy\right)
		\left(
		\int_{\mathbb{T}}|
		u_{n}|^{2}dy\right)\;dt
		\right)^{\frac{1}{2}}.
		\end{split}}$$
	
	Further simplifying and using \eqref{Ob10} and \eqref{Ob12}, we obtain
	\begin{equation}\label{Ob26}
	{
		\begin{split}
		\left\|	\frac{d}{dt}c_{n}(t)
		\right\|_{L^{2}(0,T)}
		&\leq\;C_{\alpha,\mu}\|g\|_{H^{3}(\mathbb{T})}
	\left\|w_{n}\right\|_{X_{0,0}^{T}}+\|\partial_{y}g\|_{L^{\infty}(\mathbb{T})}
        \int_{0}^{T}\int_{\mathbb{T}}|u_{n}(y,t)|^{2}dy\;dt\\
		&\quad +C_{T,\alpha,\mu}\|g\|_{H^{\frac{3}{2}}(\mathbb{T})}
		\|f\|_{X_{0,-\frac{1}{2}}^{T}}
		+\|GGg\|_{L^{2}(\mathbb{T})}
		\|u_{n}\|_{X^{T}_{0,0}}\\
		&< +\infty.
		\end{split}}
	\end{equation}
	
	On the other hand,
	\begin{equation}\label{Ob27}
		\begin{split}
		\left\|	c_{n}(t)
		\right\|_{L^{2}(0,T)}
		&\leq\;
		\|g\|_{L^{2}(\mathbb{T})}\|w_{n}\|_{X^{T}_{0,0}}
		< +\infty.
		\end{split}
	\end{equation}
	
	From  \eqref{Ob26} and \eqref{Ob27}, we have that
	$c_{n}(t)\in H^{1}(0,T).$ The Rellich's Theorem, and the fact that
	$w_{n}\rightharpoonup 0\;\; \text{in} \;X^{T}_{0,0}$
	imply the desired conclusion.
\end{proof}

We continue with the proof of Proposition \ref{Ob1}. 	From \eqref{Ob18}, \eqref{Ob28} and Lemma \ref{Ob20}, we deduce that
$$\int_{0}^{T}\int_{\mathbb{T}}
		g^{2}(x)\;w_{n}^{2}(x,t)\;dx\;dt
		\longrightarrow 0.$$

Hence,
$$\|w_{n}\|_{L^{2}((0,T);L^{2}(\widetilde{\omega}))}
\leq
\frac{4}{\|g\|^{2}_{L^{\infty}(\mathbb{T})}}
\int_{0}^{T}\int_{\widetilde{\omega}}
g^{2}(y)\;w_{n}^{2}(x,t)\;dx\;dt\longrightarrow 0,\;\;\text{as}\;\;n\longrightarrow +\infty,$$
where $\widetilde{\omega}:=\left\{ x\in \mathbb{T}:
g(x)>\frac{\|g\|_{L^{\infty}(\mathbb{T})}}{2}\right\}.$
It follows from \eqref{Ob15} and \eqref{Ob16} that
$$\|f_{n}\|_{X^{T}_{-1,-\frac{1}{2}}}
\leq \|-\partial_{x}(u_{n}^{2})-f\|_{X^{T}_{-1,-\frac{1}{2}}}
+ C\|Gu_{n}\|_{X_{0,0}^{T}}\longrightarrow 0,\;\;\text{as}\;\;n\longrightarrow +\infty.$$
	
	Applying the propagation of compactness property  (see Proposition \ref{compact2} with
	$b=\frac{1}{2},$ and $b'=0$), we obtain that
	\begin{equation}\label{Ob32}
	\begin{split}
	\|w_{n}\|_{L^{2}_{loc}((0,T);L^{2}(\mathbb{T}))}&
\longrightarrow 0,\;\;\text{as}\;\;n\longrightarrow +\infty.
	\end{split}
	\end{equation}
	
Hence, $u_{n}^{2}\longrightarrow u^{2} \;\; \text{in} \;
	L^{1}_{loc}((0,T);L^{1}(\mathbb{T})).$
	Consequently,
$\partial_{x}(u_{n}^{2})\longrightarrow
	\partial_{x}( u^{2})$ in the distributional sense.
Thus, $f=-\partial_{x}(u^{2})$ and
	$u\in X^{T}_{0,\frac{1}{2}}$ satisfy
		\begin{equation}\label{Ob37}
	\left \{
	\begin{array}{l l}
	\partial_{t}u-\partial^{3}_{x}u
	-\alpha \mathcal{H}\partial^{2}_{x}u+2\mu\partial_{x} u
	+\partial_{x}(u^{2})
	=0,&  \text{on}\; \mathbb{T}\times (0,T)\\
	u(x,t)=c(t), & \text{on}\; \omega\times (0,T).
	\end{array}
	\right.
	\end{equation}

	From the unique continuation property  (see Proposition \ref{UCP2}) we get that $u=0.$
	Now \eqref{Ob32} implies that
	 $u_{n}\longrightarrow 0$ in
	$L^{2}_{loc}((0,T);L^{2}(\mathbb{T})).$
	Hence, there exists a time $t_{0}\in[0,T]$ such that
	$u_{n}(t_{0})\longrightarrow 0$ in
	$L^{2}(\mathbb{T}).$ From \eqref{GWP3} with $\lambda=0,$ we get
$$\|u_{n}(0)\|^{2}_{L_{0}^{2}(\mathbb{T})}
	=
	\|u_{n}(t_{0})\|^{2}_{L_{0}^{2}(\mathbb{T})}
	+\int_{0}^{t_{0}}
	\|Gu_{n}\|_{L_{0}^{2}(\mathbb{T})}
	\;dt'\longrightarrow 0\;\;\text{as}\;\;n\longrightarrow
	+\infty.$$
	which contradicts the assumption that $\alpha>0,$

\noindent{\textbf{Case 2.  $\alpha=0:$}}
Using the unique continuation property for the linearized Benjamin equation which can be proved in a similar way as
Proposition \ref{UCP2} and a similar argument as those in
Proposition 4.6 \cite{14} we arrive at a contradiction.
This completes the proof.
\end{proof}

\begin{thm}\label{GEF1}
Let $\lambda=0$ in \eqref{atabilizationNL1}. Assume
$\mu\in \mathbb{R}$ and $\alpha>0.$  Then there exists $k>0$ such that for any $R_{0}> 0,$ there exists a constant $C>0$
independent of $u_{0},$ such that
for any $u_{0}\in L_{0}^{2}(\mathbb{T})$ with
$\|u_{0}\|_{L_{0}^{2}(\mathbb{T})}\leq R_{0},$
 the corresponding solution $u$ of the IVP
\eqref{atabilizationNL1} (with $\lambda=0$) satisfies
\begin{equation}\label{GEF2}
\begin{split}
\|u(\cdot,t)
\|_{L_{0}^{2}(\mathbb{T})}&\leq
C\;e^{-kt}\|u_{0}
\|_{L_{0}^{2}(\mathbb{T})},\;\;\;\text{for all}\;t\geq 0.
\end{split}
\end{equation}
\end{thm}

\begin{proof}
	This theorem is a direct consequence of the observability inequality \eqref{Ob3}
(see \cite[Theorem 4.5]{14}). Observe that the constant $k$ is independent
of $R_{0}.$
\end{proof}

Now, we prove that the solution $u$ of
\eqref{atabilizationNL1} (with $\lambda=0$) decays
exponentially in any space $H_{0}^{s}(\mathbb{T}).$
For this, we need an exponential stability result
for the linearized system
\begin{equation}\label{GEF13}
\left \{
\begin{array}{l l}
\partial_{t}w-\partial^{3}_{x}w
-\alpha \mathcal{H}\partial^{2}_{x}w+2\mu\partial_{x} w+2\;\partial_{x}(aw) =-K_{0}w,&  x\in \mathbb{T}\;\;t>0,\\
w(x,0)=w_{0}(x), & x\in \mathbb{T},
\end{array}
\right.
\end{equation}
where $a\in Z_{s,\frac{1}{2}}^{T}\cap
L^{2}([0,T];L_{0}^{2}(\mathbb{T}))$ is a given function.
This is done in the following two Lemmas.

\begin{lem}\label{GEF14}
	Let	$s\geq 0,$   $\alpha>0,$  and
	$\mu \in \mathbb{R}$ be given.
	Assume $a\in Z_{s,\frac{1}{2}}^{T}\cap
	L^{2}([0,T];L_{0}^{2}(\mathbb{T}))$ for all
	$T>0,$ and that there exists $T'>0$ such that
	\begin{equation}\label{GEF16}
	\begin{split}
	\sup_{n\geq 1}\|a\|_{Z^{[nT',(n+1)T']}_{s,\frac{1}{2}}}&
	\leq \beta.
	\end{split}
	\end{equation}
	Then for any $w_{0}\in H_{0}^{s}(\mathbb{T})$
	and  any $T>0$ there exists a unique
	solution $w\in Z^{T}_{s,\frac{1}{2}}\cap C([0,T];H_{0}^{s}(\mathbb{T}))$
	of the IVP \eqref{GEF13}. Furthermore, the following estimate holds
	\begin{equation}\label{GEF15}
	\begin{split}
	\|w\|_{Z^{T}_{s,\frac{1}{2}}}&
	\leq \upsilon(\|a\|_{Z_{s,\frac{1}{2}}^{T}})\;
	\|w_{0}\|_{H_{0}^{s}(\mathbb{T})},
	\end{split}
	\end{equation}
	where $\upsilon:\mathbb{R}^{+}\longrightarrow
	\mathbb{R}^{+}$ is a nondecreasing continuous function.
	
	Moreover, denote by $S(t)w_{0}$ the unique solution $u$
	of equation \eqref{GEF13} corresponding to
	the initial data $w_{0}.$ Then  the operator
	$S(t):H_{0}^{s}(\mathbb{T})\longrightarrow
	Z^{T}_{s,\frac{1}{2}},$ defined by
$S(t)w_{0}=w$
	is continuous in the interval $[0,T].$
\end{lem}

\begin{proof}
	We  first establish the existence and uniqueness of a solution $w\in 	Z^{T}_{s,\frac{1}{2}}\cap L^{2}([0,T];L^{2}_{0}(\mathbb{T}))$ of \eqref{GEF13} for $0<T\leq 1$ small enough and then show that $T$ can be taken arbitrarily large. Let us rewrite system \eqref{GEF13} in its integral form
	and for given initial datas $w_{0}, w_{1} \in H_{0}^{s}(\mathbb{T})$
	we	define the map
$$\Gamma(v_{j})=\displaystyle{U_{\mu}(t)w_{j}-
		\int_{0}^{t}U_{\mu}(t-\tau)
		(2\partial_{x}(a\;v_{j}))(\tau)\;d\tau
		-\int_{0}^{t}U_{\mu}(t-\tau)
		(K_{0}v_{j})(\tau)\;d\tau},$$
	where $j=0,1$ and $U_{\mu}(t)=e^{(\partial_{x}^{3}
		+\alpha\mathcal{H}\partial_{x}^{2}-\mu \partial_{x})t}.$ Assume
	$0< T\leq T'.$
	Then, calculations similar to those in Theorem
	\ref{LWP} yield
		\begin{equation}\label{GEF22}
		{
		\begin{split}
		\|\Gamma(v_{1})-\Gamma(v_{2})\|_{Z^{T}_{s,\frac{1}{2}}}
		&\leq
		C_{1}\|w_{0}-w_{1}\|_{H_{0}^{s}}+
		\;\displaystyle{
			2C_{2}\;T^{\theta}\;
			\left\|a\right\|_{Z^{T}_{s,\frac{1}{2}}}			
			\left\|v_{1}-v_{2}\right\|_{Z^{T}_{s,\frac{1}{2}}}
			+
			C_{3}\;T^{1-\epsilon}\;
			\left\|v_{2}-v_{1}\right\|_{Z^{T}_{s,\frac{1}{2}}}},
		\end{split}}
		\end{equation}
	for any $a,v_{1},v_{2}\in
	Z^{T}_{s,\frac{1}{2}}\cap L^{2}([0,T];L^{2}_{0}(\mathbb{T})).$
	Choosing $w_{1}=0,$\; $M=2 C_{1} \|w_{0}\|_{H_{0}^{s}(\mathbb{T})},$ and
	$T>0$ such that \;$\displaystyle{2 C_{2} T^{\theta} \beta
	+C_{3} T^{1-\epsilon}\leq\frac{1}{2}},$
	we obtain that the map $\Gamma$ is a contraction in a closed ball
	$B_{M}(0)$ with $M=2C_{1}\|w_{0}\|_{H_{0}^{s}(\mathbb{T})}.$
	Its unique fixed point $w$ is the desired solution of \eqref{GEF13} in
	$Z^{T}_{s,\frac{1}{2}}\cap
	L^{2}([0,T];L^{2}_{0}(\mathbb{T})).$
	Note that, the time of existence, can be taken as
	\begin{equation}\label{GEF24}
	\begin{split}
	T=\min\left\{\frac{1}{2}, \;T',\;
	\left(\frac{1}{2C_{2}\;\beta+C_{3}}
	\right)^{\frac{1}{\theta}}\right\}.
	\end{split}
	\end{equation}	
	Furthermore, \eqref{GEF22} shows that the solution depends
	continuously on the initial data and satisfies \eqref{GEF15}.
	
	Now, we prove the global existence of the solution.
	Let $T^{\ast}$ be the maximal time of existence of the solution $w$ of the IVP
	\eqref{GEF13} satisfying \eqref{GEF15} with initial data
	$w_{0}\in H_{0}^{s}(\mathbb{T}).$
	If $T^{\ast}<\infty,$ then  from \eqref{GEF16}, we have 
$${
	\begin{split}
	\lim_{r\longrightarrow T^{\ast^{-}}}
	\|w(r)\|_{H_{0}^{s}(\mathbb{T})}
	&\leq
	\lim_{r\longrightarrow T^{\ast^{-}}}
	 \upsilon(\|a\|_{Z_{s,\frac{1}{2}}^{[0,r]}})
	\|w_{0}\|_{H_{0}^{s}(\mathbb{T})}
	\leq
	\upsilon((n_{0}+1)\beta)
	\|w_{0}\|_{H_{0}^{s}(\mathbb{T})}
	<+\infty,
	\end{split}}$$
	for some $n_{0} \in \mathbb{Z}.$
	Following a similar argument as in the proof of the blow-up alternative in Theorem \ref{LWP} we finish the proof.
\end{proof}

\begin{lem}\label{GEF18}
	Let	$s\geq 0,$   $\alpha>0,$  and
	$\mu \in \mathbb{R}$ be given.
	Assume $a\in Z_{s,\frac{1}{2}}^{T}\cap
	L^{2}([0,T];L_{0}^{2}(\mathbb{T}))$ for all
	$T>0.$ Then for any $k'\in (0,k)$ there exists $T>0,$ and
	$\beta>0$ such that if
	\begin{equation}\label{GEF19}
	\begin{split}
	\sup_{n\geq 1}\|a\|_{Z^{[nT,(n+1)T]}_{s,\frac{1}{2}}}&
	\leq \beta,
	\end{split}
	\end{equation}
	the solution 	of the IVP \eqref{GEF13} satisfies
	\begin{equation}\label{GEF20}
	\begin{split}
	\|w(\cdot,t)\|_{H^{s}_{p}(\mathbb{T})}&
	\leq C\;e^{-k't}
	\|w_{0}\|_{H_{p}^{s}(\mathbb{T})},\;\;\;
	\text{for all}\;\;t\geq 0,
	\end{split}
	\end{equation}
	where $C>0$ is a constant that does not depend on $w_{0}.$
\end{lem}

\begin{proof}
	From Lemma \ref{GEF14} we have that for any $T>0$
	the IVP \eqref{GEF13} admits a unique solution
	$w\in Z^{T}_{s,\frac{1}{2}}\cap C([0,T];H_{0}^{s}(\mathbb{T}))$
	and
	\begin{equation}\label{GEF30}
	\begin{split}
	\|w\|_{Z^{T}_{s,\frac{1}{2}}}&
	\leq \upsilon(\|a\|_{Z_{s,\frac{1}{2}}^{T}})\;
	\|w_{0}\|_{H_{0}^{s}(\mathbb{T})},
	\end{split}
	\end{equation}
	where $\upsilon:\mathbb{R}^{+}\longrightarrow
	\mathbb{R}^{+}$ is a nondecreasing continuous function.
	Rewrite \eqref{GEF13} in its integral form
$$w(t)=\displaystyle{T_{0}(t)w_{0}-
		\int_{0}^{t}T_{0}(t-\tau)
		(2\;\partial_{x}(a\cdot w))(\tau)\;d\tau},$$
	where
	$T_{0}(t)=e^{(\alpha \mathcal{H}\partial_{x}^{2}+
		\partial_{x}^{3}-2\mu\partial_{x}-K_{0})t}$ is the $C_{0}$-semigroup on $H_{0}^{s}(\mathbb{T})$ with infinitesimal generator $A_{\mu}-K_{0}.$ For any $T>0,$ we infer from Corollary \ref{st35},
	Lemma \ref{LEF5}, and \eqref{GEF30} that
	\begin{equation}\label{GEF32}
	\begin{split}
	\|w(\cdot,T)\|_{H^{s}_{0}(\mathbb{T})}
	&\leq
	C_{1}e^{-kT}	\|w_{0}\|_{H^{s}_{0}(\mathbb{T})}
	+C_{2}\;\left\|a
	\right\|_{Z^{T}_{s,\frac{1}{2}}}
	\upsilon(\|a\|_{Z_{s,\frac{1}{2}}^{T}})\;
	\|w_{0}\|_{H_{0}^{s}(\mathbb{T})},
	\end{split}
	\end{equation}
	where $C_{1}>0$ is independent of $T$ and $C_{2}>0$ may depend
	on $T.$ Let
$$	y_{n}:=w(\cdot,nT),\;\;\;\text{for}\;\;n=1,2,3,...$$
	
	Using the semigroup property, we have
$${
		\begin{split}
		y_{n+1}&=w(\cdot,nT+T)\\
		&=\displaystyle{T_{0}(T)\left[T_{0}(nT)w_{0}
			-\int_{0}^{nT}T_{0}(nT-\tau)
			(2 \partial_{x}(a\cdot w))(\tau) d\tau\right]}\\
			&\qquad
		-T_{0}(nT)\int_{0}^{T}T_{0}(T-(\theta+nT))
		(2 \partial_{x}(a\cdot w))(\theta+nT) d\theta.
		\end{split}}$$
	
	Defining\;
$\displaystyle{I_{2}:=
	-T_{0}(nT)\int_{0}^{T}T_{0}(T-(\theta+nT))
	(2\;\partial_{x}(a\cdot w))(\theta+nT)\;d\theta},$
	we observe that
$${
	\begin{split}
	\|I_{2}\|_{H^{s}_{0}(\mathbb{T})}
	&\leq\displaystyle{
		c\;\left\|\int_{0}^{t}
		T_{0}(t-(\theta+nT))
		(2\;\partial_{x}(a\cdot w))(\theta+nT)\;d\theta
		\right\|_{Z^{T}_{s,\frac{1}{2}}}}\\
	&\leq\displaystyle{
		C_{2}\;\left\|a(\theta+nT)
		\right\|_{Z^{T}_{s,\frac{1}{2}}}
		\left\|w(\theta+nT)
		\right\|_{Z^{T}_{s,\frac{1}{2}}}  }\\
	&\leq\displaystyle{
		C_{2}\;\left\|a
		\right\|_{Z^{[nT,(n+1)T]}_{s,\frac{1}{2}}}
		\left\|w
		\right\|_{Z^{[nT,(n+1)T]}_{s,\frac{1}{2}}}  }\\
	&\leq\displaystyle{
		C_{2}\;\left\|a\right\|_{Z^{[nT,(n+1)T]}_{s,\frac{1}{2}}}\;
		\upsilon\left(\|a\|_{Z_{s,\frac{1}{2}}^{[nT,(n+1)T]}}\right)\;
		\|w(\cdot,nT)\|_{H_{0}^{s}(\mathbb{T})}  }\\
	&\leq\displaystyle{
		C_{2}\beta\;\upsilon(\beta)\;
		\|y_{n}\|_{H_{0}^{s}(\mathbb{T})}	}.
	\end{split}}$$
	
	Therefore,
	\begin{equation}\label{GEF36}
	\begin{split}
	\|y_{n+1}\|_{H_{0}^{s}(\mathbb{T})}
	&\leq\displaystyle{
		\|T_{0}(T)y_{n}\|_{H_{0}^{s}(\mathbb{T})}
		+C_{2}\beta\;\upsilon(\beta)\;
		\|y_{n}\|_{H_{0}^{s}(\mathbb{T})}	}\\
	&\leq\displaystyle{
		\left(C_{1}e^{-kT}
		+C_{2}\beta\;\upsilon(\beta)\right)
		\|y_{n}\|_{H_{0}^{s}(\mathbb{T})}	},\;\;\;
	\text{for}\;\;n\geq 1.
	\end{split}
	\end{equation}
	
	Choosing $T>0$ sufficiently large and $\beta$ small enough
	so that
	\begin{equation}\label{GEF37}
	\begin{split}
	C_{1}e^{-kT}
	+C_{2}\beta\;\upsilon(\beta)
	&=e^{-k't},
	\end{split}
	\end{equation}
	 we get from \eqref{GEF36} that\;
$\|y_{n+1}\|_{H_{0}^{s}(\mathbb{T})}
	\leq\displaystyle{
		e^{-k'T}
		\|y_{n}\|_{H_{0}^{s}(\mathbb{T})}	},\;
	\text{for}\;\;n\geq 1,$
	as long as \eqref{GEF19} holds. Thus, $w$ satisfies
	\eqref{GEF20} 
	and the proof is complete.
\end{proof}

\begin{thm}\label{GEF6}
	Let $\lambda=0$ in \eqref{atabilizationNL1}. Assume
	$\mu\in \mathbb{R}$, $\alpha>0$ 
	 and $k_{0}>0$ be the infimum of the numbers
	$\gamma, \;k$ given respectively in  Theorem \ref{st35} and  Theorem \ref{GEF1}. Let $s\geq 0$ and let
$k'\in (0,k_{0}]$ be given. Then	there exists a nondecreasing
continuous function
$\alpha_{s,k'}:\mathbb{R}^{+}\longrightarrow
\mathbb{R}^{+}$
such that for any $u_{0}\in H_{0}^{s}(\mathbb{T}),$
	the corresponding solution $u$ of the IVP
	\eqref{atabilizationNL1} (with $\lambda=0$) satisfies
	\begin{equation}\label{GEF7}
	\begin{split}
	\|u(\cdot,t)
	\|_{H_{0}^{s}(\mathbb{T})}&\leq
	\alpha_{s,k'}(\|u_{0}\|_{L_{0}^{2}(\mathbb{T})})\;
	e^{-k't}\|u_{0}
	\|_{H_{0}^{s}(\mathbb{T})},\;\;\;\text{for all}\;\;t\geq 0.
	\end{split}
	\end{equation}
\end{thm}

\begin{proof}
	Note that, in Theorem \ref{GEF1} we already established  \eqref{GEF7}
	for $s=0$ (with $k'=k$). Now, consider the case $s=3.$ Let
	$R_{0}>0$ be any number  and
	$u_{0}\in H_{0}^{3}(\mathbb{T})$ with
	$\|u_{0}\|_{L_{0}^{2}(\mathbb{T})}\leq
R_{0}.$
	Let $u$ be the solution of  \eqref{atabilizationNL1} with $\lambda=0$ and initial data $u_{0},$ 
	and define $v=\partial_{t}u.$ Then
	$v$ solves
	 \begin{equation}\label{GEF8}
	\left \{
	\begin{array}{l l}
	\partial_{t}v-\partial^{3}_{x}v
	-\alpha \mathcal{H}\partial^{2}_{x}v+2\mu\partial_{x} v+2\;\partial_{x}(uv) =-K_{0}v,&  x\in \mathbb{T}\;\;t>0,\\
	v(x,0)=v_{0}=u'''_{0}
	+\alpha \mathcal{H}u''_{0}-2\mu\; u'_{0}
	-2u_{0}u'_{0}-K_{0}u_{0}\in L_{0}^{2}(\mathbb{T}), & x\in \mathbb{T}.
	\end{array}
	\right.
	\end{equation}
	
	From  \eqref{GWP1} and
	 \eqref{GEF2} we infer that for any
	$T>0$ there exists aconstant $C>0$  that depends  only on $R_{0}$
	and $T$ such that
$${\normalsize
	\begin{split}
	\|u
	\|_{Z_{0,\frac{1}{2}}^{[t,t+T]}}&\leq
	C_{R_{0},T}\;e^{-kt}\|u_{0}
	\|_{L_{0}^{2}(\mathbb{T})},\;\;\;\text{for all}\;t\geq 0.
	\end{split}}$$

Therefore, for any $\epsilon>0,$ there exists $t^{\ast}>0$
such that if $t\geq t^{\ast},$ we get
\begin{equation}\label{GEF12}
\begin{split}
\|u\|_{Z_{0,\frac{1}{2}}^{[t,t+T]}}&\leq
\epsilon.
\end{split}
\end{equation}	

One can choose $\epsilon<\beta$ in \eqref{GEF12}, where $\beta$ is given by \eqref{GEF37}, and 
use the exponential stability result (Lemma \ref{GEF18})
for the  linearized system
\begin{equation}\label{GEF131}
\left \{
\begin{array}{l l}
\partial_{t}w-\partial^{3}_{x}w
-\alpha \mathcal{H}\partial^{2}_{x}w+2\mu\partial_{x} w+2\;\partial_{x}(uw) =-K_{0}w,&  x\in \mathbb{T}\;\;t>t^{\ast},\\
w(x,0)=w_{0}(x)=v(t^{\ast}), & x\in \mathbb{T},
\end{array}
\right.
\end{equation}
where $u\in Z_{s,\frac{1}{2}}^{T}\cap
L^{2}([0,T];L_{0}^{2}(\mathbb{T}))$ is a given function and $w=v(t-t^{\ast})$, to infer that
$${\normalsize
\begin{split}
\|v(\cdot, t-t^{\ast})\|_{L_{0}^{2}(\mathbb{T})}
&\leq\displaystyle{
	C\;e^{-k'(t-t^{\ast})}
	\|v(\cdot,t^{\ast})\|_{L_{0}^{2}(\mathbb{T})}	},\;\;\;
\text{for all}\;\;t\geq t^{\ast}.
\end{split}}$$
This means
$${\normalsize
\begin{split}
\|v(\cdot, t)\|_{L_{0}^{2}(\mathbb{T})}
&\leq\displaystyle{
	C\;e^{-k't}
	\|v_{0}\|_{L_{0}^{2}(\mathbb{T})}	},\;\;\;
\text{for any}\;\;t\geq 0,
\end{split}}$$
where $C>0$ depends only on $R_{0}$. It follows from Theorem
\ref{GEF1} and the equation
 \begin{equation}\label{GEF44}
\partial^{3}_{x}u=v
-\alpha \mathcal{H}\partial^{2}_{x}u+2\mu\partial_{x} u+2\;u\;\partial_{x}u +K_{0}u
\end{equation}
that
$$\|u(\cdot, t)\|_{H_{0}^{3}(\mathbb{T})}
\leq\displaystyle{
	C\;e^{-k't}
	\|u_{0}\|_{H_{0}^{3}(\mathbb{T})}	},\;\;\;
\text{for any}\;\;t\geq 0,$$
where $C>0$ depends only on $R_{0}.$

Now, we move to prove theorem for $0<s<3$. 
Applying  a similar argument as  above to $u_{1}-u_{2}$
 and $a=u_{1}+u_{2}$, where 
 $u_{1}$ and $u_{2}$ are two different solutions, we obtain the following Lipschitz stability estimate which is useful in the interpolation argument
$${\normalsize
 \begin{split}
 \|(u_{1}-u_{2})(\cdot, t)\|_{L_{0}^{2}(\mathbb{T})}
 &\leq\displaystyle{
 	C\;e^{-k't}
 	\|(u_{1}-u_{2})(\cdot,0)\|_{L_{0}^{2}(\mathbb{T})}	},\;\;\;
 \text{for any}\;\;t\geq 0.
 \end{split}}$$

The case $0<s<3$ follows by an interpolation argument similar to the one applied in Theorem \ref{GWP}. One case use similar argument for other values of $s$.
\end{proof}

Note that Theorem \ref{GEF61} is a direct consequence of Theorem
\ref{GEF6}.

\section{Time-varying feedback law}\label{section 6}
In this section we construct a smooth time-varying feedback law such that a semiglobal stabilization holds with an arbitrary large decay rate.

Let $\lambda>0,$ $\mu \in \mathbb{R},$ $\alpha>0$   and $s\geq 0$ be given. Theorem
\ref{GEF6} implies that there exists $\kappa >0$ and a nondecreasing continuous function
$\alpha_{s}:\mathbb{R}^{+}\longrightarrow
\mathbb{R}^{+}$
such that for any $u_{0}\in H_{0}^{s}(\mathbb{T}),$
	the corresponding solution $u$ of the IVP
\begin{equation}\label{TVFL2}
\left \{
\begin{array}{l l}
\partial_{t}u-\partial^{3}_{x}u
-\alpha \mathcal{H}\partial^{2}_{x}u+2\mu\partial_{x} u+2u\partial_{x}u =-GG^{\ast}u,&  \; t>t_{0},\;\;x\in \mathbb{T},\\
u(x,t_{0})=u_{0}(x),  & \; x\in \mathbb{T},
\end{array}
\right.
\end{equation}	
  satisfies
	\begin{equation}\label{TVFL1}
	\begin{split}
	\|u(\cdot,t)
	\|_{H_{0}^{s}(\mathbb{T})}&\leq
	\alpha_{s}(\|u_{0}\|_{L_{0}^{2}(\mathbb{T})})\;
	e^{-\kappa (t-t_{0})}\|u_{0}
	\|_{H_{0}^{s}(\mathbb{T})},\;\;\;\text{for all}\;t\geq t_{0}.
	\end{split}
	\end{equation}

Also,  for any fixed $\lambda' \in (0,\lambda)$
and any $u_{0}\in H_{0}^{s}(\mathbb{T})$, the Theorem \ref{LEF1} asserts that
the solution of the IVP
\begin{equation}\label{TVFL3}
\left \{
\begin{array}{l l}
\partial_{t}u-\partial^{3}_{x}u
-\alpha \mathcal{H}\partial^{2}_{x}u+2\mu\partial_{x} u+2u\partial_{x}u =-K_{\lambda}u,&  \; t>t_{0},\;\;x\in \mathbb{T},\\
u(x,t_{0})=u_{0}(x),  & \; x\in \mathbb{T},
\end{array}
\right.
\end{equation}
satisfies
	 \begin{equation}\label{TVFL4}
	\begin{split}
	\|u(\cdot,t)
	\|_{H_{0}^{s}(\mathbb{T})}&\leq
	C_{s}\;e^{-\lambda'(t-t_{0})}\|u_{0}
	\|_{H_{0}^{s}(\mathbb{T})},\;\;\;\text{for all}\;t\geq t_{0},
	\end{split}
	\end{equation}
for some $C_{s}>0,$ provided that $\|u_{0}\|_{s}\leq r_{0}$ for some  $r_{0}\in (0,1).$
Define $\rho \in C^{\infty}(\mathbb{R}^{+}; [0,1])$ a function 
such that
\begin{equation}\label{TVFL5}
\rho(r)=1,\;\text{for}\;r\leq r_{0},\;\;\;\;\;\;\;\;
\rho(r)=0,\;\text{for} \; r\geq 1.
\end{equation}

Also, consider any function $\theta \in C^{\infty}(\mathbb{R};[0,1])$
with the following properties:
\begin{equation}\label{TVFL6}
\left\{
\begin{array}{lcl}
 \theta(t+2)=\theta(t) & \mbox{for all} &  t\in \mathbb{R},\\
\theta(t)=1& \mbox{for} &  \delta\leq 1\leq 1-\delta,\\
\theta(t)=0& \mbox{for} &1\leq t\leq 2,
\end{array}
\right.
\end{equation}
for some $\delta \in (0, \frac{1}{10}).$ 
Let $T>0$ be given. We define the following time-varying feedback law
\begin{equation}\label{TVFL7}
\begin{split}
K(u,t)&:=\rho(\|u\|^{2}_{H_{0}^{s}(\mathbb{T})})
\left[ \theta(\frac{t}{T})K_{\lambda}u+\theta(\frac{t}{T}-T)GG^{\ast}u\right]
+(1-\rho(\|u\|^{2}_{H_{0}^{s}(\mathbb{T})}))GG^{\ast}u\\
&=GG^{\ast}\left\{\rho(\|u\|^{2}_{H_{0}^{s}(\mathbb{T})})
\left[ \theta(\frac{t}{T})L^{-1}_{\lambda}u+\theta(\frac{t}{T}-T)u\right]
+(1-\rho(\|u\|^{2}_{H_{0}^{s}(\mathbb{T})}))u\right\}.
\end{split}
\end{equation}

Observe that $K$  has the following behaviour of the trajectories. In a first
time, when $\|u\|_{H^{s}_{0}(\mathbb{T})}$
 is large, we choose $K=GG^{\ast}$ to guarantee the decay of the solution.
Then, after a transient period, we have $\|u\|_{H_{0}^{s}(\mathbb{T})}\leq r_{0}$
 and we get into an oscillatory regime. During each period of length $2T,$ we have three steps:
\begin{itemize}
\item A period of time for which the damping $K_{\lambda}$ is active, leading to a decay like $e^{-\lambda'(t-t_{0})};$
\item A short transition time of order $\delta$ where a deviation from the origin may occur;
\item A period of time for which the damping $GG^{\ast}$ is active, leading to a decay like $e^{-\kappa(t-t_{0})}.$
\end{itemize}
The expected decay is a \textquotedblleft mean value" of the two decays above.
We consider the system
\begin{equation}\label{TVFL8}
\left \{
\begin{array}{l l}
\partial_{t}u-\partial^{3}_{x}u
-\alpha \mathcal{H}\partial^{2}_{x}u+2\mu\partial_{x} u+2u\partial_{x}u =-K(u,t),&  t>t_{0},\;\;x\in \mathbb{T}\\
u(x,t_{0})=u_{0}(x), & x\in \mathbb{T}.
\end{array}
\right.
\end{equation}

Finally, we establish the following semiglobal stabilization result with an arbitrary
decay rate.

\begin{thm}\label{TVFL24}
Let	$s\geq 0,$ $\lambda>0,$  $\alpha>0,$  and
$\mu \in \mathbb{R}$ be given. Consider any $\lambda' \in (0, \lambda)$ and any
$\lambda''\in \left(\frac{\lambda'}{2},\frac{\lambda'+\kappa}{2}\right)$
where $\kappa$ is given in \eqref{TVFL1}. Then there exists a time
$T_{0}>0$ such that for any $T>T_{0},$ $t_{0}\in \mathbb{R}$ and
$u_{0}\in H_{0}^{s}(\mathbb{T}),$ the unique solution of the closed-loop
system \eqref{TVFL8} satisfies
$$\|u(\cdot,t)
	\|_{H_{0}^{s}(\mathbb{T})}\leq
	\gamma_{s}(\|u_{0}\|_{H_{0}^{s}(\mathbb{T})})\;
	e^{-\lambda'' (t-t_{0})}\|u_{0}
	\|_{H_{0}^{s}(\mathbb{T})},\;\;\;\text{for all}\;\;t\geq t_{0},
	$$
where $\gamma_{s}$ is a nondecreasing continuous function.
\end{thm}
\begin{proof}
The proof follows as  in \cite[Theorem 5.1]{14}.
\end{proof}

\begin{appendix} 
\section{}
\begin{lem}\label{compact8}
	A function $\phi\in C^{\infty}(\mathbb{T})$
	can be written in the form $\partial_{x}\varphi$ for some
	function $\varphi\in C^{\infty}(\mathbb{T}),$
	if and only if,
$$\int_{\mathbb{T}}\phi(x)\;dx
	=0.$$
\end{lem}

\begin{lem}\label{prop3}
	Let $s,r \in \mathbb{R},$
	and $f$ denotes the operator of multiplication by
	$f\in C^{\infty}(\mathbb{T}).$ Then, $[D^{r},f]:=D^{r}\;f-f\;D^{r}$ maps any $H^{s}(\mathbb{T})$
	into $H^{s-r+1}(\mathbb{T}),$ i.e., there exists a constant $c=c_{f}$
	depending on $f$ such that
$$\left\|[D^{r},f]\phi\right\|_{H^{s-r+1}(\mathbb{T})}
		\leq c_{f}\;\left\|\phi\right\|_{H^{s}(\mathbb{T})}.$$
\end{lem}

\begin{proof}
	This result is proved in \cite{Laurent}
	(see Lemma A.1).
\end{proof}

\begin{lem}\label{prop4}
	If $f \in C^{\infty}(\mathbb{T}),$
	then for every
	$s\in \mathbb{R}$ there exists positive constants $C$ and $C_{s}$ such that the following estimate holds
$$	\left\|f\;v\right\|_{H^{s}(\mathbb{T})}
		\leq C\;\left\|v\right\|_{H^{s}(\mathbb{T})}
		+C_{s}\;\left\|v\right\|_{H^{s-1}(\mathbb{T})}.$$
\end{lem}

\begin{proof}
	This result follows by writing $$D^{s}(fv)=
	fD^{s}v+[D^{s},f]v,$$ and using Lemma \ref{prop4}  (see
	Corollary A.2 in  \cite{Laurent}).
\end{proof}

\begin{lem}\label{prop6}
	Let $f \in C^{\infty}(\mathbb{T})$ and
	$\rho_{\epsilon}=
	e^{\epsilon^{2}\partial_{x}^{2}}$ with
	$0\leq \epsilon\leq 1.$
	Then $[\rho_{\epsilon},f]$ is uniformly
	bounded as an operator
	from $H^{s}$ into $H^{s+1}$
	and
$$\left\|[\rho_{\epsilon},f]\phi\right
		\|_{H^{s+1}(\mathbb{T})}
		\leq c_{s}\;\left\|\phi\right\|_{H^{s}(\mathbb{T})},
		\;\;\text{for all}\;\phi \in H^{s}(\mathbb{T}).$$
\end{lem}

\begin{proof}
	This result is proved in \cite{Laurent}
	(see Lemma A.3).
\end{proof}

\begin{ex}\label{examplemulti}
For $j\geq 1,$ consider the function
$v_{j}(x,t):=\psi(t)\;e^{ijx}e^{i\phi(j)t},$
where $\psi \in C^{\infty}_{c}(\mathbb{R})$
takes the value 1 on $[-1,1].$
Note that,
$$\widehat{v_{j}}(k,t)
=\psi(t)\;e^{i\phi(j)t}\;
\widehat{e^{ijx}}(k)
=\psi(t)\;e^{i\phi(j)t}\;
\delta_{kj},$$
where $\delta_{kj}$ is the Kronecker delta function.
Then,
$$\widehat{v_{j}}(k,\tau)
=\delta_{kj}\;
\left(\psi(t)\;e^{i\phi(j)t}\right)^{\wedge}(\tau)
=\delta_{kj}\;
\widehat{\psi}(\tau-\phi(j)).$$
Therefore,
$$\|v_{j}\|^{2}_{X_{0,b}}
=\displaystyle{
	\int_{\mathbb{R}}
	\langle \tau -\phi(j) \rangle^{2b}
	|\widehat{\psi}(\tau-\phi(j))|^{2}\;d\tau}
\leq c_{b}\|\psi\|^{2}_{H^{b}_{t}(\mathbb{R})}.$$

Thus, the sequence
$\{v_{j}\}$ is uniformly bounded in the space $X_{0,b},$
for every $b\geq 0.$

However, multiplying $v_{j}$ by $\varphi(x)=e^{ix},$ we observe
$$\|e^{ix} v_{j}\|^{2}_{X_{0,b}}
=\displaystyle{
	\int_{\mathbb{R}}
	\langle \tau -\phi(1+j) \rangle^{2b}
	|	\widehat{\psi}(\tau-\phi(j))|^{2}\;d\tau}.$$

Using that $\tau-\phi(1+j)= \tau-\phi(j)+P(j)$ with $P(j)=3j^{2}+(3-2\alpha)j+1+2\mu-\alpha,$ we have
$$\|e^{ix} v_{j}\|^{2}_{X_{0,b}}
\sim \int_{\mathbb{R}}
	\left(1+ |\tau +P(j)| \right)^{2b}
	|	\widehat{\psi}(\tau)|^{2}\;d\tau
\approx j^{4b},$$
for $j$ large enough.
\end{ex}

\end{appendix}






\subsection*{Acknowledgements}
Francisco Vielma is supported by Grant $\sharp$ 2015/06131-5, Sao Paulo Research Foundation (FAPESP) Brazil. The authors would like to thank Prof. Felipe Linares, Prof. Lionel Rosier and Prof. Ademir Pastor for many helpful discussions and suggestions to complete this work.


\end{document}